\documentclass[11pt]{article}

\usepackage{silence}
\usepackage[left=1in,top=1in,right=1in,bottom=1in,head=.1in]{geometry}

\usepackage{mystyle}
\usepackage[most]{tcolorbox}

\renewcommand{\bm}[1]{#1}

\renewcommand{\mathbf}[1]{#1}
\renewcommand{\boxed}{}

\newcommand{\ha}{\hat{\alpha}}
\newcommand{\hb}{\hat{\beta}}
\newcommand{\hl}{\bar{\lambda}}
\newcommand{\m}{\mathsf{pmsd}}

\usepackage{tocbibind}
\usepackage{microtype}
\usepackage{titletoc}
\usepackage{placeins}

\newcommand{\titletext}{Prediction-Only Distillation in Linear and Logistic Regression}

\title{\titletext}

\author{
    Hien Dang\footremember{utsds}{Department of Statistics and Data Sciences, University of Texas, Austin, TX 78712, USA.} \\ {\footnotesize \texttt{\url{hiendang@utexas.edu}}}
    \and
    Pratik Patil\footrecall{utsds} \\ {\footnotesize \texttt{\url{pratikpatil@utexas.edu}}} 
    \and
    Alessandro Rinaldo\footrecall{utsds} \\ {\footnotesize \texttt{\url{alessandrorinaldo@utexas.edu}}} 
}

\date{\vspace{-15pt}}

\begin{document}
\maketitle

\begin{abstract}
   Self-distillation (SD) is typically studied when the student is retrained on the teacher’s original training inputs.
   In many practical deployments, however, the labeled training data are no longer available, and one has access only to the trained predictor and fresh unlabeled covariates.
   We study SD in this \emph{prediction-only} regime through a fresh-$X$ prediction-mixed scheme: a pure-distilled student is trained on fresh covariates pseudo-labeled by the teacher, and the final predictor is an affine combination of the teacher and student predictions.
   For ridge regression under proportional asymptotics, we derive deterministic equivalents for the optimally mixed prediction risk under general anisotropic covariance and deterministic signal.
   We show that this risk is \emph{strictly smaller} than the teacher risk for almost every pair of teacher and student regularization levels, including when the fresh covariates are out-of-distribution and even when their covariance is isotropic.
   We further show that the optimal mixing weight cannot be identified from unlabeled data alone, but can be consistently estimated in a single post-training step using a small independent labeled calibration set, without additional model fitting.
   Finally, for binary logistic regression, we show that prediction mixing can outperform both the teacher and the pure-distilled classifier.
\end{abstract}

\section{Introduction}
\label{sec:intro}

Knowledge distillation (KD) has become a fundamental technique in modern machine learning \citep{bucilua2006model, ba2014deep, hinton2015distilling}.
In its original formulation, KD transfers knowledge from a high-capacity teacher to a more compact student and allows the student to achieve comparable performance with fewer parameters.
This paradigm has since been extended to \emph{self-distillation}, in which a model is retrained using its own predictions while retaining the same architecture \citep{furlanello2018born, zhang2021self}.
Extensive empirical evidence suggests that self-distillation can improve generalization \citep{ahn2019variational, chen2017learning, chen2022knowledge, gou21knowledge}.

Self-distillation (SD) is typically studied when the student is retrained on the teacher's original training data.
In many practical deployments, however, those labeled data are no longer available: one can query a trained predictor but cannot access its training data or parameters.
We instead assume access to a stream of fresh unlabeled covariates, which may be out-of-distribution (OOD) relative to the teacher's training covariates.
We refer to this regime as the \emph{prediction-only} setting.
It arises naturally in post-deployment adaptation and model reuse, particularly when the original training data cannot be retained or shared because of privacy, safety, or storage constraints
\citep{nayak2019zeroshot, liang2022dine}.
This setting raises a basic question:
\begin{tcolorbox}[
    colback=gray!10,
    colframe=gray!50,
]
\centering
Can fresh, potentially OOD, unlabeled data be used to improve a fixed predictor,\\even when its original training data are unavailable?
\end{tcolorbox}

\bigskip

A natural approach is to train a student on fresh unlabeled covariates using the teacher's predictions as pseudo-labels.
Because no ground-truth labels are used, this procedure yields a \emph{pure-distilled} (PD) student \citep{lopez2015unifying}.
Pure-distillation has been studied in several settings, including high-dimensional ridgeless regression \citep{ildiz2025high}, ridge regression \citep{moniri2025on}, random-feature ridge regression \citep{wu2026improved}, and Gaussian mixture models \citep{carmon2019unlabeled, takanami2025effect, saglietti2022solvable}.
Perhaps surprisingly, the PD student can outperform its teacher.
Such gains are not guaranteed, however, and the PD student can instead perform substantially worse; see, for example, \Cref{fig:teaser}.

\begin{figure}[t]
    \centering
\includegraphics[width=0.95\columnwidth]{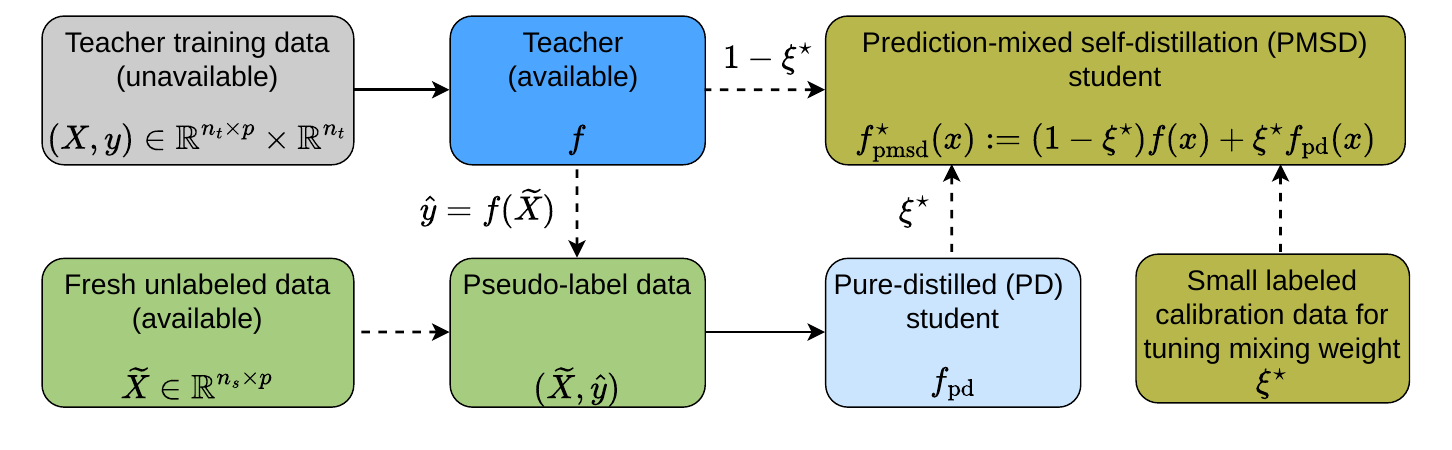}
    \caption{Schematic of the prediction-only setting and fresh-$X$ prediction-mixed SD. Solid arrows indicate training steps; dashed arrows indicate other information flows.}
    \label{fig:setup}
\end{figure}

Recent work on same-$X$ ridge SD establishes strict-improvement guarantees \citep{das2023understanding, pareek2024understanding, dang2026optimal, lecoiu2026self}.
In ridge regression, the optimally mixed student lies on the affine path between the teacher and the PD student; this observation suggests directly combining their predictions as the final output.
We call this affine combination predictor the \emph{prediction-mixed self-distillation} (PMSD) student.
Its mixing weight is real-valued and need not lie in $[0,1]$, so the combination need not be convex.
The method is illustrated in \Cref{fig:setup}.

This approach is operationally attractive: it requires no access to the original labeled data and reduces adaptation to a one-dimensional post hoc aggregation problem, with no additional model fitting.
Establishing an analogous strict-improvement guarantee in the fresh-$X$ setting is theoretically nontrivial.
The same-$X$ result relies on geometric identities induced by the shared design, and these identities disappear when the student is trained on a fresh design independent of the teacher's design.
Before turning to theory, \Cref{fig:teaser} provides an empirical preview of the fresh-$X$ setting on the UCI Blog Feedback dataset.
In both the in-distribution and out-of-distribution settings, the optimally mixed PMSD student outperforms the teacher and the PD student.

\begin{figure*}[t]
   \centering
    \begin{subfigure}[t]{0.48\textwidth}
    \centering
    \includegraphics[width=\textwidth]{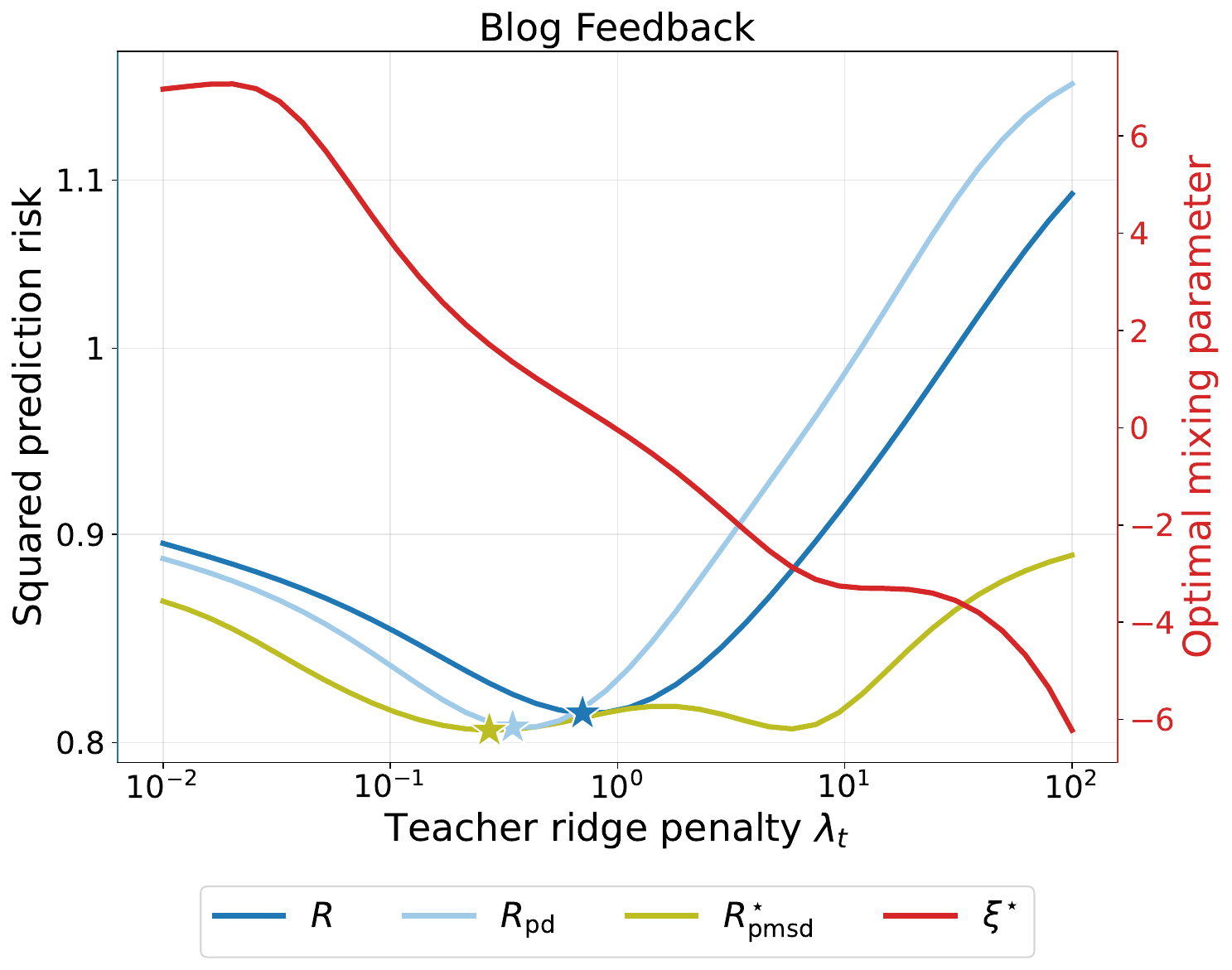}
    \caption{Fresh unlabeled covariates follow the teacher's training distribution.}
    \label{fig:teaser_a}
  \end{subfigure}
  \quad
    \begin{subfigure}[t]{0.48\textwidth}
    \centering
    \includegraphics[width=\textwidth]{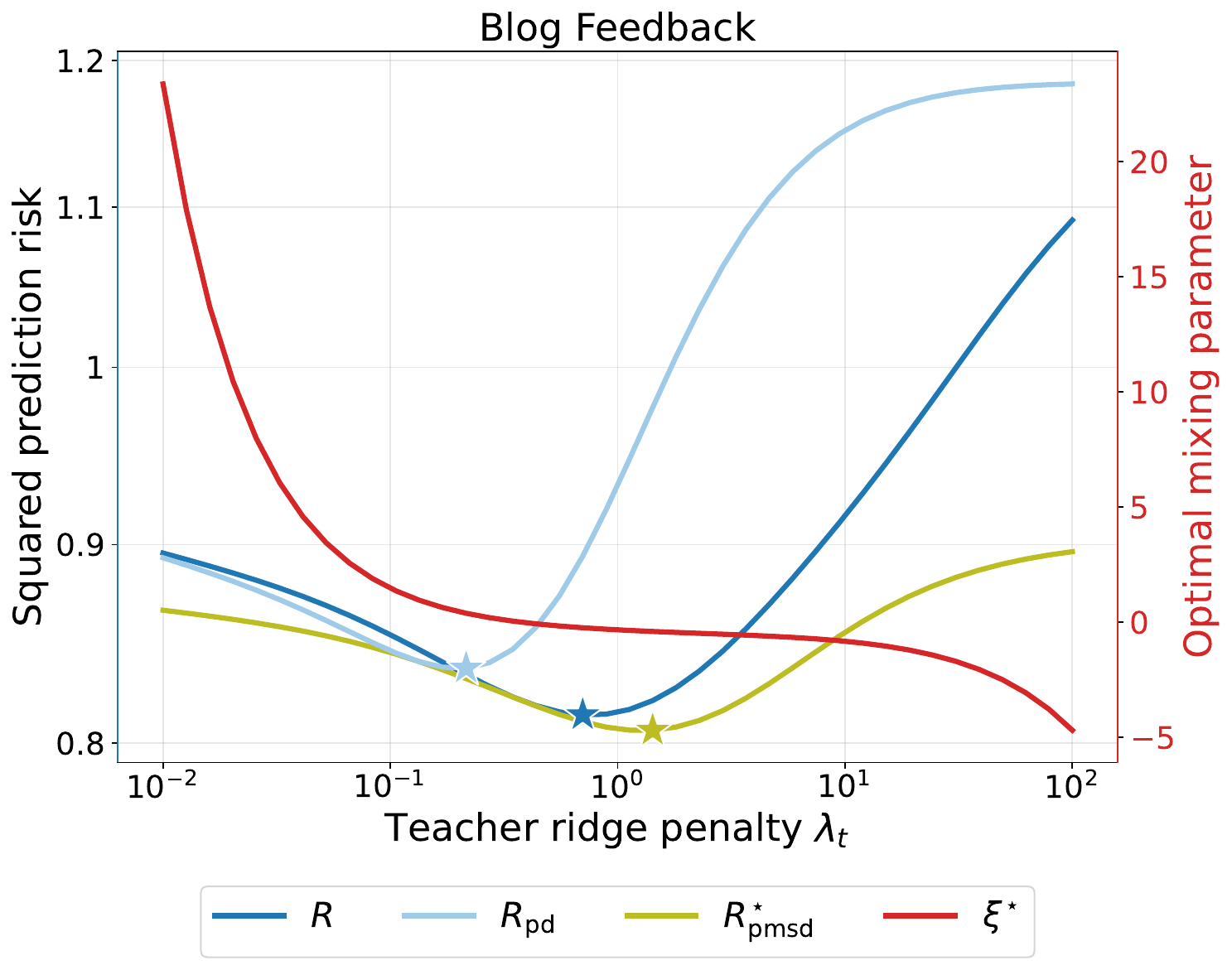}
    \caption{Fresh unlabeled covariates are sampled i.i.d.\ from the isotropic Gaussian distribution $\mathcal{N}(0, I_p)$.}
  \end{subfigure}
    \caption{\textbf{Strict improvement of the PMSD student.} 
    Test squared prediction risks of the ridge teacher ($R$, in {\color{pyblue} blue}), the fresh-$X$ pure-distilled student ($R_{\pd}$, in {\color{pylightblue} light blue}), and the optimally mixed PMSD student ($R_{\m}$, in {\color{pyolive} olive}), plotted as functions of the teacher's ridge penalty $\lambda_t$.
    The oracle optimal mixing weight $\xi^{\star}$ is shown in {\color{pyred} red}. Setting: UCI Blog Feedback data with $n_t = 2619$, $n_s = 5240$, and $p = 280$. We use the representative choice $\lambda_s = \lambda_t$. Experiments with alternative choices of $\lambda_s$, including values optimally tuned over a grid, are presented in \Cref{sec:additional_exp_regression}.}
    \label{fig:teaser}
\end{figure*}

This empirical preview motivates three theoretical and practical questions:
\begin{enumerate}[leftmargin=7mm]
    \item Under what conditions can the PMSD student improve upon the teacher and the PD student?
    \item How does its performance depend on the amount of fresh unlabeled data?
    \item Can the optimal mixing weight be tuned when the teacher's training data are unavailable?
\end{enumerate}

Our main contributions, which also provide an outline of the paper, are as follows:

\begin{itemize}[leftmargin=7mm]
\item \textbf{High-dimensional characterization for ridge regression (\Cref{sec:ridge}).}
We first derive exact oracle identities for the optimal PMSD mixing weight (\Cref{prop:oracle_affine}).
Under proportional asymptotics, we then characterize the PMSD risk and optimal mixing weight under general anisotropic covariance and deterministic signal (\Cref{thm:general_ridge}).

\item \textbf{Strict improvement and non-monotonicity (\Cref{sec:strict_nonmonotonicity}).}
Using these deterministic equivalents, we show that the PMSD student strictly improves upon the teacher for almost every pair of regularization levels $(\lambda_t, \lambda_s)$ used by the teacher and the PD student (\Cref{prop:teacher_bad_set,prop:strict_improv_ood}).
Thus, strict improvement is generic even when $\lambda_t$ is unknown or $\lambda_s$ is not optimally tuned, provided that the mixing weight is chosen optimally.

In the same-$\lambda$ isotropic setting, we also highlight a phenomenon absent from same-$X$ self-distillation: at fixed regularization, the PMSD risk is generically \emph{non-monotone} in the unlabeled aspect ratio, so more unlabeled data need not improve performance (\Cref{sec:ridge_mono_main}).

\item \textbf{Consistent data-dependent tuning (\Cref{sec:tuning_ridge}).}
We show that the teacher's outputs and fresh unlabeled covariates alone do not identify the oracle mixing weight.
We then prove that a small independent labeled calibration sample enables consistent estimation in a single post-training step, without additional model fitting (\Cref{thm:calibration_consistency}).

\item \textbf{Gains in logistic regression (\Cref{sec:logistic}).}
Finally, we extend the analysis beyond regression under squared loss by studying binary logistic regression.
We show that prediction mixing can strictly outperform both the teacher and the PD student (\Cref{thm:class_mixing,thm:class_mixing_2,thm:class_mixing_large_p}).
\end{itemize}

\section{Related work}
\label{sec:related_works}

\textbf{Same-$X$ self-distillation.}
A recent line of work studies self-distillation (SD) in regression settings where the original labeled data are available and the student is retrained on the same design.
\citet{mobahi2020self} interpret repeated SD as a regularization mechanism in Hilbert space.
\citet{das2023understanding} and \citet{pareek2024understanding} analyze one-step and repeated SD for linear models under fixed-design assumptions.
Under random design, \citet{dang2026optimal} establish pointwise strict improvement guarantees, proportional asymptotic results, and a tuning method for ridge regression.
\citet{lecoiu2026self} further show that, for spiked covariance matrices with $s$ spikes, $s$ rounds of self-distillation attain optimal performance among spectral shrinkage estimators.
Their results motivate using a different regularization level in each round of self-distillation.
Our paper studies a different regime: the \emph{prediction-only fresh-$X$} setting, where the original labeled sample is unavailable and the same-design geometric identities no longer apply.
We instead ask whether prediction mixing can improve the teacher when its training data and regularization level $\lambda_t$ are unknown and the student penalty $\lambda_s > 0$ is arbitrary.
Such pointwise guarantees are important because finding the optimal ridge penalty can be challenging.
For example, \citet{stephenson2021can} show that the leave-one-out cross-validation loss is generally neither convex nor quasi-convex as a function of the ridge penalty $\lambda$.

\textbf{Fresh-$X$ pseudo-labeling and pure-distillation.}
Empirical work shows that training on unlabeled data using pseudo-labels generated by the teacher can improve performance \citep{xie2020noisystudent, pham2021meta}.
On the theoretical side, pure-distillation on fresh unlabeled covariates has been studied in high-dimensional ridgeless regression \citep{ildiz2025high} and ridge regression \citep{moniri2025on}.
\citet{wu2026improved} analyze scaling laws for the teacher and the PD student in random feature ridge regression.
These works often consider a student that is stronger than the teacher, for example because it is more overparameterized or has a better regularization mechanism; this regime is known as ``weak-to-strong'' distillation \citep{burns2023weak}.
Although the PD student can outperform the teacher, such gains are not guaranteed.
For example, in high-dimensional ridge regression, \citet{moniri2025on} show that when the teacher is over-regularized, the PD student cannot improve upon it and may perform worse; see \Cref{fig:teaser} for this phenomenon.
These works focus primarily on the benefits of pure-distillation in teacher–student transfer, whereas we ask whether combining predictions from the teacher and the PD student can guarantee an improvement over the teacher.
We also study how the PMSD risk depends on the fresh unlabeled sample size and how to calibrate the optimal mixing weight post hoc.

\textbf{Theory of self-distillation in classification.}
A substantial literature studies the benefits of retraining on the teacher's pseudo-labels in classification.
This includes linear probing on frozen feature backbones \citep{das2023understanding, jeong2025rethinking} and logistic regression on Gaussian mixture data \citep{carmon2019unlabeled, javanmard2025self, das2024retraining, takanami2025effect, saglietti2022solvable}.
In the Gaussian mixture setting, \citet{das2024retraining} show that retraining with the teacher's hard labels can improve accuracy under certain conditions on the label noise rate and sample size.
In a different direction, \citet{carmon2019unlabeled} show that training a student with the teacher's hard labels on unlabeled data can improve robustness to adversarial examples.
For binary and multiclass classification with frozen feature backbones (i.e., linear probing), \citet{das2023understanding, jeong2025rethinking} show that the pure-distilled student, trained only on the teacher's soft labels, can tolerate higher label noise rates and improve accuracy over the teacher.
Our logistic regression analysis and linear probing experiments (see \Cref{sec:logistic}) are complementary: we study a prediction mixing mechanism that can improve upon both the teacher and the pure-distilled student.

\section{Ridge regression}
\label{sec:ridge}

In this section, we introduce prediction-mixed distillation for ridge regression, derive its structural risk decomposition, and obtain deterministic equivalents under proportional asymptotics.

\subsection{Setup and preliminaries}

Let $f_{\lambda_t}: \RR^p \to\R$ be the teacher ridge predictor trained on an \emph{unknown} labeled dataset $\cD_{\rm lab} := (X, y) \in \R^{n_t \times p} \times \R^{n_t}$.
Here, $X \in \mathbb{R}^{n_t \times p}$ is the design matrix with i.i.d. $p$-dimensional observations $x_i \in \mathbb{R}^p$ as its rows, and $y \in \mathbb{R}^{n_t}$ is the vector of corresponding scalar ground-truth labels $y_i$, for $i = 1, \ldots, n_t$.
Denote the regularization parameter used by the teacher by $\lambda_t > 0$ (potentially depending on $\cD_{\rm lab}$).
That is, the teacher ridge predictor is the random function 
\begin{equation}
\label{eq:teacher}
x \in \RR^p \mapsto f_{\lambda_t}(x) := x^\top \argmin_{\beta \in \RR^p}\big\{ \| y - X \beta \|_2^2 / n_t + \lambda_t \| \beta \|_2^2\big\}.
\end{equation}
In the prediction-only regime considered here, we can query the teacher at any input but need not know the explicit form of $f_{\lambda_t}$ or its regularization parameter $\lambda_t$.
Thus, we observe only the teacher's predictions at the queried points.

Given a fresh sample of unlabeled covariates $\cD_{\rm unlab} := \tilde X \in \RR^{n_s \times p}$ consisting of $n_s$ i.i.d.\ observations that may be \emph{out-of-distribution} (OOD) relative to the teacher's training distribution, we query the teacher on this unlabeled set and denote its predictions by $\hat y_{\lambda_t} := f_{\lambda_t}(\tilde X) \in \RR^{n_s}$.
The pure-distilled (PD) student $f_{\pd, \lambda_s}$ is the ridge predictor trained on $(\tilde X, \hat y_{\lambda_t})$ with penalty $\lambda_s > 0$\footnote{For brevity, we omit the dependence on $\lambda_t$ in the notation for the PD student when the context is clear.},
\begin{equation}
\label{eq:pd_student}
x \in \RR^p \mapsto f_{\pd, \lambda_s}(x) := x^\top \argmin_{\beta \in \RR^p}\big\{ \| \hat y_{\lambda_t} - \tilde X \beta \|_2^2 / n_s + \lambda_s \| \beta \|_2^2\big\}.
\end{equation}

A natural approach to improving the student's predictive performance is to combine the teacher's queried predictions with the student's own PD predictions.
We focus on \emph{affine mixing strategies}, in which the student deploys an affine combination of the available predictions.
Specifically, for a mixing parameter $\xi \in \RR$, the prediction-mixed self-distillation (PMSD) student predictor is given by
\begin{equation}
\label{eq:affine_path}
 x \in \RR^p \mapsto    f_{\m, \lambda_t, \lambda_s, \xi}(x):=(1-\xi)f^{\te}_{\lambda_t}(x)+\xi f_{\pd, \lambda_s}(x).
\end{equation}
We consider affine rather than convex combinations because negative weights or weights exceeding one may be preferable.
Indeed, empirical and theoretical work shows that the optimal mixing weight lies outside the unit interval in many settings \citep{das2023understanding, pareek2024understanding, dang2026optimal}.

Let $(x_0,y_0)$ denote an independent test point drawn from the teacher's training distribution, corresponding to the out-of-sample in-distribution setting.
For any predictor $f$, define the conditional squared prediction risk
\begin{equation}
\label{eq:freshX_pred_risk}
  R(f):=\E\big[(y_0-f(x_0))^2\mid \cD_{\rm lab},\cD_{\rm unlab}\big].
\end{equation}
For convenience, denote $R_{\te}(\lambda_t):=R(f_{\lambda_t}^{\te})$, $R_{\pd}(\lambda_s):=R(f_{\pd, \lambda_s})$, and $R_{\m}(\lambda_t, \lambda_s, \xi):=R(f_{\m, \lambda_t, \lambda_s, \xi})$.
When the regularization parameters are clear from context, we use the shorthand $R_{\te}$, $R_{\pd}$, and $R_{\m}$.

\subsection{Optimal prediction-mixed risk decomposition}

We first decompose the optimally mixed PMSD risk in terms of three key quantities: the teacher risk $R_{\te}(\lambda_t)$, the PD student risk $R_{\pd}(\lambda_s)$, and the teacher-student residual correlation
\[
C(\lambda_t, \lambda_s):=\E\!\bigl[(y_0-f^{\te}_{\lambda_t}(x_0))(y_0-f_{\pd, \lambda_s}(x_0))\mid \cD_{\rm lab},\cD_{\rm unlab}\bigr].
\]
Let $\xi^\star(\lambda_t, \lambda_s) \in \argmin_{\xi \in \RR} R_{\m}(\lambda_t, \lambda_s,\xi)$ denote the optimal mixing weight, and let $R_{\m}^\star(\lambda_t, \lambda_s) := R_{\m}(\lambda_t, \lambda_s, \xi^{\star})$ denote the corresponding optimal PMSD risk.

\begin{proposition}[Optimal prediction-mixed decomposition]
\label{prop:oracle_affine}
Suppose that $\lambda_t$, $\lambda_s$, and $D(\lambda_t, \lambda_s) := R_{\te}(\lambda_t) + R_{\pd}(\lambda_s) - 2 C(\lambda_t, \lambda_s)$ are strictly positive.
Then,
\begin{align}
\label{eq:oracle_risk}
\xi^{\star}(\lambda_t, \lambda_s)&=\frac{R_{\te}(\lambda_t) - C(\lambda_t, \lambda_s)}{D(\lambda_t, \lambda_s)}, \\
R_{\m}^{\star}(\lambda_t, \lambda_s) &= R_{\te}(\lambda_t) -\frac{(R_{\te}(\lambda_t) - C(\lambda_t, \lambda_s))^2}{D(\lambda_t, \lambda_s)} = 
R_{\pd}(\lambda_s) -\frac{(R_{\pd}(\lambda_s) - C(\lambda_t, \lambda_s))^2}{D(\lambda_t, \lambda_s)}. \nonumber
\end{align}
\end{proposition}

The decomposition implies that $R^{\star}_{\m} \leq \min\{R_{\te}, R_{\pd}\}$.
The gain over the teacher is determined by the magnitude of $R_{\te}(\lambda_t)-C(\lambda_t, \lambda_s)$, while the sign of this difference determines the sign of $\xi^\star$, which may therefore be negative.
The denominator is the mean squared difference between the teacher and the PD student:
\begin{equation}
    D(\lambda_t, \lambda_s) =
    R_{\te}(\lambda_t) + R_{\pd}(\lambda_s) - 2 C(\lambda_t, \lambda_s) = \EE\big[(f_{\lambda_t}(x_0)-f_{\pd,\lambda_s}(x_0))^2\mid \cD_{\rm lab},\cD_{\rm unlab} \big] \geq 0.
    \nonumber
\end{equation}
It is zero if and only if $f_{\lambda_t}$ and $f_{\pd,\lambda_s}$ coincide almost surely on the test distribution, in which case $R_{\m} \equiv R_{\te} \equiv R_{\pd}$ for all $\xi$, so any $\xi$ is trivially optimal.
Hence the assumption $D > 0$ is mild.

\subsection{Proportional asymptotics under general anisotropic covariance}

In this subsection, we derive deterministic equivalents for the optimal mixing weight and the conditional PMSD risk in the proportional asymptotic regime, where $n_t, n_s, p \to \infty$ at the same rate.

Throughout, we assume that the teacher's training data $\{(x_i,y_i)\}_{i=1}^{n_t}$ are i.i.d.\ samples from a distribution satisfying the standard assumptions used in the random matrix analysis of high-dimensional regression \citep{hastie2022surprises, bartlett2021deep}.
We further assume that the fresh unlabeled covariates $\{\tilde{x}_i\}_{i=1}^{n_s}$ are i.i.d.\ samples whose population covariance matrix is \emph{simultaneously diagonalizable} with the population covariance matrix of the teacher's covariates $\{x_i\}_{i=1}^{n_t}$; equivalently, the two covariance matrices commute.
This assumption of simultaneous diagonalizability across covariate sources has also been considered in the transfer learning literature \citep{mallinar2024minimum, song2024generalization}.

\begin{assumption}[Data distribution] 
\label{def:dist}
We assume the following about the data distribution.
    \begin{enumerate}[leftmargin=7mm]
        \item [(a)] The training covariates $X \sim P_x^{\otimes n_t}$ admit the representation $X=Z\Sigma_t^{1/2}$, where $\Sigma_t\in\RR^{p\times p}$ is deterministic and positive definite, with eigenvalues bounded away from $0$ and $\infty$, and $Z\in\RR^{n_t\times p}$ has i.i.d.\ entries with mean $0$, variance $1$, and uniformly bounded $(4+\mu)$-th moment for some $\mu>0$.
        \item [(b)] The fresh unlabeled covariates $\tilde X\sim \tilde P_x^{\otimes n_s}$ admit the representation $\tilde X=\tilde Z\Sigma_s^{1/2}$, where $\Sigma_s\in\RR^{p\times p}$ is deterministic and positive definite, with eigenvalues bounded away from $0$ and $\infty$.
        Moreover, $\tilde Z\in\RR^{n_s\times p}$ has i.i.d.\ entries with mean $0$, variance $1$, and uniformly bounded $(4+\tilde\mu)$-th moment for some $\tilde\mu>0$.
        \item [(c)] $\Sigma_t$ and $\Sigma_s$ are simultaneously diagonalizable, i.e., there exists an orthogonal matrix $U \in \RR^{p \times p}$ such that $\Sigma_t = U^{\top} \operatorname{diag}(\sigma_1,\ldots,\sigma_p) U$ and $\Sigma_s = U^{\top} \operatorname{diag}(\tilde\sigma_1,\ldots,\tilde\sigma_p) U$.
        \item [(d)] The training labels $y\sim P_y$ have mean $0$ and uniformly bounded $(4+\nu)$-th moment for some $\nu>0$.
    \end{enumerate}
\end{assumption}

\begin{figure*}[!t]
  \centering
    \begin{subfigure}[t]{0.6\textwidth}
    \centering
    \includegraphics[width=\textwidth]{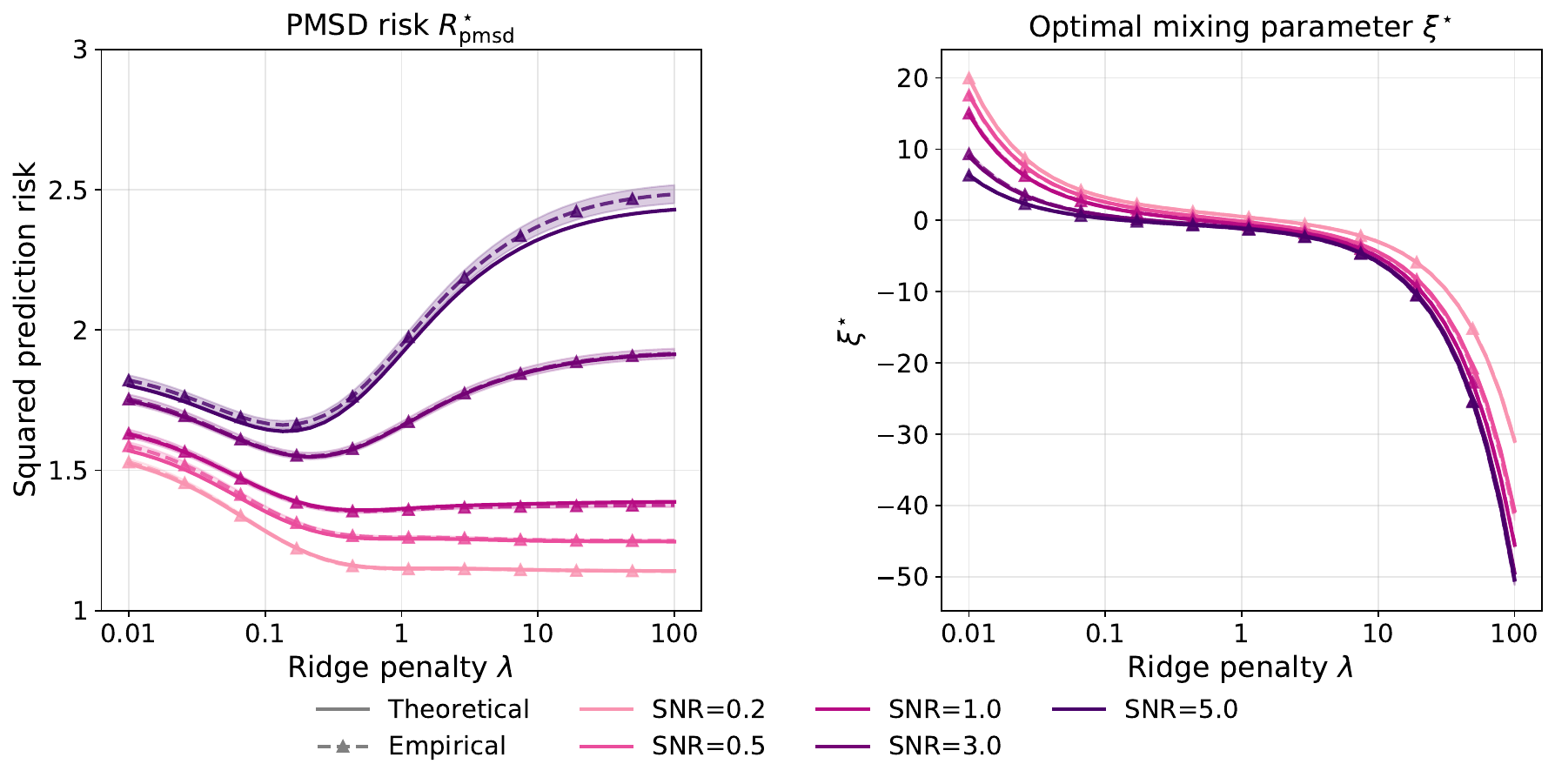}
    \caption{Theoretical versus empirical risks and mixing weights.}
    \label{fig:over_snr}
  \end{subfigure}
  \hfill
    \begin{subfigure}[t]{0.38\textwidth}
    \centering
    \includegraphics[width=\textwidth]{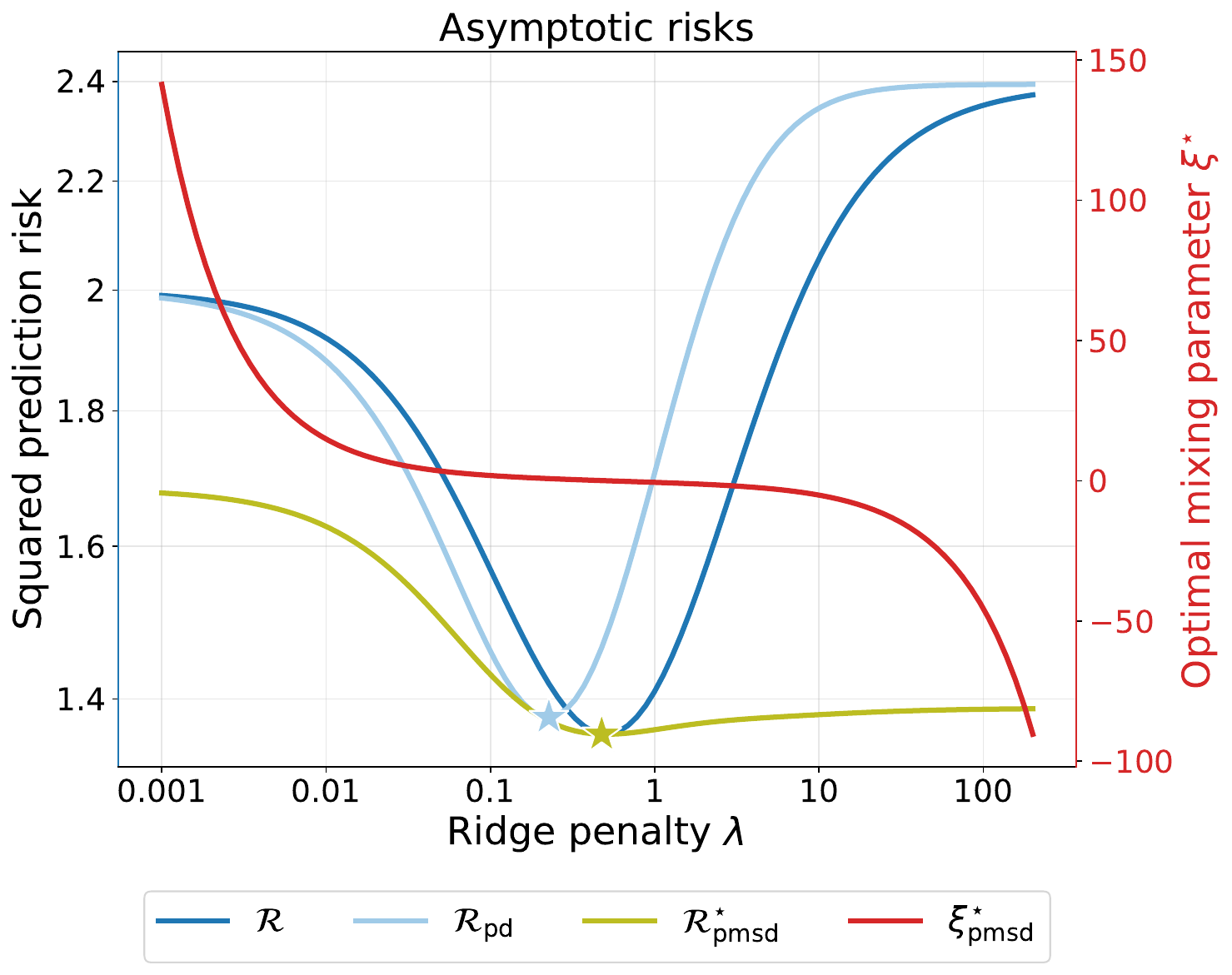}
    \caption{Same-$\lambda$ balanced regime with $\gamma_t = \gamma_s$.}
    \label{fig:asymp_curve}
  \end{subfigure}
  \caption{Setting: $\Sigma_t = \Sigma_s$ and $\lambda_t = \lambda_s = \lambda$. 
  Both panels use an AR1 covariance design, and the true signal $\beta$ aligns with the top eigenvectors of $\Sigma_t$. We set $p = 200$, $n_t = n_s = 400$, and $\sigma^2 = 1$ in both panels; the right panel has $r^2 = 1$. Similar figures with $\Sigma_t \neq \Sigma_s$ and $\lambda_t \neq \lambda_s$ appear in \Cref{sec:additional_exps_synthetic_ridge}.}
  \label{fig:synthetic_ridge}
\end{figure*}

\textbf{Key quantities and notation.} We define the population projection and variance parameters as
$\beta:=\Sigma_t^{-1}\E[xy]$ and $\sigma^2:=\Var(y-x^\top\beta)$, respectively.
The signal energy is $r^2 := \| \beta \|_2^2$ and $\SNR:=r^2/\sigma^2$.
Denote the labeled and unlabeled data aspect ratios $\gamma_t := p/n_t$ and $\gamma_s := p/n_s$, respectively.
We introduce several scalar quantities to characterize the limiting risk.
Let $\otr(A):=\tr(A)/p$ denote the normalized trace.
For any fixed $\lambda_t > 0$ and $\lambda_s > 0$, define $\kappa_t = \kappa_t(\lambda_t) > 0$ and $\kappa_s = \kappa_s(\lambda_s) > 0$ as the unique solutions to the corresponding fixed-point equations,
\begin{equation}
\label{eq:kappa_fp_mainpaper}
\kappa_t = \lambda_t + \gamma_t \kappa_t \,\otr\!\big(\Sigma_t (\Sigma_t +\kappa_t I_p)^{-1}\big) \quad \text{and} \quad
\kappa_s = \lambda_s + \gamma_s \kappa_s \,\otr\!\big(\Sigma_s (\Sigma_s +\kappa_s I_p)^{-1}\big),
\end{equation}
respectively, and let $G_t = (\Sigma_t + \kappa_t I_p)^{-1}$ and $G_s = (\Sigma_s + \kappa_s I_p)^{-1}$ be the associated resolvents.
For fixed $(\Sigma_t, \gamma_t)$ and $(\Sigma_s, \gamma_s)$, the maps $\lambda_t \mapsto \kappa_t$ and $\lambda_s \mapsto \kappa_s$ are both one-to-one on $(0,\infty)$.
Finally, for $i, j \geq 0$, define the mixed signal-covariance alignment and trace functionals as
\begin{equation}
    q_{i,j}:=\beta^\top G_t^i G_s^j \Sigma_t \,\beta,
    \qquad 
    \tilde q_{i,j}:=\beta^\top G_t^i G_s^j \Sigma_s \,\beta,
  \qquad
  t_{i,j}:=\gamma_t \,\otr(\Sigma_t^2 G_t^i G_s^j),
\end{equation}
and set
$
  b_t := (1 - t_{2,0})^{-1}
$
and
$
  b_s := \big(\gamma_s \otr(\Sigma_s \Sigma_t G_s^2) \big) / \big(1 - \gamma_s \otr(\Sigma_s^2 G_s^2) \big).
$

Our first result states that the three oracle ingredients in \Cref{prop:oracle_affine} have deterministic equivalents, in the sense defined in \Cref{app:DE_Q_U-preliminaries}, under general anisotropic covariance and deterministic signal.
The exact closed forms are more algebraically involved than those for same-$X$ ridge distillation in \cite{dang2026optimal} because they involve two coupled resolvents, but remain explicit in terms of finitely many traces of functions of $(G_t, G_s, \Sigma_t, \Sigma_s)$ and finitely many quadratic forms $\beta^\top f(\Sigma_t, \Sigma_s)\beta$.

\begin{theorem}[General anisotropic risk asymptotics]
\label{thm:general_ridge}
Under \Cref{def:dist}, for fixed $\lambda_t > 0$ and $\lambda_s > 0$, as $n_t, n_s, p \to \infty$ with $0 < \liminf \gamma_t \leq \limsup \gamma_t < \infty$ and $0 < \liminf \gamma_s \leq \limsup \gamma_s < \infty$, we have
\begin{align}
  R_{\te}(\lambda_t) \pto \mathcal R_{\te}(\lambda_t),
\qquad
R_{\pd}(\lambda_s) \pto \mathcal R_{\pd}(\lambda_s),
\qquad
C(\lambda_t, \lambda_s)\pto \mathcal C(\lambda_t, \lambda_s),   
\end{align}
where
\begin{align}
    \sR(\lambda_t) &= \sigma^2 + \kappa_t^2 b_t q_{2,0} + \sigma^2 b_t t_{2,0}, \nonumber \\
    \sC(\lambda_t, \lambda_s) &= \sigma^2 + \kappa_t^2 b_t q_{2,0} + \sigma^2 b_t t_{2,0} - \kappa_s (
        -\kappa_t q_{1,1}
        +\kappa_t^2 q_{2,1}
        +\kappa_t^2 b_t t_{2,1} q_{2,0}
        +\sigma^2 b_t t_{2,1}
    ) \nonumber,
    \\
\sR_{\pd}(\lambda_s) &=   \kappa_s^2
 \Big(
 q_{0,2} + b_s \tilde q_{0,2} - 2 \kappa_t (q_{1,2} + b_s \tilde q_{1,2}) + \kappa_t^2 
 (q_{2,2} + b_s \tilde q_{2,2})
 \nonumber \\
 &\hspace{3em}
 + b_t (\kappa_t^2 q_{2,0} + \sigma^2) \big(t_{2,2} + b_s \gamma_t \otr(\Sigma_t \Sigma_s G_t^2 G_s^2) \big)
 \Big) -
 \sR(\lambda_t) + 2\sC(\lambda_t, \lambda_s). \nonumber
\end{align}
In the nontrivial case $r^2+\sigma^2>0$, the limiting denominator $\mathcal R_{\te}(\lambda_t)+\mathcal R_{\pd}(\lambda_s)-2\mathcal C(\lambda_t,\lambda_s)$ is strictly larger than 0, and
\begin{align}
    \label{eq:xi_star_asymp}
   &\xi^{\star}(\lambda_t, \lambda_s)\pto
\frac{\mathcal R_{\te}(\lambda_t)-\mathcal C(\lambda_t, \lambda_s)}{\mathcal R_{\te}(\lambda_t)+\mathcal R_{\pd}(\lambda_s) - 2 \mathcal C(\lambda_t, \lambda_s)}, 
\\
\label{eq:R_star_asymp}
&R_{\m}^{\star}(\lambda_t, \lambda_s) \pto \sR_{\m}^{\star}(\lambda_t, \lambda_s) :=
\mathcal R_{\te}(\lambda_t)
-
\frac{\bigl(\mathcal R_{\te}(\lambda_t)-\mathcal C(\lambda_t, \lambda_s)\bigr)^2}{\mathcal R_{\te}(\lambda_t)+\mathcal R_{\pd}(\lambda_s) - 2 \mathcal C(\lambda_t, \lambda_s)}. 
\end{align}
\end{theorem}

As shown in \Cref{fig:over_snr}, the theoretical predictions closely align with the empirical risks even at moderate values of $n_t$, $n_s$, and $p$. 

\section{Strict improvement and non-monotonicity}
\label{sec:strict_nonmonotonicity}

Using the risk characterization from \Cref{sec:ridge}, we now study two qualitative questions: when the PMSD student strictly improves upon the teacher and the pure-distilled student, and how its optimal risk depends on the amount of fresh unlabeled data.

\subsection{Strict improvement over the teacher}
\label{sec:strict_improv}

Leveraging the deterministic equivalents from \Cref{thm:general_ridge}, we study when the PMSD student strictly improves upon the teacher.
Equation \eqref{eq:R_star_asymp} shows that the PMSD risk is asymptotically strictly less than the teacher risk whenever
\begin{equation}
   \mathcal R_{\te}(\lambda_t)-\mathcal C(\lambda_t, \lambda_s) \neq 0, 
\end{equation}
and the two risks coincide otherwise.
Thus, characterizing when the PMSD student fails to improve upon the teacher amounts to understanding when $\mathcal R_{\te}(\lambda_t)-\mathcal C(\lambda_t, \lambda_s)$ vanishes.

Since $\Sigma_t$ and $\Sigma_s$ are simultaneously diagonalizable, let $\Sigma_t = U^{\top} \operatorname{diag}(\sigma_1, \ldots, \sigma_p) U$ and $\Sigma_s = U^{\top} \operatorname{diag}(\tilde \sigma_1, \ldots, \tilde \sigma_p) U$ be their eigendecompositions.
Define the squared signal-covariance alignment by $v_i := ((U\beta)_i)^2$ for $i \in [p]$.
Let $m$ and $\tilde m$ ($1 \le m, \tilde m \le p$) denote the number of distinct eigenvalues of $\Sigma_t$ and $\Sigma_s$, respectively. Write $\{\tilde \sigma_{(1)}, \ldots, \tilde \sigma_{(\tilde m)}\}$ for the corresponding set of distinct eigenvalues of $\Sigma_s$, with $\tilde \sigma_{(1)} \leq \cdots \leq \tilde \sigma_{(\tilde m)}$.
The term $\mathcal R_{\te}(\lambda_t)-\mathcal C(\lambda_t, \lambda_s)$ can be rewritten as follows (see \Cref{sec:proof_of_strict_improvement}):
\begin{align}
\label{eq:R-C}
\mathcal R_{\te}(\lambda_t)-\mathcal C(\lambda_t,\lambda_s)
&= \sum_{i=1}^{p} \frac{c_i}{\tilde \sigma_i+\kappa_s} = \sum_{j=1}^{\tilde m} \frac{d_j}{\tilde \sigma_{(j)}+\kappa_s},
\end{align}
where
\begin{align}
c_i
&:= \frac{\kappa_t \sigma_i^2}{(\sigma_i+\kappa_t)^2}
\biggl(
\frac{\kappa_t^2 b_t q_{2,0}+\sigma^2 b_t}{\kappa_t n_t}
- v_i
\biggr),
\qquad i=1,\ldots,p, \label{eq:c_i_def} \\
d_j
&:= \sum_{k:\,\tilde \sigma_k=\tilde\sigma_{(j)}} c_k,
\qquad j=1,\ldots,\tilde m. \label{eq:d_j_def}
\end{align}
For a fixed and unknown $\lambda_t > 0$ used by the teacher, we are interested in the choices of $\lambda_s > 0$ for which the PMSD student strictly improves upon the teacher.
Hence, we study the cardinality of the set of $\lambda_s$ for which there is no strict improvement at a given $\lambda_t$, which we call the ``student-tie'' set:
\[
\Lambda_{s}(\lambda_t) := \{ \lambda_s > 0: \mathcal R_{\te}(\lambda_t)-\mathcal C(\lambda_t, \lambda_s) = 0\}.
\]

Fix $\lambda_t$ and suppose that $\mathcal R_{\te}(\lambda_t)-\mathcal C(\lambda_t, \lambda_s)$ is not identically zero as a function of $\lambda_s$. Setting the right-hand side of \eqref{eq:R-C} to zero and clearing the denominators gives
\[
\sum_{j=1}^{\tilde m} \, d_j \prod_{k \neq j} (\tilde \sigma_{(k)} + \kappa_s) = 0,
\]
which is a polynomial equation in $\kappa_s$ of degree at most $\tilde m-1$.
The one-to-one correspondence between $\lambda_s$ and $\kappa_s$ in equation \eqref{eq:kappa_fp_mainpaper} implies that there are at most $\tilde m-1$ values of $\lambda_s$ for which $\mathcal R_{\te}(\lambda_t)-\mathcal C(\lambda_t, \lambda_s) = 0$.
Therefore, $|\Lambda_{s}(\lambda_t)| \leq \tilde m - 1$ in this case.

On the other hand, if, for a given $\lambda_t$, the gap $\mathcal R_{\te}(\lambda_t)-\mathcal C(\lambda_t, \lambda_s)$ is identically zero as a function of $\lambda_s$, then the PMSD student cannot improve upon the teacher for any choice of $\lambda_s$.
Equivalently, $\Lambda_{s}(\lambda_t)=\RR^+$.
We call this the ``teacher-degeneracy'' set:
\[
\Lambda_{t} := \{ \lambda_t > 0: \mathcal R_{\te}(\lambda_t)-\mathcal C(\lambda_t, \lambda_s) = 0 \text{ for all } \lambda_s>0\}.
\]

\begin{proposition}[Upper bound on teacher-degeneracy set size]
\label{prop:teacher_bad_set}
$\lambda_t$ belongs to $\Lambda_t$ if and only if $d_j=0$ for every $j\in[\tilde m]$, where the $d_j$ are defined in \eqref{eq:d_j_def}. Moreover, in the nontrivial case $r^2+\sigma^2>0$, a crude upper bound on $|\Lambda_t|$ is $4m-1$, where $m$ is the number of distinct eigenvalues of $\Sigma_t$.
\end{proposition}

Generally, this degeneracy condition requires $\lambda_t$ to satisfy a specific system of $\tilde m$ equations determined jointly by the covariance structure of $\Sigma_t$, the signal-covariance alignment, and the noise variance $\sigma^2$. The bound on $|\Lambda_t|$ is crude, as it uses only the weaker necessary condition
\[
\sum_{j=1}^{\tilde m} d_j = 0,  
\] 
which holds whenever $\lambda_t\in\Lambda_t$. When $\lambda_t \notin \Lambda_{t}$, we further bound the size of the ``student-tie'' set $\Lambda_{s}(\lambda_t)$ as follows.

\begin{proposition}[Upper bound on student-tie set size]
\label{prop:strict_improv_ood}
For any $\lambda_t \notin \Lambda_t$, the cardinality of $\Lambda_{s}(\lambda_t)$ is bounded above by the number of sign changes in the sequence $\{d_1,\ldots,d_{\tilde m}\}$ after deleting zero entries, where the $d_j$ are defined in \eqref{eq:d_j_def}.
\end{proposition}

Thus, under the mild nondegeneracy condition $\lambda_t \notin \Lambda_t$, the deterministic equivalent of the optimal PMSD risk is \emph{strictly smaller} than that of the teacher for almost every $\lambda_s$.
Equivalently, the two risks can tie only at finitely many values of $\lambda_s$ for each fixed $p$. When $\Sigma_s$ has $p$ distinct eigenvalues,
\[d_j = c_j = \frac{\kappa_t \sigma_j^2}{(\sigma_j + \kappa_t)^2} 
\biggl( \frac{\kappa_t^2 b_t q_{2,0} + \sigma^2 b_t}{\kappa_t n_t} - v_j \biggr), \quad j = 1, \ldots, p,
\]
Thus, many sign changes in the sequence $\{ d_1, \ldots, d_p \}$ can occur only if the sequence $\{ v_1, \ldots, v_p\}$ oscillates repeatedly across the \emph{exact} threshold $\frac{\kappa_t^2 b_t q_{2,0} + \sigma^2 b_t}{\kappa_t n_t}$, which is typically a restrictive configuration.
For common covariance-signal models, we obtain the following stronger guarantees.

\subsubsection{In-distribution fresh unlabeled covariates}

First, consider fresh unlabeled covariates drawn from the teacher's training distribution, so that $\Sigma_s = \Sigma_t$.

\begin{corollary}
\label{cor:strict_improv_special_cases}
Assume $\Sigma_s = \Sigma_t$. The following guarantees of strict improvement hold.
\begin{enumerate}[leftmargin=7mm]
      \item [(a)] For an isotropic design $\Sigma_t = I_p$ with $r^2>0$, when
      \[\lambda_t \notin \Lambda_t = \{\gamma_t \sigma^2 / r^2\},
      \]
      we have $\sR_{\m}^{\star}(\lambda_t, \lambda_s) < \sR(\lambda_t)$ for all $\lambda_s > 0$.
    \item [(b)] For a spiked covariance model with $s$ spikes, when
    \[\lambda_t \notin \Lambda_t = \{ \lambda: d_{j}(\lambda) = 0 \text{ for every } j \in [s + 1] \},
    \]
    the inequality $\sR_{\m}^{\star}(\lambda_t, \lambda_s) < \sR(\lambda_t)$ fails for at most $s$ values of $\lambda_s > 0$.
    \item [(c)] For $\Sigma_t=U^{\top}\operatorname{diag}(\sigma_1,\ldots,\sigma_p)U$ with $p$ distinct eigenvalues satisfying $\sigma_1<\cdots<\sigma_p$, suppose the signal alignment is monotone along the spectrum and that there exists an $i$ such that
    \[
    v_i \neq \frac{\kappa_t^2 b_t q_{2,0} + \sigma^2 b_t}{\kappa_t n_t}.
    \]
    Then $\lambda_t\notin\Lambda_t$, and $\sR_{\m}^{\star}(\lambda_t,\lambda_s)<\sR(\lambda_t)$ fails for at most one $\lambda_s>0$.
\end{enumerate}

\end{corollary}

For example, in the isotropic design setting, the ``teacher-degeneracy'' set is the singleton $\Lambda_t = \{\lambda_t^{\star}\}$, where $\lambda_t^{\star}=\gamma_t\sigma^2/r^2$ is the ridge-optimal value.
Consequently, the PMSD student does not strictly improve upon the optimally tuned teacher, but it does improve for every other $\lambda_t>0$ and any choice of $\lambda_s>0$.
\Cref{fig:teaser_a} illustrates this property on the Blog Feedback dataset, using fresh unlabeled covariates taken from the same data pool as the teacher's data (but with different samples).

As another consequence of \Cref{thm:general_ridge}, when the fresh sample size matches the teacher's sample size and the same regularization is used, fresh-$X$ prediction mixing recovers the tangent identity that drives strict improvement in same-$X$, same-$\lambda$ ridge self-distillation in \cite{dang2026optimal}.

\begin{corollary}[Same-$\lambda$ balanced regime]
\label{cor:balanced_regime}
Assume $\Sigma_s = \Sigma_t$.
If $\lambda_t = \lambda_s = \lambda$ and $\gamma_s = \gamma_t$, then $\kappa_s = \kappa_t = \kappa$ and the fresh-$X$ optimal mixing asymptotically matches the same-$X$ ridge geometry.
In particular,
\begin{equation}
\label{eq:balanced_tangent}
\mathcal R_{\te}(\lambda)-\mathcal C(\lambda, \lambda) =-\frac{\kappa}{2b_t^2}\,\mathcal R_{\te}'(\lambda).
\end{equation}
Hence, whenever $\mathcal R_{\te}'(\lambda)\neq 0$, $\sR_{\m}^{\star}(\lambda, \lambda) < \sR(\lambda)$
and $\sign\!\big(\xi^{\star}(\lambda, \lambda)\big) = -\sign\!\big({\sR}'(\lambda)\big)$.
\end{corollary}

The deterministic limiting risks and the sign rule for the mixing weight in this same-$\lambda$ balanced regime are illustrated in \Cref{fig:asymp_curve}: (i) the PMSD and teacher risks coincide only at the teacher's stationary point, and (ii) the sign of the optimal mixing weight is opposite to the sign of the derivative of the teacher risk curve.

\subsubsection{Out-of-distribution fresh unlabeled covariates}

\begin{figure*}[t]
   \centering
    \begin{subfigure}[t]{0.48\textwidth}
    \centering
    \includegraphics[width=\textwidth]{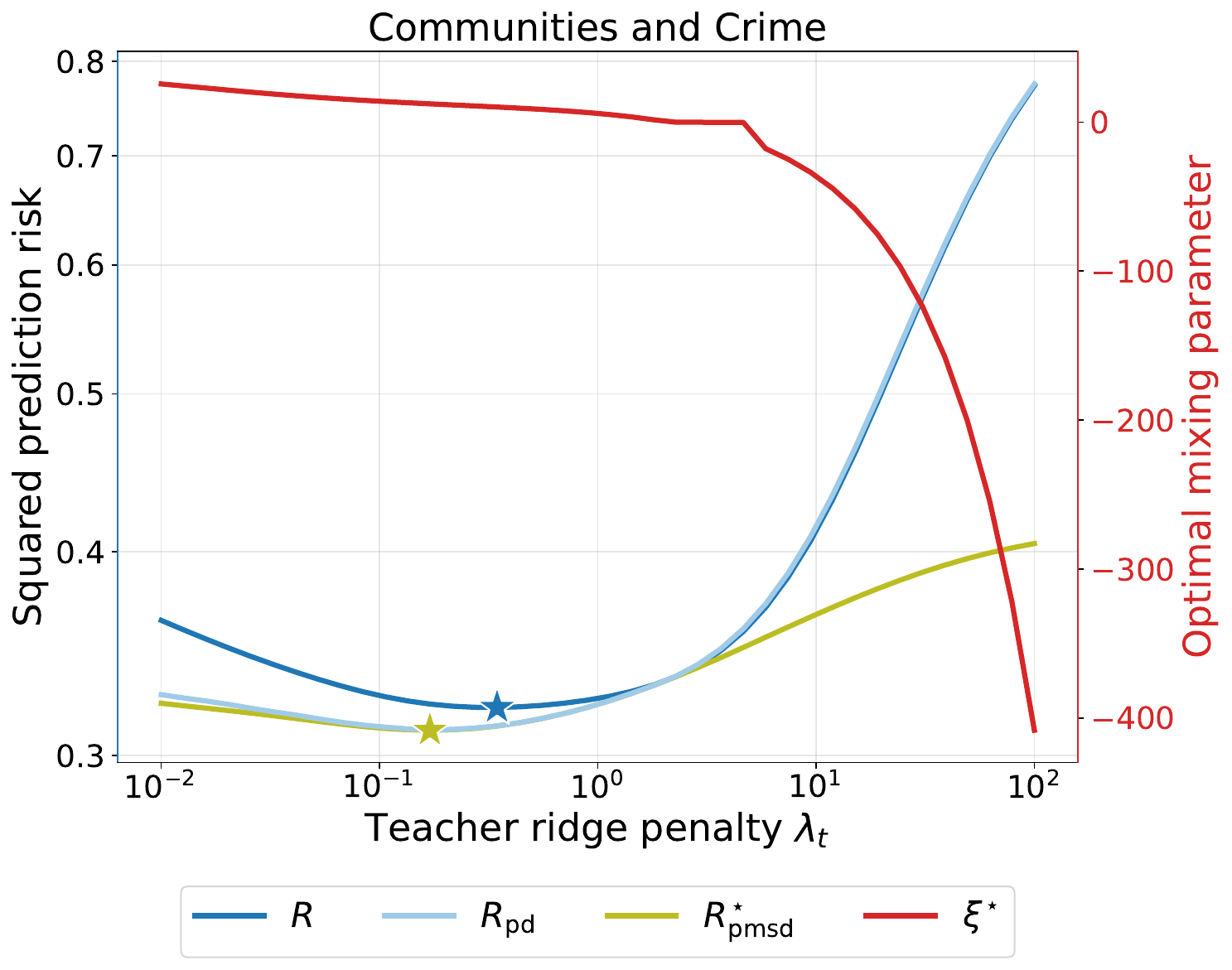}
    \caption{$n_t = 400$, $n_s = 800$, $p = 99$.}
  \end{subfigure}
  \quad
    \begin{subfigure}[t]{0.48\textwidth}
    \centering
    \includegraphics[width=\textwidth]{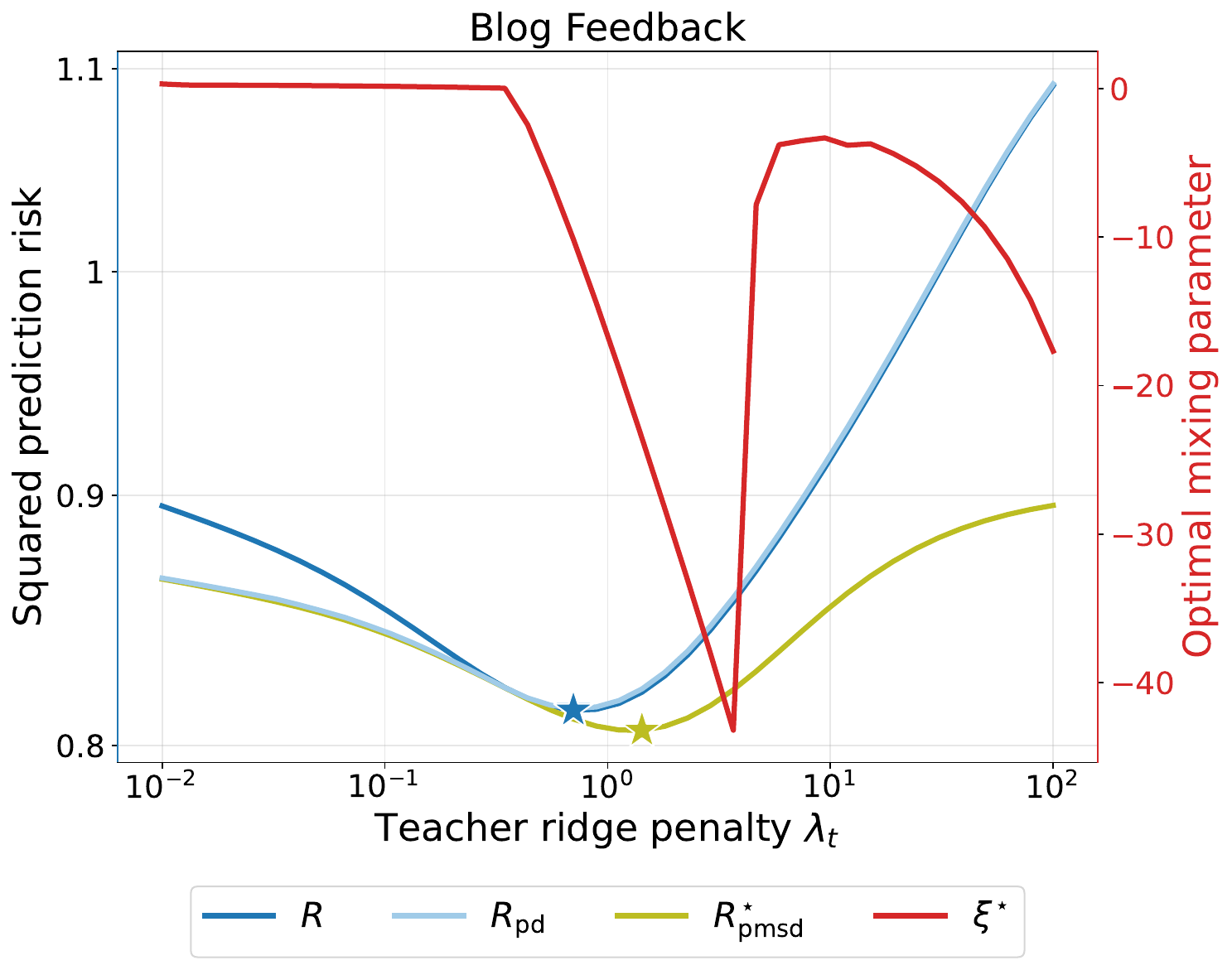}
    \caption{$n_t = 2619$, $n_s = 5240$, $p = 280$.}
  \end{subfigure}
    \caption{\textbf{PMSD using fresh covariates from an isotropic Gaussian distribution on real-world regression tasks.} At each value of $\lambda_t$, the student penalty $\lambda_s$ is tuned separately over a grid for the PD and PMSD students.}
    \label{fig:fresh_isotropic_real_world_same_lambda}
\end{figure*}

In many settings, we may not know $\Sigma_t$ or have access to fresh covariates from the same distribution as the teacher's training data.
A simple and practical alternative is to sample the fresh unlabeled covariates from an isotropic distribution, which commutes with any $\Sigma_t$.
In this case, we obtain the following strict improvement guarantee.

\begin{corollary}
    \label{corr:strict_improv_isotropic_fresh}
    Assume $\Sigma_s = I_p$. For any $\lambda_t > 0$ such that $\sum_{i=1}^{p} c_i \neq 0$, we have $\sR_{\m}^{\star}(\lambda_t, \lambda_s) < \sR(\lambda_t)$ for all $\lambda_s > 0$.
\end{corollary}

To see this, setting $\Sigma_s = I_p$ in equation \eqref{eq:R-C} gives
\begin{equation}
    \mathcal R_{\te}(\lambda_t)-\mathcal C(\lambda_t, \lambda_s) 
    = \frac{\sum_{i=1}^{p} c_i}{1 + \kappa_s}.
\end{equation}  

Thus, if $d_1 = \sum_{i=1}^{p} c_i \neq 0$, strict improvement holds for every $\lambda_s > 0$.
If $d_1 = 0$, the two risks tie for all $\lambda_s$.
For example, \Cref{cor:strict_improv_special_cases} shows that when $\Sigma_t = I_p$, this degeneracy occurs only at the ridge-optimal value $\lambda_t^{\star}$.
More generally, the ``teacher-degeneracy'' condition can be rewritten as
\begin{align}
    \label{eq:ood_degeneracy}
    \sum_{i=1}^{p} \frac{ \kappa_t \sigma_i^2}{(\sigma_i + \kappa_t)^2} 
\biggl(\frac{\kappa_t^2 b_t q_{2,0} + \sigma^2 b_t}{\kappa_t n_t} - v_i \biggr) = 0,
\end{align}

which is a polynomial in $\kappa_t$ with degree at most $4m-1$ (see \Cref{sec:proof_teacher_bad_set}).

We now compare the ``teacher-degeneracy'' condition with the stationary points of the teacher risk function to see whether isotropic fresh covariates can improve upon a tuned teacher.
By equation \eqref{eq:R_prime}, the stationary condition $\sR'(\lambda_t)=0$ is equivalent to
\begin{align}
 \label{eq:ood_R_prime}
  \sum_{i=1}^{p}
 \frac{\kappa_t \sigma_i^2}{(\sigma_i + \kappa_t)^3}
 \biggl(\frac{\kappa_t^2 b_t q_{2,0} + \sigma^2 b_t}{\kappa_t n_t} - v_i \biggr) = 0.
\end{align}

Therefore, strict improvement over a teacher at its stationary point fails if and only if both equations \eqref{eq:ood_degeneracy} and \eqref{eq:ood_R_prime} hold.
Equivalently,
\begin{align} 
    \ddfrac{\sum_{i=1}^p \frac{\kappa_t \sigma_i^2}{(\sigma_i + \kappa_t)^2} v_i}{\sum_{i=1}^p \frac{\kappa_t \sigma_i^2}{(\sigma_i + \kappa_t)^2}}
    = 
    \ddfrac{\sum_{i=1}^p \frac{\kappa_t \sigma_i^2}{(\sigma_i + \kappa_t)^3} v_i}{\sum_{i=1}^p \frac{\kappa_t \sigma_i^2}{(\sigma_i + \kappa_t)^3}}
    = 
    \frac{\kappa_t^2 b_t q_{2,0} + \sigma^2 b_t}{\kappa_t n_t}.
\end{align}
In general, this imposes a strong alignment condition: two different weighted averages of the squared signal-alignment values $\{v_i\}_{i=1}^p$ must equal the same threshold.
Thus, we expect this simultaneous equality to be non-generic unless the covariance or signal alignment has special structure.

We illustrate the benefit of the PMSD student using isotropic fresh unlabeled covariates on real-world regression tasks in \Cref{fig:fresh_isotropic_real_world_same_lambda}. Additional experiments appear in \Cref{sec:additional_exp_regression}.

\subsection{Strict improvement over the pure-distilled student}

To compare the PMSD student with the pure-distilled student, we use a decomposition analogous to equation \eqref{eq:R_star_asymp}:
\begin{align}
    \sR^{\star}_{\m}(\lambda_t, \lambda_s) = \mathcal R_{\pd}(\lambda_s)
    -
    \frac{\bigl(\mathcal R_{\pd}(\lambda_s)-\mathcal C(\lambda_t, \lambda_s)\bigr)^2}{\mathcal R_{\te}(\lambda_t)+\mathcal R_{\pd}(\lambda_s) - 2 \mathcal C(\lambda_t, \lambda_s)}.
\end{align}
Thus, the PMSD student strictly improves upon the PD student whenever
\[\mathcal R_{\pd}(\lambda_s)-\mathcal C(\lambda_t, \lambda_s) \neq 0.\]

Since this condition is quite involved, we specialize to the isotropic design setting, where it admits a clean characterization.

\begin{proposition}[Strict improvement over the PD student]
\label{prop:strict_improv_pd}
Assume $\Sigma_t = \Sigma_s = I_p$ and $r^2+\sigma^2>0$.
Then, for any $\lambda_t > 0$, there are at most two values of $\lambda_s > 0$ satisfying $\mathcal R_{\pd}(\lambda_s)-\mathcal C(\lambda_t, \lambda_s) = 0$.
\end{proposition}

Consequently, in the isotropic setting, the PMSD student strictly improves upon the PD student for all but at most two choices of $\lambda_s > 0$.

\subsection{Non-monotonicity in the amount of unlabeled data}
\label{sec:ridge_mono_main}

\begin{figure*}[t]
   \centering
    \begin{subfigure}[t]{0.3\textwidth}
    \centering
    \includegraphics[width=\textwidth]{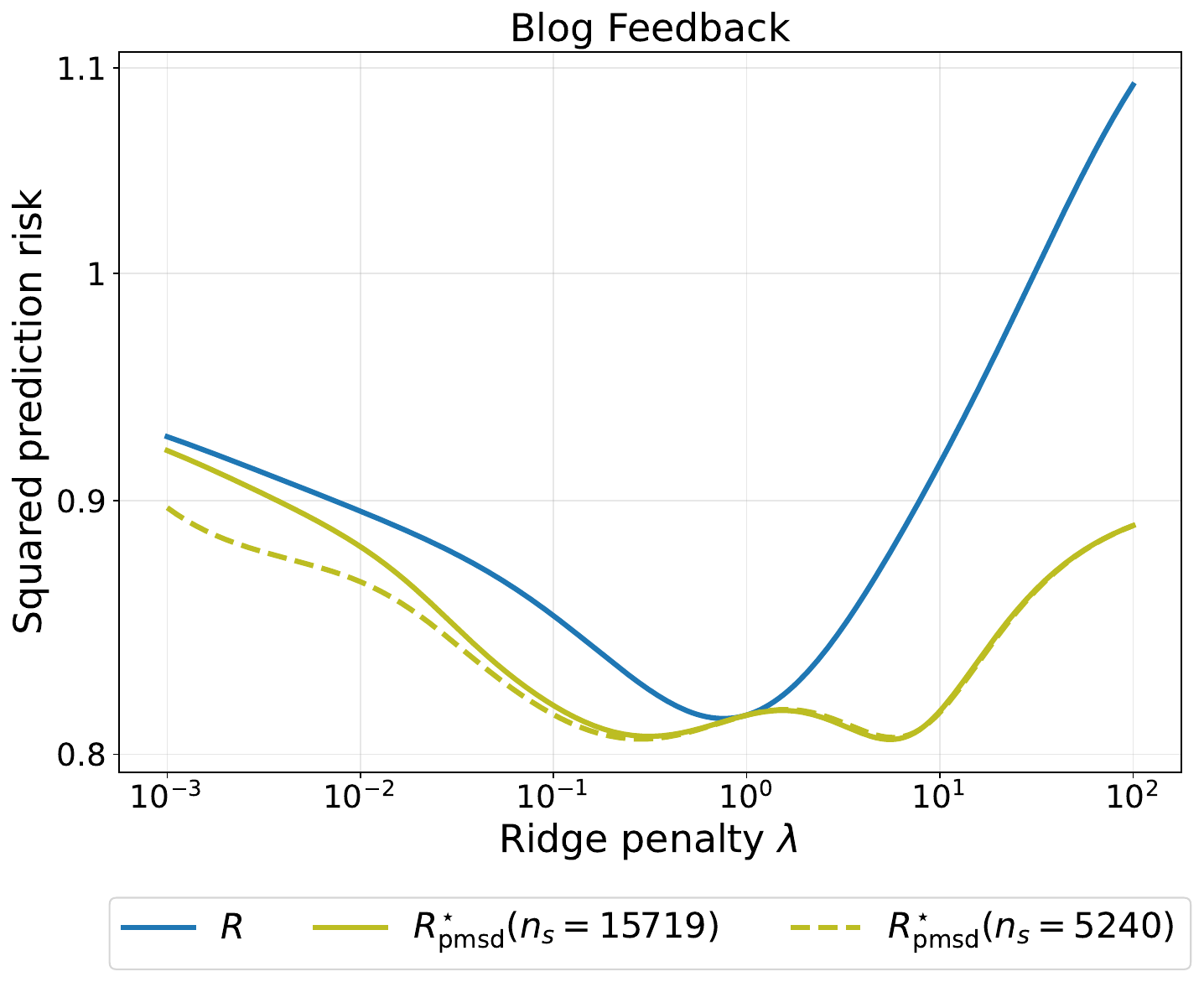}
    \caption{Ridge regression on the Blog Feedback dataset with $p = 280$, $n_t = 2619$.}
  \end{subfigure}
  \hfill
    \begin{subfigure}[t]{0.3\textwidth}
    \centering
    \includegraphics[width=\textwidth]{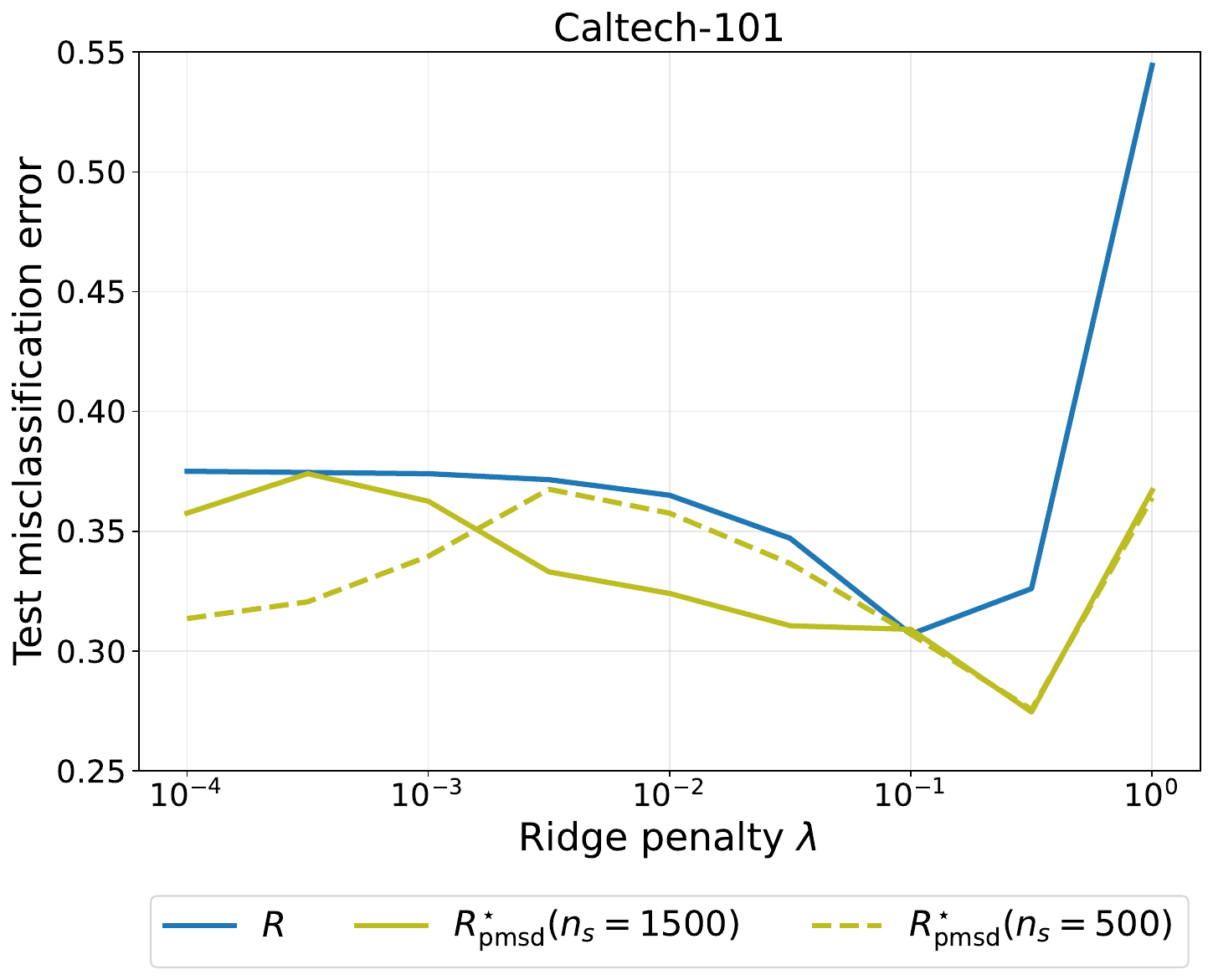}
    \caption{Linear probing on Caltech-101 using pretrained ResNet-34 features with $p = 512$, $n_t = 1500$, $n_{\mathrm{cal}} = 500$, and corruption rate $\rho = 0.4$; see \Cref{sec:logistic} for definitions.}
  \end{subfigure}
    \hfill
  \begin{subfigure}[t]{0.3\textwidth}
    \centering
    \includegraphics[width=\textwidth]{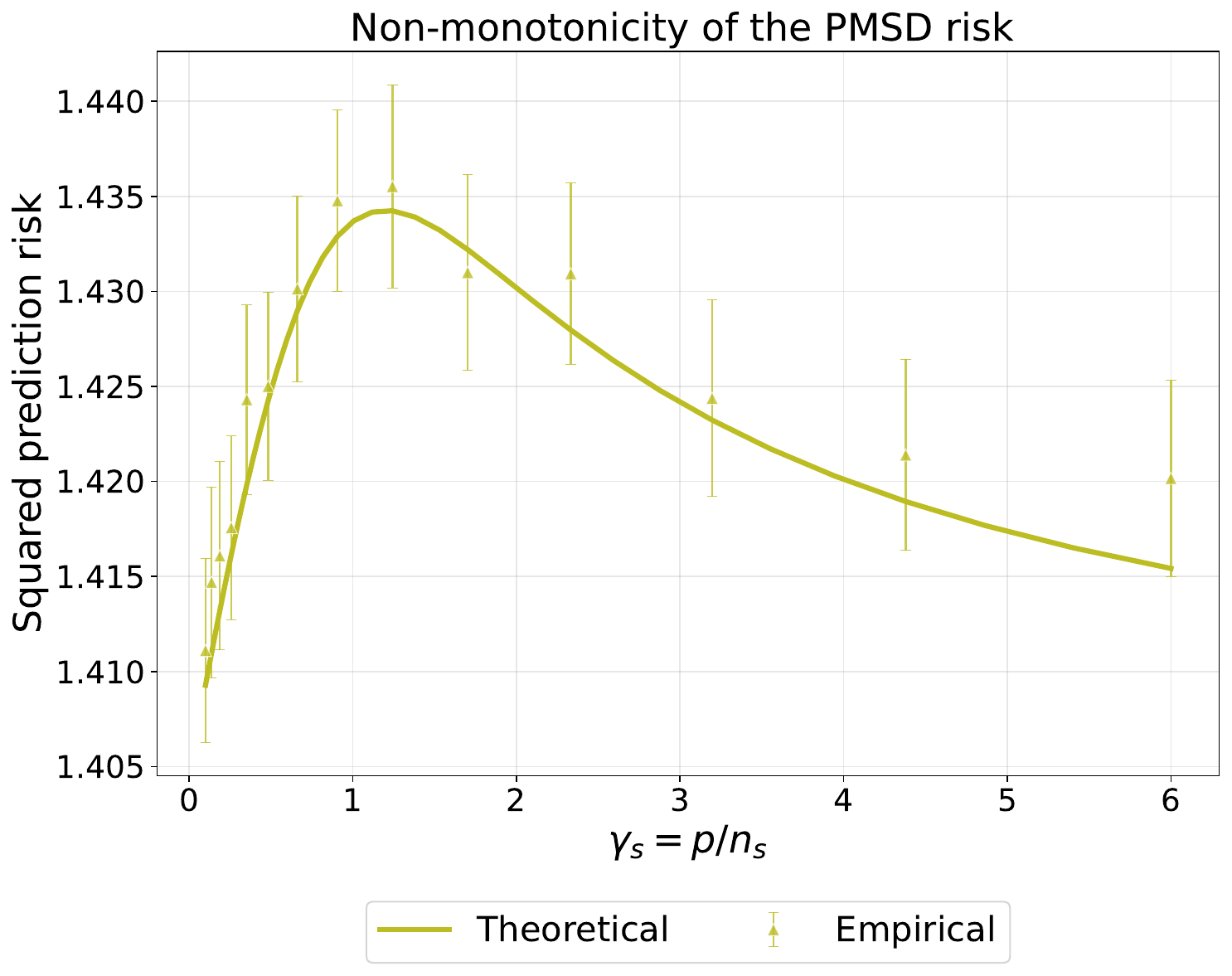}
    \caption{Same-$\lambda$ isotropic design and isotropic signal. Setting: $\lambda = 0.2$, $p = 400$, $n_t = 800$, and $r^2 = \sigma^2 = 1$.}
  \end{subfigure}
    \caption{Non-monotonicity of the PMSD risk with the amount of fresh unlabeled data in the same-$\lambda$, in-distribution setting.}
    \label{fig:monotonicity}
\end{figure*}

A natural question is whether increasing the amount of fresh unlabeled data always improves the PMSD student.
In the proportional regime, this amounts to fixing the labeled aspect ratio $\gamma_t$ and studying the optimal PMSD risk as a function of $\gamma_s$, where smaller $\gamma_s$ corresponds to more unlabeled data.
We specialize to the same-$\lambda$ isotropic setting for this analysis.
Surprisingly, for a fixed $\lambda_t = \lambda_s = \lambda$, the optimal PMSD risk is \emph{not} globally monotone in $\gamma_s$.
Non-monotone risk profiles as functions of aspect ratio are well documented in high-dimensional prediction, and subsampling- and aggregation-based procedures can sometimes restore monotonicity \citep{patil2022mitigating, patil2023bagging}.
The phenomenon here is different: the labeled training sample and teacher remain fixed, while only the amount of fresh unlabeled data varies.
Even so, the optimal PMSD risk has a simple one-dimensional structure and is unimodal in $\gamma_s$.
See \Cref{sec:appendix_monotonicity} for the formal result and \Cref{fig:monotonicity} for illustrations from synthetic and real data experiments.

\section{Data-dependent tuning}
\label{sec:tuning_ridge}

The results in \Cref{sec:strict_nonmonotonicity} demonstrate the benefits of fresh-$X$ distillation with optimal prediction mixing $\xi^{\star}$.
However, the deterministic equivalents in \Cref{thm:general_ridge} depend on population quantities that are unknown in practice, such as the covariance spectrum and signal-covariance alignment.
In this section, we study how to tune the mixing weight $\xi$. We show that fresh unlabeled samples alone are insufficient, whereas an independent labeled calibration set facilitates consistent tuning.

Recall that the optimal \emph{oracle} mixing weight from \Cref{prop:oracle_affine} depends on residual correlations with the unknown response.
Thus, in the fully prediction-only regime, unlabeled covariates and teacher outputs alone are insufficient to identify $\xi^\star$.
This is an information-theoretic limitation, not a computational one.
To see this, fix the observed teacher predictor $f^{\te}$ and the induced fresh-$X$ PD student $f_{\pd}$.
These predictors are observable in the prediction-only setting, while the response law is not.
Consider two response models with the same covariate distribution: $y=f^{\te}(x)+\varepsilon$ and $y=f_{\pd}(x)+\varepsilon$, where $\varepsilon$ is mean-zero noise independent of $x$.
Their oracle mixing weights are $\xi^\star=0$ and $\xi^\star=1$, respectively.
Therefore, no estimator based only on unlabeled samples and teacher queries can be uniformly consistent for $\xi^\star$.
We give a formal non-identifiability statement in \Cref{sec:appendix_tuning}.

However, $\xi^\star$ can be consistently estimated using an independent labeled calibration set $\cD_{\rm cal}:=\{(x_i^{\rm cal},y_i^{\rm cal})\}_{i=1}^{n_{\rm cal}}$, which is used only for evaluation and not retraining.
Conditional on the teacher predictor and fresh unlabeled data, both $f^{\te}$ and $f_{\pd}$ are fixed predictors.
For an independent test point $(x,y)$ drawn from the target distribution, define the residuals $e^{\te}:=y-f^{\te}(x)$ and $e^{\pd}:=y-f_{\pd}(x)$.
Define the calibration residuals
\[
  \hat e_i^{\te}:=y_i^{\rm cal}-f^{\te}(x_i^{\rm cal})
  \quad \text{and} \quad
  \hat e_i^{\pd}:=y_i^{\rm cal}-f_{\pd}(x_i^{\rm cal}),
\]
and estimate the three oracle quantities by
\[
  \widehat R_{\te}^{\rm cal}:=\frac1{n_{\rm cal}}\sum_{i=1}^{n_{\rm cal}}(\hat e_i^{\te})^2,
  \qquad
  \widehat R_{\pd}^{\rm cal}:=\frac1{n_{\rm cal}}\sum_{i=1}^{n_{\rm cal}}(\hat e_i^{\pd})^2,
  \qquad
  \widehat C^{\rm cal}:=\frac1{n_{\rm cal}}\sum_{i=1}^{n_{\rm cal}}\hat e_i^{\te}\hat e_i^{\pd}.
\]
Define the one-shot plug-in calibration estimators of $\xi^\star$ and $R_{\m}^\star$ by
\begin{equation}
\label{eq:cal_xi}
  \widehat\xi^\star_{\rm cal}
  :=
  \frac{\widehat R_{\te}^{\rm cal}-\widehat C^{\rm cal}}
  {\widehat R_{\te}^{\rm cal}+\widehat R_{\pd}^{\rm cal}-2\widehat C^{\rm cal}},
  \quad \text{and} \quad
  \widehat R_{\m,{\rm cal}}^\star
  :=
  \widehat R_{\te}^{\rm cal}
  -
  \frac{(\widehat R_{\te}^{\rm cal}-\widehat C^{\rm cal})^2}
  {\widehat R_{\te}^{\rm cal}+\widehat R_{\pd}^{\rm cal}-2\widehat C^{\rm cal}}.
\end{equation}
The next result establishes their consistency.

\begin{theorem}[Calibration consistency]
\label{thm:calibration_consistency}
Fix the teacher and the fresh unlabeled sample, and suppose the oracle denominator $D:=R_{\te}+R_{\pd}-2C$ is bounded away from zero.
If the calibration sample is independent, $n_{\rm cal}\to\infty$, and the conditional variances of $(e^{\te})^2$, $(e^{\pd})^2$, and $e^{\te}e^{\pd}$ are uniformly bounded, then
\[
  \widehat\xi^\star_{\rm cal}-\xi^\star \pto 0,
  \qquad
  \widehat R_{\m,{\rm cal}}^\star-R_{\m}^\star \pto 0.
\]
Moreover, the same convergence holds uniformly over any fixed finite collection of candidate students under the corresponding uniform conditions.
\end{theorem}

\begin{figure*}[t]
  \centering
    \begin{subfigure}[t]{0.32\textwidth}
    \centering
    \includegraphics[width=\textwidth]{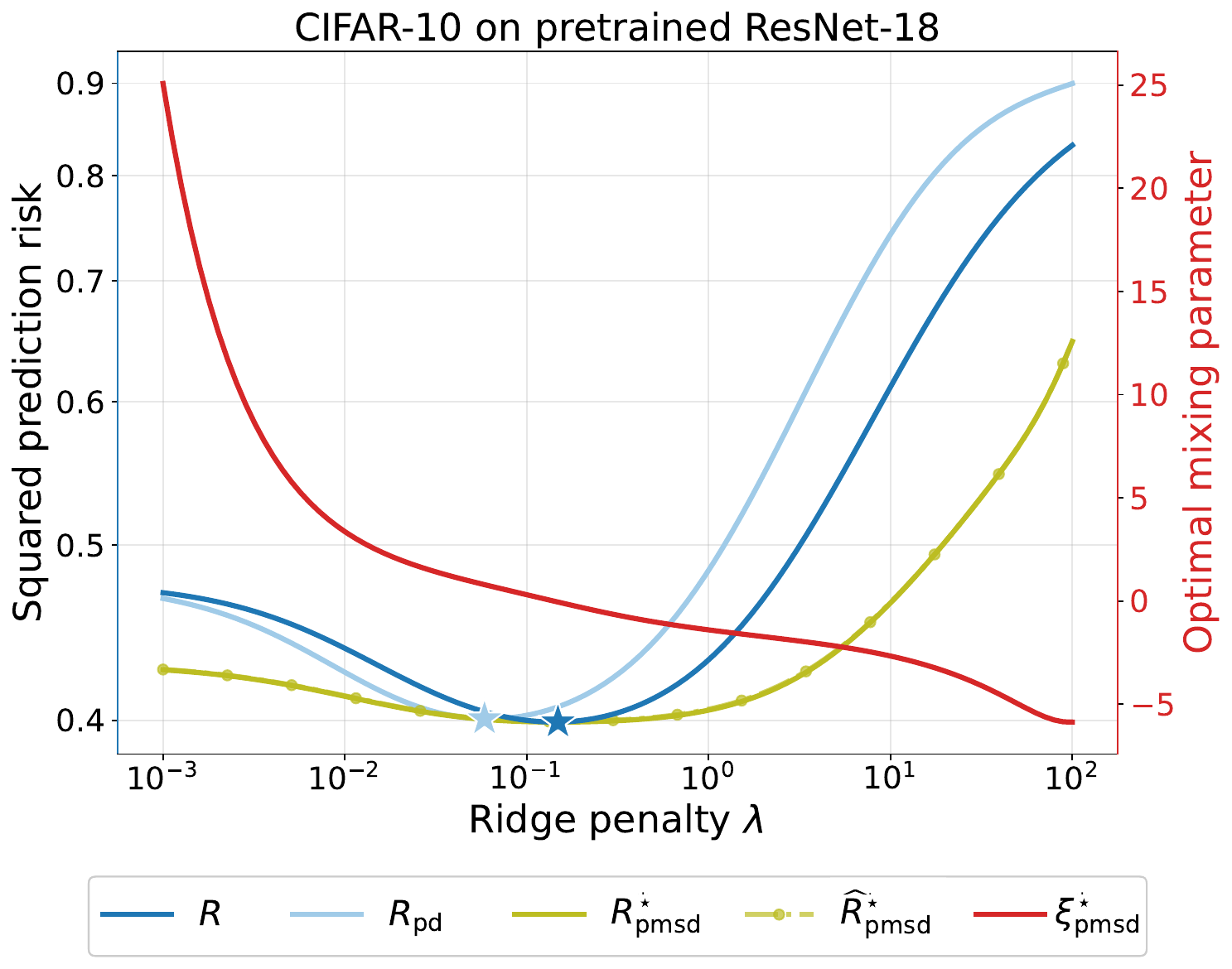}
    \caption{$p = 512$, $n_t = 2000$, $n_s = 4000$, $n_{\mathrm{cal}} = 200$.}
  \end{subfigure}
  \hfill
    \begin{subfigure}[t]{0.32\textwidth}
    \centering
    \includegraphics[width=\textwidth]{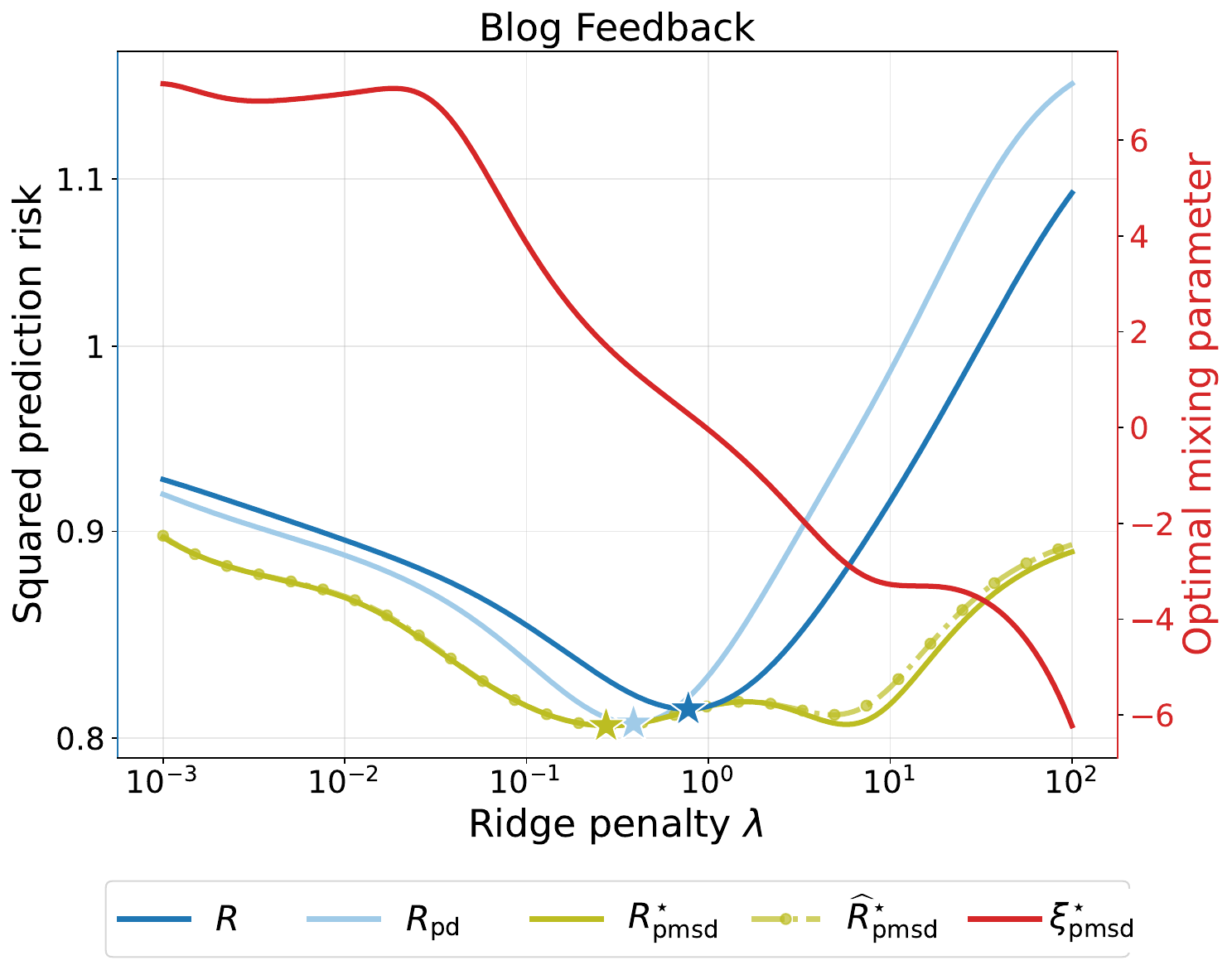}
    \caption{$p = 280$, $n_t = 2619$, $n_s = 5240$, $n_{\mathrm{cal}} = 524$.}
  \end{subfigure}
  \hfill
   \begin{subfigure}[t]{0.32\textwidth}
    \centering
    \includegraphics[width=\textwidth]{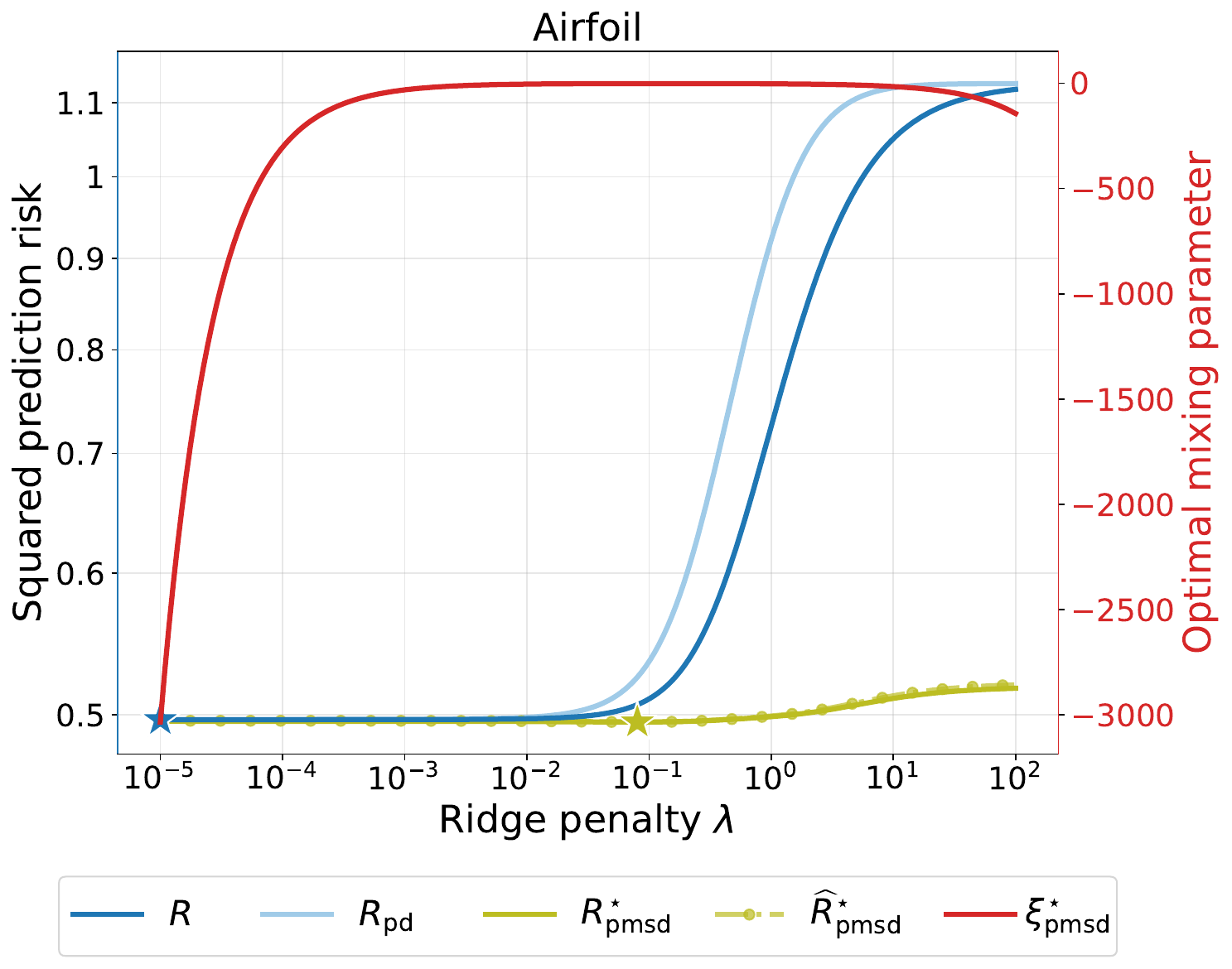}
    \caption{$p = 5$, $n_t = 150$, $n_s = 150$, $n_{\mathrm{cal}} = 76$.}
  \end{subfigure}
  \caption{PMSD risk estimated using a small calibration set of size $n_{\mathrm{cal}}$ on real datasets. We use the same-$\lambda$ setting with $\lambda_t = \lambda_s = \lambda$.}
  \label{fig:ridge_tuning}
\end{figure*}

The assumptions required for consistency are minimal.
Because the calibration set estimates only the risk and residual correlation of fixed predictors, no asymptotic condition involving $p/n_{\rm cal}$ is required. For consistency, $n_{\rm cal}$ need only diverge, regardless of the growth rate of $p$.
Computationally, tuning $\xi$ is inexpensive: after training the PD student once, \eqref{eq:cal_xi} requires only evaluating the two predictors on the calibration points and computing three empirical averages.
The uniform finite-candidate statement also allows the same calibration set to choose among finitely many student penalties $\lambda_s$ or unlabeled sample sizes $n_s$.
\Cref{fig:ridge_tuning} illustrates calibration-based tuning with small labeled samples on real datasets.
Reducing the label cost of this calibration step through alternative sampling or model-assisted evaluation schemes is a natural direction for future work \citep{fogliato2024framework}.

\section{Logistic regression}
\label{sec:logistic}

We now turn to the fresh-$X$ prediction-only setting for logistic regression with cross-entropy loss. Our goal is to show that the PMSD student can strictly outperform both the teacher and the PD student.

\textbf{Setting.} A theoretical analysis of self-distillation for classification is more challenging than for regression. To streamline the presentation, we adopt the setting of \citet{das2023understanding, jeong2025rethinking}, which we review here. We study a binary classification problem in which each instance $x \in \cX$ has a ground-truth label $y(x) \in \{0,1\}$. Let $\cP$ denote the underlying distribution on $\cX$. We have access to an unlabeled dataset $\cD_{\rm unlab} = \{x_i\}_{i=1}^{2n}$ consisting of $2n$ i.i.d.\ draws from $\cP$. The corresponding labels $y_i:=y(x_i)$ are deterministic functions of the instances but are unobserved. We assume that the dataset is balanced and index the observations so that $y_i = 1$ for $i \in \mathcal{S}_1 := \{1, \ldots, n\}$ and $y_i = 0$ for $i \in \mathcal{S}_0 := \{n+1, \ldots, 2n\}$.

\textbf{Teacher.} We have access to a teacher classifier that returns a hard label $\hat y(x) \in \{0,1\}$ for any $x\in\cX$. Although its form is unknown, we can query its predictions. Suppose that its population classification error is the same within each class and equals an unknown $\rho\in(0,1)$:
\[
\PP_{x\sim\cP}\bigl(\hat y(x)\neq y(x)\mid y(x)=0\bigr)
=
\PP_{x\sim\cP}\bigl(\hat y(x)\neq y(x)\mid y(x)=1\bigr)
=\rho.
\]
Following \citet{das2023understanding}, let $\hat n:=\lfloor n\rho\rfloor$ and assume that the teacher makes $\hat n$ errors in each class. After reindexing the observations within each class, write the teacher's labels on $\cD_{\rm unlab}$ as
\[
\hat{y}_i := \hat y(x_i) =
\begin{cases}
1 - y_i & \text{for } i \in \underbrace{\{1, \ldots, \hat{n} \}}_{\mathcal{S}_{1,\text{bad}}} \, \cup \, \underbrace{\{n+1,\ldots,n+\hat{n}\}}_{\mathcal{S}_{0,\text{bad}}},
\\
y_i & \text{for } i \in 
\underbrace{\{\hat{n}+1,\ldots,n\}}_{\mathcal{S}_{1,\text{good}}} \, \cup \, \underbrace{\{n+\hat{n}+1,\ldots,2n\}}_{\mathcal{S}_{0,\text{good}}}.
\end{cases}
\]

Thus, $|\mathcal{S}_{1,\text{bad}}|=|\mathcal{S}_{0,\text{bad}}|=\hat n$, and the fraction of incorrect labels in each class converges to $\rho$ as $n\to\infty$.

\textbf{Students.} We assume a feature map $\phi:\cX\to\cH$, where $\cH$ is a Hilbert space, and keep $\phi$ fixed throughout. We train a linear classifier in $\cH$ on the teacher's hard pseudo-labels $\{\hat y_i\}_{i=1}^{2n}$ using $\ell_2$-regularized logistic loss:
\begin{align}
    \label{eq:class_training_loss}
    \theta_{\pd} := \argmin_{\theta}
    \frac{1}{2n} \sum_{i=1}^{2n}
    \text{CE}\big(\hat y_i, \sigma( \langle \theta, \phi(x_i) \rangle) \big) + \frac{\lambda}{2} \| \theta \|^2,
\end{align}
where $\text{CE}(q, \hat q) := -q\log(\hat q)-(1-q)\log(1-\hat q)$ is the binary cross-entropy loss and $\sigma(z):=(1+e^{-z})^{-1}$ is the sigmoid function. The resulting PD student's soft prediction is $y_{\pd}(x) := \sigma(\langle \theta_{\pd}, \phi(x) \rangle) \in (0,1)$. This setting corresponds to linear probing in neural networks, where a learnable linear layer and sigmoid activation are placed on top of a frozen feature backbone. Given a mixing weight $\xi \in \RR$, the PMSD classifier is
\begin{equation}
\label{eq:logistic_affine}
y_{\m, \xi}(x)=\II \bigl\{(1-\xi)\hat y(x)+\xi y_{\pd}(x)\ge 0.5 \bigr\}.
\end{equation}

We make the following assumption on the feature map $\phi$ and the unlabeled dataset $\cD_{\rm unlab}$.

\begin{assumption}[Feature norm and correlation structure]
    \label{assp:feature_corr}
The features have unit norm, $\|\phi(x)\|=1$ for all $x\in\cX$. Let $\Phi\in\mathbb{R}^{2n\times 2n}$ denote the Gram matrix on $\cD_{\rm unlab}$, with $[\Phi]_{ij}:=\langle\phi(x_i),\phi(x_j)\rangle$. We assume that $[\Phi]_{ij}=0$ when $y_i\neq y_j$ and $[\Phi]_{ij}=c\in(0,1)$ when $y_i=y_j$ and $i\neq j$.
\end{assumption}

This assumption, formulated by \citet{das2023understanding}, is technically convenient because it makes the data separable with respect to their unobserved ground-truth classes, thereby ensuring the existence of a perfect classifier. Moreover, \citet{jeong2025rethinking} empirically validate this assumption on six real datasets: CIFAR-100, Caltech-101, Caltech-256, Flowers-102, Food-101, and Stanford Cars. They use features from a ResNet-34 pretrained on ImageNet; see their Figure 1 and Appendix D.2.

\textbf{Goal.} Given a soft prediction rule $y_{\theta}:\cX\to[0,1]$, an instance $x$ is classified correctly when $y(x)=\II\{y_{\theta}(x)\geq 0.5\}$. By construction, the teacher's accuracy on $\cD_{\rm unlab}$ converges to $100(1-\rho)\%$. We show in \Cref{thm:class_mixing} that, when $\rho$ is sufficiently close to $0.5$ from below, the PD student trained by \eqref{eq:class_training_loss} also achieves only $100(1-\rho)\%$ accuracy as $n\to\infty$. Our goal is to construct a PMSD classifier as in \eqref{eq:logistic_affine} that achieves asymptotic $100\%$ accuracy and thus perfectly recovers the true labels in the limit.

Prior work \citep{das2023understanding, jeong2025rethinking} characterizes conditions on $\rho$ and $\lambda$ under which the PD student $y_{\pd}$ itself achieves this goal. These works refer to the result as ``$100\%$ population accuracy'' because the generalization gap vanishes for finite-dimensional linear classifiers as $n\to\infty$ \citep{kakade2008complexity, koltchinskii2002empirical}.

\begin{theorem}
    \label{thm:class_mixing}
    Let $\hl_n := 2n \lambda$ and suppose that $\hl_n \to \hl \in (0,\infty)$ as $n \to \infty$. Under \Cref{assp:feature_corr}, there exists a threshold $\rho_0 = \rho_0(c, \hl) \in (0,0.5)$ such that, for every $\rho \in (\rho_0, 0.5)$, the PD student $y_{\pd}$ has $100(1-\rho)\%$ population accuracy. On the other hand, for every $\rho < 0.5$, there exists $\xi \geq 0$ such that the PMSD student $y_{\m, \xi}$ achieves $100\%$ population accuracy.
\end{theorem}

Self-distillation has been observed to induce label averaging among highly correlated instances \citep{jeong2025rethinking}. When $\rho$ is small, the PD student can therefore correct mislabeled samples by aggregating labels from correctly classified neighbors, although incorrect neighbors also reduce its confidence on correctly labeled samples. Relative to the teacher, the student can thus trade confidence for accuracy. As $\rho$ approaches $0.5$ from below, however, incorrectly labeled neighbors increasingly mislead the student, and its predictions no longer cross the classification threshold on the mislabeled samples. Choosing $\xi>1$ extrapolates beyond the PD student's soft prediction and can amplify this corrective direction, although an excessively large $\xi$ may hurt performance. We illustrate this behavior with synthetic experiments in \Cref{fig:synthetic_constant_small,fig:synthetic_uniform_small}.

\subsection{Soft teacher–student prediction mixing}
\label{sec:extensions_logistic}

\Cref{thm:class_mixing} mixes the teacher's hard label with the PD student's soft prediction and strictly improves upon both when $\rho_0 < \rho < 0.5$. We now consider mixing soft predictions from both models. Specifically, we treat the PD student $y_{\pd}(\cdot) = \sigma(\langle \theta_{\pd}, \phi(\cdot) \rangle)$ defined by \eqref{eq:class_training_loss} as a teacher that produces soft labels. Using its predictions on $\cD_{\rm unlab}$, we train a second PD student $y_{\pd}^{(2)}(\cdot) = \sigma(\langle \theta_{\pd}^{(2)}, \phi(\cdot) \rangle)$ with the same regularization parameter $\lambda$:
\begin{align}
    \label{eq:class_training_loss_2}
    \theta_{\pd}^{(2)} &:= \argmin_{\theta}
    \frac{1}{2n} \sum_{i=1}^{2n}
    \text{CE} \big(y_{\pd}(x_i), \sigma( \langle \theta, \phi(x_i) \rangle) \big) + \frac{\lambda}{2} \| \theta \|^2.
\end{align}
For a mixing weight $\xi\in\RR$, the resulting PMSD classifier thresholds the affine combination of the two soft predictions:
\begin{equation}
\label{eq:logistic_affine_2}
y_{\m, \xi}^{(2)}(x)=\II \bigl\{(1-\xi) y_{\pd}(x) + \xi y_{\pd}^{(2)}(x)\ge 0.5 \bigr\}.
\end{equation}

Previous results (Theorem 5 of \citet{das2023understanding} and Theorem 4.1 of \citet{jeong2025rethinking}) show that, as $n\to\infty$, the second student $y_{\pd}^{(2)}$ can tolerate a higher label noise rate $\rho$ than its teacher $y_{\pd}$, subject to $\rho<0.5$. Specifically, there exists a threshold $\rho_1$ with $0 < \rho_0 < \rho_1 < 0.5$ such that (i) for $\rho_0 < \rho < \rho_1$, the student $y_{\pd}^{(2)}$ attains $100\%$ accuracy while the teacher $y_{\pd}$ has $100(1-\rho)\%$ accuracy, and (ii) for $\rho_1 < \rho < 0.5$, both achieve $100(1-\rho)\%$ accuracy. In contrast, we show that, for every $\rho<0.5$, an appropriate mixing parameter $\xi\geq0$ allows the PMSD classifier $y_{\m, \xi}^{(2)}$ in \eqref{eq:logistic_affine_2} to attain $100\%$ accuracy.

\begin{theorem}
    \label{thm:class_mixing_2}
    Let $\hl_n := 2n \lambda$ and suppose that $\hl_n \to \hl \in (0,\infty)$ as $n \to \infty$. Under \Cref{assp:feature_corr}, for every $\rho < 0.5$, there exists $\xi \geq 0$ such that the PMSD student $y_{\m, \xi}^{(2)}$ achieves $100\%$ population accuracy.
\end{theorem}

When $\rho > 0.5$, incorrectly labeled neighbors outnumber correctly labeled ones. The resulting label averaging prevents the student from correcting misclassified samples and also causes it to misclassify samples that the teacher labeled correctly. In fact, both $y_{\pd}$ and $y_{\pd}^{(2)}$ achieve $0\%$ population accuracy as $n \to \infty$. Nevertheless, the PMSD student $y_{\m, \xi}^{(2)}$ can achieve asymptotic $100\%$ population accuracy. See \Cref{fig:synthetic_constant_large,fig:synthetic_uniform_large} for synthetic illustrations with $\rho > 0.5$.

\begin{theorem}
    \label{thm:class_mixing_large_p}
    Under \Cref{assp:feature_corr}, let $\rho>0.5$ and take the regularization parameter to be a sequence $\lambda_n=\Theta(1)$. Then, as $n \to \infty$, both $y_{\pd}$ and $y_{\pd}^{(2)}$ achieve $0\%$ population accuracy. On the other hand, there exists $\xi > 1$ such that the PMSD student $y_{\m, \xi}^{(2)}$ achieves $100\%$ population accuracy.
\end{theorem}

In this regime, the first PD student $y_{\pd}$ misclassifies both the samples that the original teacher labeled correctly and those it labeled incorrectly. Although the soft predictions of the second student $y_{\pd}^{(2)}$ move in the correct direction relative to those of $y_{\pd}$, the shift is insufficient to alter the induced hard predictions. With an appropriate mixing weight $\xi > 1$, the PMSD student amplifies this corrective signal and extrapolates far enough to change the final predictions.

\begin{table}[!t]
\caption{Test classification accuracy on CIFAR-100 with hierarchical corruption $\rho = 0.2$. The mixing weight is tuned over the grid $[-20, 20]$ for all linear probing experiments.}
\label{tab:cifar100_0.2_hierarchical}
\centering
\resizebox{\textwidth}{!}{
\begin{tabular}{lccccccccc}
\toprule
Regularization level $\lambda$ & $10^{-4}$ & $10^{-3.5}$ & $10^{-3}$ & $10^{-2.5}$ & $10^{-2}$ & $10^{-1.5}$ & $10^{-1}$ & $10^{-0.5}$ & $10^{0}$\\
\midrule \arrayrulecolor{black!50}
Teacher (\%) & $52.8$ & $52.9$ & $52.9$ & $53.2$ & $54.0$ & $56.1$ & $57.7$ & $50.4$ & $35.9$ \\
PD student (\%) & $53.1$ & $53.8$ & $55.0$ & $56.1$ & $57.3$ & $56.8$ & $47.3$ & $20.8$ & $2.3$ \\
\midrule \arrayrulecolor{black}
Optimal PMSD student (\%) & $54.5$ & $54.9$ & $55.6$ & $56.6$ & $57.4$ & $57.3$ & $59.0$ & $54.1$ & $43.1$ \\
Estimated optimal mixing weight $\hat\xi^{\star}$ & $10.5$ & $4.1$ & $2.1$ & $1.3$ & $0.9$ & $0.7$ & $-17.8$ & $-15.0$ & $-20.0$
\\
\bottomrule
\end{tabular}%
}
\end{table}

\subsection{Linear probing experiments on real datasets}

To illustrate the performance of PMSD, we consider multiclass classification with $\ell_2$-regularized cross-entropy loss on Caltech-101, Caltech-256 \citep{griffin2007caltech}, and CIFAR-100 \citep{krizhevsky2009learning}. For each dataset, we train a linear classifier with a softmax output on frozen features from a ResNet-34 pretrained on ImageNet. We use this setup for both the teacher and the pure-distilled student: the teacher is trained on corrupted one-hot labels, whereas the student is trained on the teacher's soft pseudo-labels using the same $\lambda$. \Cref{tab:cifar100_0.2_hierarchical} reports the test accuracies of the teacher, PD student, and PMSD student. We estimate the mixing weight $\xi$ on a small calibration set using a fine grid over $[-20, 20]$. Additional results across datasets and corruption levels appear in \Cref{sec:linear_probing_tables}; experimental details are provided in \Cref{sec:experiment_details}.

\section{Discussion}
\label{sec:extensions}

In this section, we discuss several possible extensions, compare the PMSD student with the same-$X$ loss-mixed SD student, and examine the limitations of our proposed PMSD approach in the prediction-only regime, with a focus on high-dimensional ridge regression.
This comparison clarifies both the advantages and the limitations of prediction-only distillation.

\subsection{Negative regularization}

Our analysis extends readily to negative regularization parameters $\lambda_t$ and $\lambda_s$, provided they are bounded below to prevent the risk from diverging.
Negative regularization can be optimal on real datasets \citep{kobak2020optimal} and has since been studied in the high-dimensional ridge regression literature \citep{wu2020optimal, patil2021uniform, patil2024optimal}.
Under this extension, the strict improvement results in \Cref{sec:strict_improv} remain essentially unchanged.
In particular, \Cref{prop:teacher_bad_set} and \Cref{prop:strict_improv_ood} hold regardless of the signs of $\lambda_t$ and $\lambda_s$: the cardinality of the student-tie set is controlled by the number of \emph{real-valued} solutions to a polynomial equation in $\kappa_s$.
Although the correspondence between $\lambda_s$ and $\kappa_s$ need not be one-to-one under negative regularization, each value of $\kappa_s$ determines at most one value of $\lambda_s$ through \eqref{eq:kappa_fp_mainpaper}.

\subsection{Fresh-X prediction-mixed versus same-X loss-mixed SD}

\begin{figure*}[t]
  \centering
    \begin{subfigure}[t]{0.45\textwidth}
    \centering
    \includegraphics[width=\textwidth]{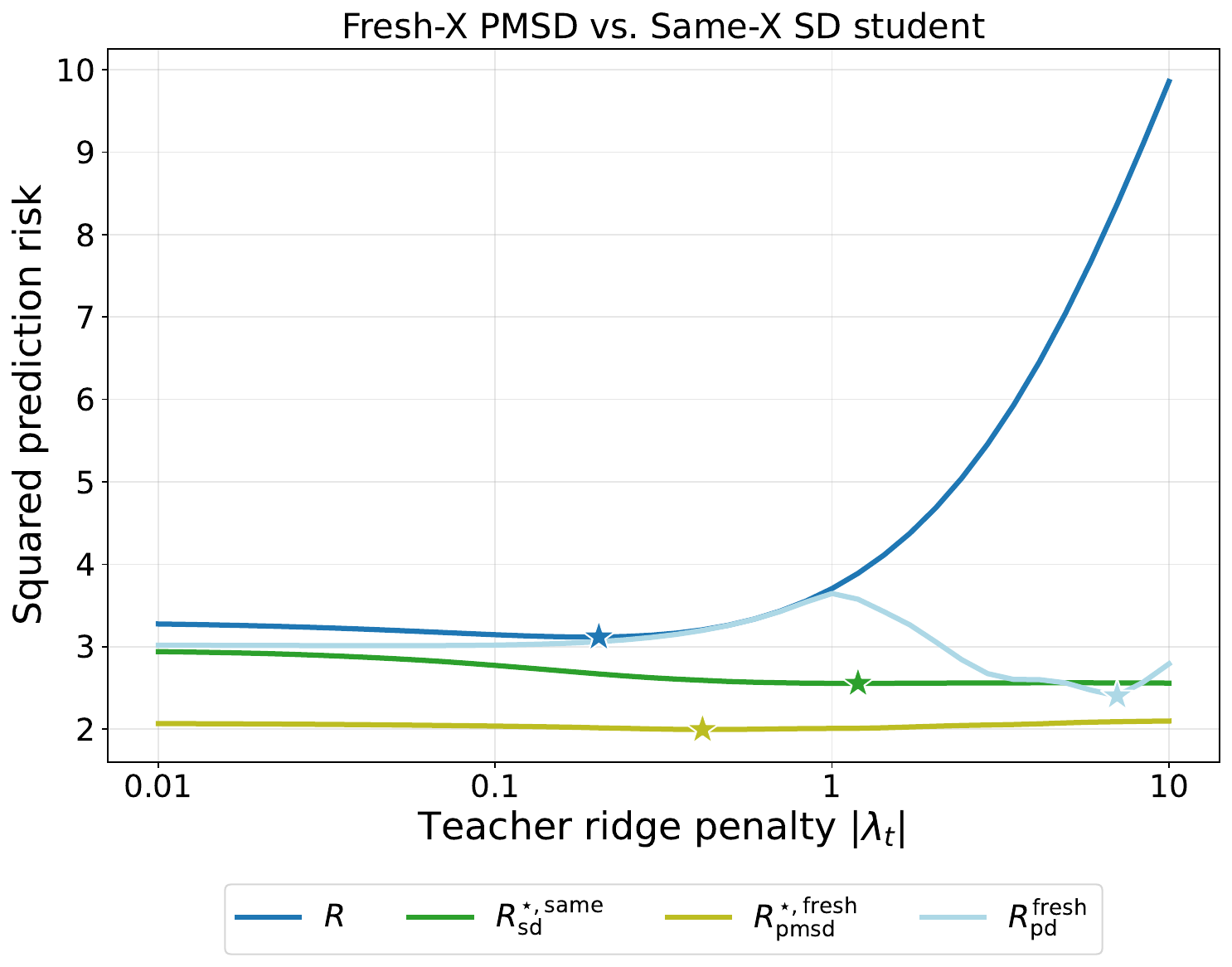}
    \caption{One-spike covariance model.}
     \label{fig:vs_same_x_spike}
  \end{subfigure}
  \hfill
    \begin{subfigure}[t]{0.45\textwidth}
    \centering
    \includegraphics[width=\textwidth]{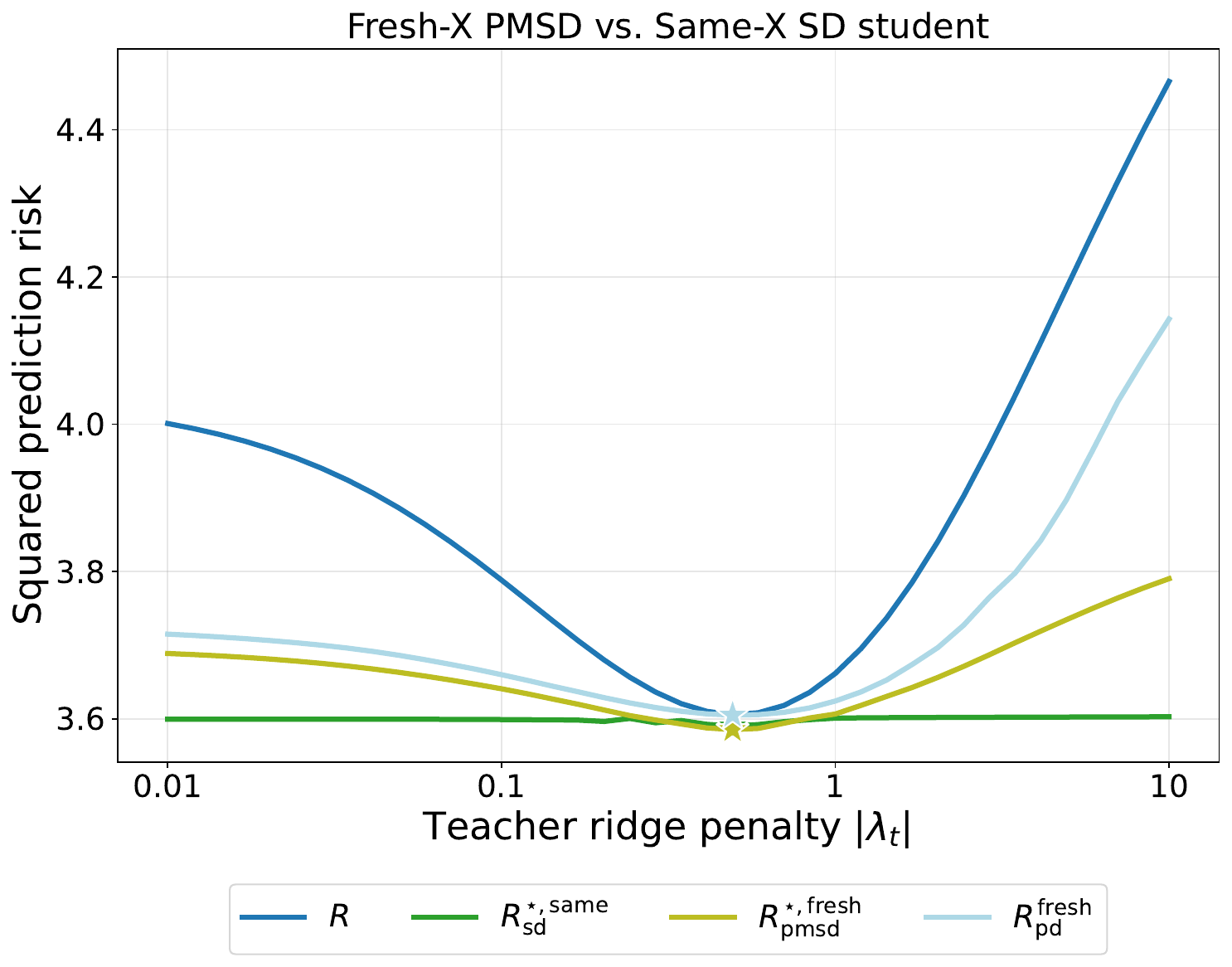}
    \caption{Isotropic covariance model.}
     \label{fig:vs_same_x_iso}
  \end{subfigure}
  \caption{\textbf{Fresh-$X$ PMSD versus same-$X$ loss-mixed SD.} Setting: $p = 400$, $n_t = 200$, $n_s = 4000$, and $\Sigma_t = \Sigma_s$. The $n_s$ fresh unlabeled samples are used only by the fresh-$X$ PD and PMSD students. For each value of $|\lambda_t|$ on the x-axis, we evaluate both signs of the teacher regularization, tune $\lambda_s$ over a grid on $[-100, 100]$ in each case, and report the smaller risk. The signal is $\beta = a u_1 + \sqrt{r^2-a^2}v$, where $a = 1.7$, $r = 2$, $u_1$ is the principal direction of $\Sigma_t$, and $v$ is orthogonal to $u_1$.}
  \label{fig:vs_same_x}
\end{figure*}

Fresh-$X$ PMSD and the same-$X$ loss-mixed SD methods studied by \citet{das2023understanding, pareek2024understanding, dang2026optimal, lecoiu2026self} operate with different information and therefore have different performance limits.
Consider a teacher whose regularization parameter $\lambda_t$ is mistuned for the target risk.
A same-$X$ student can reuse the original labeled sample to correct the teacher, whereas prediction-only PMSD can use only the teacher's predictions on fresh unlabeled covariates.
Consequently, fresh unlabeled data need not recover information lost in the teacher, regardless of the fresh sample size.
\Cref{fig:vs_same_x} illustrates both the benefits and the limitations of PMSD relative to same-$X$ loss mixing.

In the one-spike covariance model, \citet{lecoiu2026self} show that same-$X$ loss-mixed SD, with jointly optimized $\lambda_t$, $\lambda_s$, and $\xi \in \mathbb{R}$ and with negative regularization allowed, is optimal within the spectral shrinkage class given the labeled sample $\mathcal{D}_{\text{lab}}=(X, y)$.
Once additional fresh unlabeled data $\mathcal{D}_{\text{unlab}}$ are available, the fresh-$X$ PMSD student, which uses only the teacher's ridge predictions, can nevertheless outperform this optimal same-$X$ student; see \Cref{fig:vs_same_x_spike}.
The gains are especially pronounced in the overparameterized regime $p>n_t$.

The comparison reverses in the isotropic design setting, where the optimally tuned teacher is optimal within the spectral shrinkage class.
Because the same-$X$ loss-mixed student is itself a spectral shrinkage estimator \citep{lecoiu2026self}, it cannot outperform the optimally tuned teacher.
If the given teacher is mistuned, however, the same-$X$ student can attain the optimal ridge risk by setting $\lambda_s=\lambda_t^\star$ and effectively ignoring the teacher's predictions while using the original training data.
In contrast, PMSD relies only on the mistuned teacher's predictions on fresh unlabeled covariates.
It can substantially improve upon this teacher but generally cannot attain the optimal ridge risk or outperform the optimally tuned teacher; see \Cref{fig:vs_same_x_iso}.
\Cref{fig:vs_same_x_ar1} presents a similar comparison for an AR1 design.

\subsection[Optimal fresh-X covariance structure]{Optimal fresh-$X$ covariance structure}

Our framework also raises a natural design question: given the teacher predictor, from which distribution should the fresh unlabeled covariates be sampled to minimize the PMSD risk?
Within the class of covariances $\Sigma_s$ that are simultaneously diagonalizable with $\Sigma_t$, this amounts to maximizing the teacher-to-PMSD risk reduction in \Cref{thm:general_ridge},
\begin{equation}
\label{eq:risk_reduction}
\mathcal R_{\te}(\lambda_t) - \sR_{\m}^{\star}(\lambda_t, \lambda_s)
= \frac{\bigl(\mathcal R_{\te}(\lambda_t)-\mathcal C(\lambda_t, \lambda_s)\bigr)^2}{\mathcal R_{\te}(\lambda_t)+\mathcal R_{\pd}(\lambda_s) - 2 \mathcal C(\lambda_t, \lambda_s)}.
\end{equation}

Because this quantity also depends on $\lambda_t$ and $\lambda_s$, the optimal choice $\Sigma_s^{\star}$ generally varies with the regularization parameters.
Jointly optimizing $\Sigma_s$ and $\lambda_s$ for a given teacher is considerably more involved, and we leave this problem for future work.

We instead consider the underparameterized regime $\gamma_t, \gamma_s < 1$ and the post-asymptotic ridgeless limit $\lambda_t, \lambda_s \to 0$.
Recall that $\Sigma_t = U^{\top} \operatorname{diag}(\sigma_1, \ldots, \sigma_p) U$, $\Sigma_s = U^{\top} \operatorname{diag}(\tilde \sigma_1, \ldots, \tilde \sigma_p) U$, and $v_i = ((U\beta)_i)^2$.
For each $i \in [p]$, define $w_i = \tilde \sigma_i^{-1}$ and let $w = (w_1, \ldots, w_p)^{\top}$.
The limiting risk reduction in \eqref{eq:risk_reduction} is invariant under the rescaling $\Sigma_s\mapsto c\Sigma_s$ for $c > 0$, so we impose the normalization $p^{-1}\sum_{i=1}^p w_i=1$.
Define
\[
s_i := \sigma_i v_i + \frac{\sigma^2 \gamma_t}{p(1 - \gamma_t)},
\quad
t_i := \frac{\gamma_s}{2p(1-\gamma_s)} \biggl( v_i + \frac{\sigma^2 \gamma_t}{p \sigma_i (1 - \gamma_t)} \biggr), \quad i = 1, \ldots, p,
\]
and
\[
S = \operatorname{diag}(s_1, \ldots, s_p), 
\quad t = (t_1, \ldots, t_p)^{\top},
\quad \theta = (\sigma_1, \ldots, \sigma_p)^{\top}.
\]

\begin{proposition}
\label{prop:optimize_Sigma_s}
    Under \Cref{def:dist}, further assume that $\gamma_t, \gamma_s < 1$ and $\sigma^2>0$, and consider the ridgeless limit $\lambda_t, \lambda_s \to 0$.
    Then maximizing the risk reduction in \eqref{eq:risk_reduction} over the spectrum $\{ \tilde \sigma_1, \ldots, \tilde \sigma_p \}$ of the positive definite matrix $\Sigma_s$ is equivalent, up to positive rescaling, to the optimization problem
    \begin{align}
    \underset{w\in\mathbb R_{+}^p}{\operatorname{minimize}}
    \quad&
    w^{\top}
    \bigl(
    S+t\theta^{\top}+\theta t^{\top}
    \bigr)w,\\
    \operatorname{subject\ to}
    \quad&
    \mathbf 1_p^{\top}w=p, \nonumber
    \end{align}
    where $w_i=\tilde \sigma_i^{-1}$ for $i \in [p]$.
\end{proposition}

This is a linearly constrained quadratic program in $w$ and can be solved numerically using standard quadratic programming methods.
In general, the spectrum of the optimal $\Sigma_s^\star$ depends on the teacher covariance, the signal-covariance alignment, the noise level, and the data aspect ratios.
In particular, matching the teacher covariance need not be optimal: $\Sigma_s^\star$ need not be proportional to $\Sigma_t$.
This differs from a recent result on optimal \emph{labeled}-data selection for ridgeless estimators \citep{rezaei2025high}, reflecting the different role of the fresh sample in our prediction-only setting.

Together with our calibration and logistic regression results, these observations outline the scope of PMSD: prediction mixing can exploit fresh unlabeled data to improve a fixed teacher, and a small labeled sample can make this improvement operational, but PMSD cannot in general recover the information in the original labeled sample.
Optimizing the fresh covariate distribution beyond the simultaneously diagonalizable ridgeless setting is a natural direction for future work.

\section*{Acknowledgments}
We thank Sujay Sanghavi, Adel Javanmard, Peter M\"{u}eller, Antonio Linero, Edward George, and participants in the Best of Statistical Science (``BOSS'') Workshop for helpful discussions.

\bibliographystyle{plainnat}
\bibliography{references-v2}

\clearpage
\appendix

\begin{center}
\Large
{\bf \framebox{Supplement}}
\end{center}

\bigskip

This supplement accompanies the paper ``Prediction-Only Distillation in Linear and Logistic Regression''.
It contains additional details, illustrations, and proofs of all the theoretical results.
The supplement is organized as follows.

\addcontentsline{toc}{section}{Appendix} 
\startcontents
\printcontents{}{1}{\setcounter{tocdepth}{2}}

\clearpage
\section{Notation}
\label{sec:notation}

\textbf{General.}
We denote vectors by lowercase letters and matrices by uppercase letters.
We use blackboard letters for standard number systems: $\NN$ denotes the positive integers, $\RR$ the real numbers, and $\overline{\RR}$ the extended real line.

For a positive integer $n$, we write $[n]:=\{1,\ldots,n\}$.
For $x,y\in\RR$, we write $x \wedge y:=\min\{x,y\}$ and $x \vee y:=\max\{x,y\}$.
For an event or set $A$, $\ind_A$ denotes its indicator.

\textbf{Linear algebra.}
For a vector $x$, $\|x\|_2$ denotes its Euclidean norm; when the subscript is omitted for simplicity, $\|x\|$ has the same meaning.
For a matrix $X \in \RR^{n \times p}$, $X^\top \in \RR^{p \times n}$ denotes its transpose and $X^{\dagger} \in \RR^{p \times n}$ its Moore--Penrose pseudoinverse.
For a square matrix $A \in \RR^{p \times p}$, $\tr(A)$ denotes its trace and $\otr(A):=\tr(A)/p$ its normalized trace.
If $A$ is invertible, $A^{-1}$ denotes its inverse.
For a symmetric positive semidefinite matrix $\Sigma$, $\Sigma^{1/2}$ denotes its principal square root.
The $p \times p$ identity matrix is denoted by $I_p$, or simply by $I$ when its dimension is clear.
For a matrix $X$, $\|X\|_{\mathrm{op}}$ denotes the operator norm induced by the Euclidean vector norm, and $\|X\|_F$ denotes the Frobenius norm.
For a matrix $M$, $\|M\|_{\tr}$ denotes its trace norm, equal to the sum of its singular values.

\textbf{Asymptotics.}
For $Y\geq 0$, we write $X=O_\alpha(Y)$ if $|X|\leq C_\alpha Y$, where $C_\alpha$ may depend on the ambient parameter $\alpha$ but not on the other parameters under consideration.
We use $O_p$ for the corresponding probabilistic order notation.
The notation $\Theta(1)$ denotes a quantity whose absolute value is bounded above and away from zero by constants independent of the asymptotic index.
We write $X_n\pto X$ for convergence in probability and $X_n\asto X$ for almost sure convergence.
When the mode of convergence is clear from context, we write $X_n \to Y_n$ to mean $|X_n-Y_n|\to 0$ in that mode.
We use $C$ and $C'$ for generic positive absolute constants whose values may change from line to line.

\vspace{-0.5em}
\section{Further related work}
\label{sec:further_rw}

\textbf{Retraining on synthetic labels and model collapse.}
Beyond a single round of pure-distillation, recent work studies how recursively training on model-generated data can degrade performance, a phenomenon known as \emph{model collapse} \citep{shumailov2024ai, alemohammad2023self, dohmatob2024model, gerstgrasser2024model, schaeffer2025position, dohmatob2024strong}.
\citet{he2025golden} and \citet{garg2025preventing} study infinitely repeated schemes for ridge and ridgeless regression that mix ground-truth and synthetic labels with a fixed weight $w \in [0,1]$ to prevent such degradation. \citet{bakshi2026collapse} studies an extension of this iterative training framework to the next-token prediction setting, using a mixture of real and synthetic data at the population level.
\citet{dohmatob2024strong} use the term ``strong model collapse'' when a model trained on synthetic data has worse population risk than one trained only on ground-truth data.
Our setting differs in two respects: we study a fixed, finite distillation procedure using fresh unlabeled covariates, and the final predictor mixes predictions rather than recursively replacing ground-truth data with synthetic data.
Accordingly, our results concern the benefits of post hoc prediction mixing, rather than the long-run stability of repeated synthetic training; see \Cref{sec:strict_nonmonotonicity,sec:logistic}.

\textbf{Risk estimation and tuning.}
Our calibration analysis in \Cref{sec:tuning_ridge} connects to the literature on out-of-sample risk estimation for high-dimensional ridge regression and related models, including analyses based on leave-one-out and generalized cross-validation \citep{rad2020scalable, patil2021uniform, patil2022estimating, wei2022more}.
Risk estimation has also been studied for finite ensembles of penalized estimators and subagged regularized $M$-estimators \citep{bellec2025corrected, koriyama2024precise}.
Unlike those settings, PMSD combines two predictors linked through pseudo-labeling, and our goal is to estimate their residual correlation and optimal mixing weight using a small independent labeled calibration sample, without retraining either model.

\enlargethispage{\baselineskip}

\clearpage

\section{Proofs in Section~\ref{sec:ridge}}
\subsection{Preliminaries}

Define the empirical covariance matrices for the teacher's training covariates and the student's fresh covariates, together with their resolvents, by
\begin{align}
\widehat\Sigma:=\frac{1}{n_t} X^\top X,\qquad \widetilde\Sigma:=\frac{1}{n_s} \tilde X^\top\tilde X,
\qquad
Q_{\lambda_t}:=(\widehat\Sigma+\lambda_t I_p)^{-1},\qquad
\widetilde Q_{\lambda_s}:=(\widetilde\Sigma+\lambda_s I_p)^{-1}. 
\nonumber
\end{align}

The teacher coefficient is $\beta_{\lambda_t}:=Q_{\lambda_t} X^\top y/n_t$.
The \emph{fresh-$X$ pure-distilled} (PD) coefficient is
\[
\beta_{\pd,\lambda_s}
:=\widetilde Q_{\lambda_s} \,\widetilde\Sigma\,\beta_{\lambda_t}
=\big(I_p - \lambda_s \widetilde Q_{\lambda_s} \big)\beta_{\lambda_t}.
\]
The \emph{prediction-mixed} student family is
\[
f_{\m,\lambda_t, \lambda_s, \xi}(x)=(1-\xi)f_{\lambda_t}(x)+\xi f_{\pd,\lambda_s}(x),
\qquad
f_{\lambda_t}(x)=x^\top\beta_{\lambda_t},\ \ f_{\pd,\lambda_s}(x)=x^\top\beta_{\pd,\lambda_s}.
\]
Recall the in-distribution risks and correlation
\begin{align}
  R_{\te}(\lambda_t) &:=\E[(y_0-f_{\lambda_t}(x_0))^2\mid \cD_{\rm lab}],
  \qquad
  R_{\pd}(\lambda_s):=\E[(y_0-f_{\pd,\lambda_s}(x_0))^2 \mid \cD_{\rm lab}, \cD_{\rm unlab}],
  \nonumber \\
  C(\lambda_t, \lambda_s)&:=\E[(y_0 - f_{\lambda_t}(x_0)) (y_0 -f_{\pd,\lambda_s}(x_0))\mid\cD_{\rm lab}, \cD_{\rm unlab}], \nonumber
  \\
  R_{\m}(\lambda_t, \lambda_s, \xi) &= 
  \E[(y_0-f_{\m,\lambda_t, \lambda_s, \xi}(x_0))^2 \mid \cD_{\rm lab}, \cD_{\rm unlab}],
  \nonumber
\end{align}
Denote the deterministic equivalents of the first three quantities by $\sR(\lambda_t)$, $\sR_{\pd}(\lambda_s)$, and $\sC(\lambda_t, \lambda_s)$, respectively.
Also define $D(\lambda_t, \lambda_s) := R_{\te}(\lambda_t) + R_{\pd}(\lambda_s) - 2C(\lambda_t, \lambda_s) \geq 0$.

Recall the random-design model in \Cref{def:dist}.
The population $L_2$-projection coefficient, residual, and residual variance are
\[
  \beta := \E[xx^\top]^{-1}\E[xy] = \Sigma_t^{-1}\E[xy],
  \qquad
  \varepsilon := y - x^\top\beta,
  \qquad
  \sigma^2 := \E[\varepsilon^2].
\]
Then $\E[x\varepsilon]=0$ (equivalently, $\E[\bz\,\varepsilon]=0$) and $\E[\varepsilon]=0$.
Throughout, we consider in-distribution prediction: the test pair $(x_0,y_0)$ is an independent copy of $(x,y)$, so $y_0=x_0^\top\beta+\varepsilon_0$, with $\E[x_0\varepsilon_0]=0$ and $\E[\varepsilon_0^2]=\sigma^2$.

\subsection{Proof of \Cref{prop:oracle_affine}}

\begin{proof}[Proof of \Cref{prop:oracle_affine}]
Let $e_{\te}:=y_0-f_{\lambda_t}^{\te}(x_0)$ and $e_{\pd}:=y_0-f_{\pd,\lambda_s}(x_0)$.
Since $y_0 - f_{\m, \lambda_t, \lambda_s, \xi}(x_0) = (1-\xi)e_{\te} + \xi e_{\pd}$, we have
\begin{align}
   R_{\m}(\lambda_t, \lambda_s, \xi)
  &=
  (1-\xi)^2R_{\te}(\lambda_t)+\xi^2R_{\pd}(\lambda_s)+2\xi(1-\xi)C(\lambda_t, \lambda_s)
  \nonumber \\
  &=
  R_{\te}(\lambda_t) - 2 \xi (R_{\te}(\lambda_t)-C(\lambda_t,\lambda_s)) + \xi^2 D(\lambda_t, \lambda_s). 
\end{align}
If $D(\lambda_t, \lambda_s)>0$, this is a strictly convex quadratic in $\xi$ with unique minimizer
$\xi^\star(\lambda_t, \lambda_s)=(R_{\te}(\lambda_t)-C(\lambda_t, \lambda_s))/D(\lambda_t, \lambda_s)$.
Substitution gives
\begin{equation}
    R_{\m}^{\star}(\lambda_t, \lambda_s) = R_{\te}(\lambda_t) -\frac{(R_{\te}(\lambda_t) - C(\lambda_t, \lambda_s))^2}{D(\lambda_t, \lambda_s)}.
\end{equation}
Equivalently, expanding the same quadratic around $\xi=1$ gives
\[
R_{\m}^{\star}(\lambda_t, \lambda_s)
=R_{\pd}(\lambda_s)-\frac{(R_{\pd}(\lambda_s)-C(\lambda_t,\lambda_s))^2}{D(\lambda_t,\lambda_s)},
\]
which proves both identities in \Cref{prop:oracle_affine}.
\end{proof}

\subsection{Deterministic equivalents for the proof of \Cref{thm:general_ridge}}

Fix $\lambda_t > 0$.
Let $\kappa_t=\kappa_t(\lambda_t)>0$ denote the teacher fixed-point solution
\begin{equation}
  \kappa_t
  =
  \lambda_t +\gamma_t \kappa_t \,\otr\!\big(\Sigma_t (\Sigma_t +\kappa_t I_p)^{-1}\big),
\end{equation}
and define
\[
  G_t:=(\Sigma_t +\kappa_t I_p)^{-1}.
\]
Likewise, let $\kappa_s=\kappa_s(\lambda_s)>0$ denote the PD student fixed-point solution
\begin{equation}
  \kappa_s
  =
  \lambda_s + \gamma_s \kappa_s\,\otr\!\big(\Sigma_s (\Sigma_s +\kappa_s I_p)^{-1}\big),
\end{equation}
and define
\[
  G_s:=(\Sigma_s +\kappa_s I_p)^{-1}.
\]

For $i, j \geq 0$, define the mixed signal-covariance alignment and trace functionals by
\begin{equation}
q_{i,j}:=\beta^\top G_t^i G_s^j \Sigma_t \,\beta,
    \qquad 
    \tilde q_{i,j}:=\beta^\top G_t^i G_s^j \Sigma_s \,\beta,
  \qquad
  t_{i,j}:=\gamma_t \,\otr(\Sigma_t^2 G_t^i G_s^j),
\end{equation}
and set 
\[
  b_t := (1 - t_{2,0})^{-1}, 
  \quad
 b_s := \frac{\gamma_s \otr(\Sigma_s \Sigma_t G_s^2) } {1 - \gamma_s \otr(\Sigma_s^2 G_s^2) }.
\]

\begin{lemma}[Deterministic equivalents with a deterministic insertion]
\label{lem:fresh_insertion_DE}
Let $A \in\RR^{p\times p}$ be deterministic, symmetric, uniformly bounded in operator norm, and commuting with $\Sigma_t$.
Define
\[
  \tau_A:=\gamma_t \,\otr(\Sigma_t G_t A G_t).
\]
Then, for each fixed $\lambda_t>0$,
\begin{align}
\label{eq:QA_DE}
  \lambda_t Q_{\lambda_t} A
  &\asymp
  \kappa_t \, G_t A,
  \\
\label{eq:QAQ_DE}
  \lambda_t^2 Q_{\lambda_t} A Q_{\lambda_t}
  &\asymp
  \kappa_t^2 G_t A G_t
  +
  \kappa_t^2 b_t\,\tau_A\, G_t \Sigma_t G_t,
\end{align}
and
\begin{equation}
\frac{1}{n_t} \tr(A\widehat\Sigma Q_{\lambda_t}^2)
  \pto
  b_t\,\tau_A,
\end{equation}
where $b_t := (1 - t_{2,0})^{-1}$ and $t_{2,0} = \gamma_t \otr(\Sigma_t^2 G_t^2)$.
\end{lemma}

\begin{proof}[Proof of \Cref{lem:fresh_insertion_DE}]
The first display \eqref{eq:QA_DE} follows immediately from the scaled resolvent deterministic equivalent
\[
  \lambda_t Q_{\lambda_t} \asymp \kappa_t G_t
\]
by right-multiplying with the bounded deterministic matrix $A$.

For \eqref{eq:QAQ_DE}, we repeat the $2\times 2$ block linearization used in Lemma B.9 of \citet{dang2026optimal}, with $A$ as the off-diagonal coupling matrix.
Concretely, define
\[
  \mathbb H(J)
  :=
  \begin{pmatrix}
    \widehat\Sigma+\lambda_t I_p & J A\\
    J A & \widehat\Sigma+\lambda_t I_p
  \end{pmatrix},
  \qquad
  \mathbb G(J):=\mathbb H(J)^{-1}.
\]
Differentiating at $J=0$ gives
\[
  -(\mathbb G'(0))_{12}=Q_{\lambda_t} A Q_{\lambda_t}.
\]
On the deterministic-equivalent side, the same linearization yields a self-consistent equation for
the off-diagonal block whose solution is
\[
  \lambda_t^2 Q_{\lambda_t} A Q_{\lambda_t}
  \asymp
  \kappa_t^2 G_t A G_t
  +
  \kappa_t^2 b_t\,\tau_A\, G_t \Sigma_t G_t,
\]
where the scalar correction factor is
\[
  b_t=(1-t_{2,0})^{-1},
  \qquad
  t_{2,0} =\gamma_t \,\otr(\Sigma_t^2 G_t^2).
\]

Finally, using the exact identity
\[
  \widehat\Sigma Q_{\lambda_t}^2 = Q_{\lambda_t} - \lambda_t Q_{\lambda_t}^2,
\]
we have
\[
  \frac{1}{n_t} \tr(A\widehat\Sigma Q_{\lambda_t}^2)
  =
  \gamma_t \,\otr(A Q_{\lambda_t})-\lambda_t \gamma_t \,\otr(A Q_{\lambda_t}^2).
\]
Let
\[
  s_A(\lambda_t):=\gamma_t \,\otr(A Q_{\lambda_t}).
\]
By \eqref{eq:QA_DE},
\[
  s_A(\lambda_t)\pto \gamma_t\,\frac{\kappa_t}{\lambda_t}\,\otr(A G_t).
\]
Differentiating this equivalent and using $\kappa_t'(\lambda_t)=b_t$ and $G_t'(\lambda_t)=-\kappa_t'(\lambda_t)G_t^2=-b_tG_t^2$ gives
\[
  \frac{1}{n_t} \tr(A\widehat\Sigma Q_{\lambda_t}^2)
  =
  s_A(\lambda_t)+\lambda_ts_A'(\lambda_t)
  \pto
  \gamma_t b_t\,\otr\!\big(A G_t (I_p-\kappa_t G_t)\big).
\]
Since $I_p-\kappa_t G_t=\Sigma_t G_t$, and $A$ commutes with $\Sigma_t$ and $G_t$, this becomes
\[
  \gamma_t b_t\,\otr(A G_t \Sigma_t G_t)
  =
  \gamma_t b_t\,\otr(\Sigma_t G_t A G_t)
  =
  b_t\,\tau_A.
\]
\end{proof}

\subsection{Proof of \Cref{thm:general_ridge}}

\begin{proof}[Proof of \Cref{thm:general_ridge}]
    By Theorem 3.1 of \citet{dang2026optimal}, the teacher risk has the deterministic equivalent
    \begin{align}
        \label{eq:teacher_risk}
        \sR(\lambda_t) = \kappa_t^2 b_t q_{2,0} + \sigma^2 b_t t_{2,0} + \sigma^2.
    \end{align}
    
    The next two lemmas give
    \begin{align}
        R_{\te}(\lambda_t) - C(\lambda_t, \lambda_s)
      &\pto \sR_{\te}(\lambda_t) - \sC(\lambda_t, \lambda_s) =
       \kappa_s\big(
        -\kappa_t q_{1,1}
        +\kappa_t^2 q_{2,1}
        +\kappa_t^2 b_t t_{2,1} q_{2,0}
        +\sigma^2 b_t t_{2,1}
      \big), \nonumber \\
        D(\lambda_t, \lambda_s)
  &\pto
   \sD(\lambda_t, \lambda_s)
  =
 \kappa_s^2
 \big(
 q_{0,2} + b_s \tilde q_{0,2} - 2 \kappa_t (q_{1,2} + b_s \tilde q_{1,2}) + \kappa_t^2 
 (q_{2,2} + b_s \tilde q_{2,2})
 \nonumber \\
 &\hspace{9em}
 + b_t (\kappa_t^2 q_{2,0} + \sigma^2) \big(t_{2,2} + b_s \gamma_t \otr(\Sigma_t \Sigma_s G_t^2 G_s^2) \big)
 \big). \nonumber
    \end{align}

    Since $\sD(\lambda_t, \lambda_s) = \sR_{\te}(\lambda_t) + \sR_{\pd}(\lambda_s) - 2 \sC(\lambda_t, \lambda_s)$, these identities yield closed-form expressions for $\sR_{\pd}(\lambda_s)$ and $\sC(\lambda_t, \lambda_s)$:
    \begin{align}
    \sC(\lambda_t, \lambda_s) &= \sigma^2 + \kappa_t^2 b_t q_{2,0} + \sigma^2 b_t t_{2,0} - \kappa_s (
        -\kappa_t q_{1,1}
        +\kappa_t^2 q_{2,1}
        +\kappa_t^2 b_t t_{2,1} q_{2,0}
        +\sigma^2 b_t t_{2,1}
    ),
    \\
    \sR_{\pd}(\lambda_s) &=  \kappa_s^2
 \big(
 q_{0,2} + b_s \tilde q_{0,2} - 2 \kappa_t (q_{1,2} + b_s \tilde q_{1,2}) + \kappa_t^2 
 (q_{2,2} + b_s \tilde q_{2,2})
 \nonumber \\
 &\hspace{3em}
 + b_t (\kappa_t^2 q_{2,0} + \sigma^2) \big(t_{2,2} + b_s \gamma_t \otr(\Sigma_t \Sigma_s G_t^2 G_s^2) \big)
 \big) -
 \sR(\lambda_t) + 2\sC(\lambda_t, \lambda_s). 
    \end{align}

In the nontrivial case $r^2+\sigma^2>0$, we have the scalars $\kappa_t, \kappa_s, b_t, b_s, q_{2,0}, t_{2,2} > 0$, and
\begin{align*}
q_{0,2} - 2 \kappa_t q_{1,2} + \kappa_t^2 q_{2,2}
&= \sum_{i=1}^p
\frac{\sigma_i v_i}{(\tilde{\sigma}_i+\kappa_s)^2}
\left(
1-\frac{2\kappa_t}{\sigma_i+\kappa_t}
+\frac{\kappa_t^2}{(\sigma_i+\kappa_t)^2}
\right) \nonumber \\
&=
\sum_{i=1}^p
\frac{\sigma_i v_i}{(\tilde{\sigma}_i+\kappa_s)^2}
\left(
\frac{\sigma_i}{\sigma_i+\kappa_t}
\right)^2 \nonumber \\
&=
\sum_{i=1}^p
\frac{\sigma_i^3 v_i}
{(\sigma_i+\kappa_t)^2(\tilde{\sigma}_i+\kappa_s)^2}
\geq 0,
\\
\tilde q_{0,2} - 2 \kappa_t \tilde q_{1,2} + \kappa_t^2 \tilde q_{2,2}
&= \sum_{i=1}^p
\frac{\tilde{\sigma}_i v_i}
     {(\tilde{\sigma}_i+\kappa_s)^2}
\left(
1-\frac{2\kappa_t}{\sigma_i+\kappa_t}
+\frac{\kappa_t^2}{(\sigma_i+\kappa_t)^2}
\right)\\
&=
\sum_{i=1}^p
\frac{\tilde{\sigma}_i v_i}
     {(\tilde{\sigma}_i+\kappa_s)^2}
\left(
\frac{\sigma_i}{\sigma_i+\kappa_t}
\right)^2\\
&=
\sum_{i=1}^p
\frac{\sigma_i^2\tilde{\sigma}_i v_i}
     {(\sigma_i+\kappa_t)^2
      (\tilde{\sigma}_i+\kappa_s)^2}
\geq 0.
\end{align*}

Then, the limiting denominator $\sD(\lambda_t, \lambda_s) = \sR_{\te}(\lambda_t) + \sR_{\pd}(\lambda_s) - 2 \sC(\lambda_t, \lambda_s)$ is strictly larger than 0 since
\begin{align}
    &\kappa_s^2
 \Big(
 q_{0,2} + b_s \tilde q_{0,2} - 2 \kappa_t (q_{1,2} + b_s \tilde q_{1,2}) + \kappa_t^2 
 (q_{2,2} + b_s \tilde q_{2,2})
 + b_t (\kappa_t^2 q_{2,0} + \sigma^2) \big(t_{2,2} + b_s \gamma_t \otr(\Sigma_t \Sigma_s G_t^2 G_s^2) \big)
 \Big) \nonumber \\
 &= \underbrace{\kappa_s^2}_{> 0}
 \Big( \underbrace{q_{0,2} - 2 \kappa_t q_{1,2} + \kappa_t^2 q_{2,2}}_{\geq 0} 
 + \underbrace{b_s(\tilde q_{0,2} - 2 \kappa_t \tilde q_{1,2} + \kappa_t^2 \tilde q_{2,2})}_{\geq 0}
 + \underbrace{ b_t (\kappa_t^2 q_{2,0} + \sigma^2) \big(t_{2,2} + b_s \gamma_t \otr(\Sigma_t \Sigma_s G_t^2 G_s^2) \big)}_{> 0}
 \Big).
 \nonumber
\end{align}

Substituting these deterministic equivalents into the oracle identities in \Cref{prop:oracle_affine} gives
\begin{align}
    \xi^{\star}(\lambda_t, \lambda_s)&\pto
\frac{\mathcal R_{\te}(\lambda_t)-\mathcal C(\lambda_t, \lambda_s)}{\mathcal R_{\te}(\lambda_t)+\mathcal R_{\pd}(\lambda_s) - 2 \mathcal C(\lambda_t, \lambda_s)},
\nonumber \\
R_{\m}^{\star}(\lambda_t, \lambda_s) &\pto
\mathcal R_{\te}(\lambda_t)
-
\frac{\bigl(\mathcal R_{\te}(\lambda_t)-\mathcal C(\lambda_t, \lambda_s)\bigr)^2}{\mathcal R_{\te}(\lambda_t)+\mathcal R_{\pd}(\lambda_s) - 2 \mathcal C(\lambda_t, \lambda_s)}. \nonumber
\end{align}
\end{proof}

\begin{lemma}[Deterministic equivalent for $R_{\te}-C$]
\label{lem:fresh_num_DE_general}
Under \Cref{def:dist}, as $n_t, n_s, p \to\infty$ with
\[
  0 < \liminf \gamma_t \leq \limsup \gamma_t < \infty,
  \qquad
  0 < \liminf \gamma_s \leq \limsup \gamma_s < \infty,
\]
for each fixed $\lambda_t>0$ and $\lambda_s>0$,
\begin{equation}
\label{eq:RminusCfresh_DE_final}
  R_{\te}(\lambda_t)- C(\lambda_t, \lambda_s)
  \pto
   \kappa_s\big(
    -\kappa_t q_{1,1}
    +\kappa_t^2 q_{2,1}
    +\kappa_t^2 b_t t_{2,1} q_{2,0}
    +\sigma^2 b_t t_{2,1}
  \big).
\end{equation}
\end{lemma}

\begin{proof}[Proof of \Cref{lem:fresh_num_DE_general}]
    Recall that
\[
  \beta_{\lambda_t} := Q_{\lambda_t} X^\top y/n_t,
  \qquad
  \beta_{\pd,\lambda_s}
  :=
  \widetilde Q_{\lambda_s} \widetilde\Sigma\,\beta_{\lambda_t}
  =
  (I_p-\lambda_s \widetilde Q_{\lambda_s})\beta_{\lambda_t},
\]

Therefore,
\begin{equation}
  R_{\te}(\lambda_t)-C(\lambda_t, \lambda_s)
  = (\beta_{\lambda_t} -\beta)^{\top}  \Sigma_t (\beta_{\lambda_t} - \beta_{\pd, \lambda_s})
  =
  \lambda_s \,(\beta_{\lambda_t} -\beta)^\top \Sigma_t \widetilde Q_{\lambda_s} \beta_{\lambda_t}.
\end{equation}

By the fresh sample resolvent deterministic equivalent,
\[
  \lambda_s \widetilde Q_{\lambda_s} \asymp \kappa_s G_s,
\]
hence
\begin{equation}
  R_{\te}(\lambda_t)-C(\lambda_t, \lambda_s)
  =
  \kappa_s\,(\beta_{\lambda_t}-\beta)^\top \Sigma_t G_s\,\beta_{\lambda_t}
  +o_p(1).
\end{equation}

Now expand
\[
  \beta_{\lambda_t} -\beta = -\lambda_t Q_{\lambda_t} \beta + Q_{\lambda_t} X^\top\varepsilon/n_t,
  \qquad
  \beta_{\lambda_t} = (I_p - \lambda_t Q_{\lambda_t})\beta + Q_{\lambda_t} X^\top\varepsilon/n_t.
\]
With $A_s:=\Sigma_t G_s$, this gives
\begin{align}
(\beta_{\lambda_t} -\beta)^\top A_s \beta_{\lambda_t}
&=
-\lambda_t \,\beta^\top Q_{\lambda_t} A_s \beta
+\lambda_t^2 \beta^\top Q_{\lambda_t} A_s Q_{\lambda_t} \beta
+\frac{1}{n_t^2}\varepsilon^\top X Q_{\lambda_t} A_s Q_{\lambda_t} X^\top \varepsilon
+\Lin(\lambda_t),
\end{align}
where $\Lin(\lambda_t)$ is linear in $\varepsilon$.
By the same leave-one-out concentration argument used in
Lemma B.3 of \citet{dang2026optimal}, we have $\Lin(\lambda_t)=o_p(1)$.

Apply \Cref{lem:fresh_insertion_DE} with $A_s=\Sigma_t G_s$.
Because $A_s$ commutes with $\Sigma_t$,
\[
  \tau_{A_s}
  =
  \gamma_t\,\otr(\Sigma_t^2 G_t^2 G_s)
  =
  t_{2,1}.
\]
Also,
\[
  \beta^\top G_t A_s \beta
  =
  \beta^\top G_t \Sigma_t G_s \beta
  =
  q_{1,1},
  \qquad
  \beta^\top G_t A_s G_t \beta
  =
  \beta^\top G_t^2 \Sigma_t G_s \beta
  =
  q_{2,1}.
\]
Therefore,
\[
  -\lambda_t \,\beta^\top Q_{\lambda_t} A_s \beta
  \pto
  -\kappa_t \,q_{1,1},
\]
\[
  \lambda_t^2 \beta^\top Q_{\lambda_t} A_s Q_{\lambda_t} \beta
  \pto
  \kappa_t^2 q_{2,1}
  +
  \kappa_t^2 b_t t_{2,1}\,q_{2,0},
\]
and
\[
  \frac{1}{n_t^2}\varepsilon^\top X Q_{\lambda_t} A_s Q_{\lambda_t} X^\top \varepsilon
  \pto
  \sigma^2\cdot
  \frac{1}{n_t}\tr(A_s\widehat\Sigma Q_{\lambda_t}^2)
  \pto
  \sigma^2 b_t t_{2,1}.
\]
Thus,
\[
  \sR_{\te}(\lambda_t) - \sC(\lambda_t, \lambda_s)
  =
  \kappa_s\big(
    -\kappa_t q_{1,1}
    +\kappa_t^2 q_{2,1}
    +\kappa_t^2 b_t t_{2,1} q_{2,0}
    +\sigma^2 b_t t_{2,1}
  \big).
\]
\end{proof}

\begin{lemma}[Deterministic equivalent for $D$]
\label{lem:freshX_DE_delta_D_Rpd}
Under \Cref{def:dist}, as $n_t, n_s, p \to\infty$ with
\[
  0 < \liminf \gamma_t \leq \limsup \gamma_t < \infty,
  \qquad
  0 < \liminf \gamma_s \leq \limsup \gamma_s < \infty,
\]
for each fixed $\lambda_t>0$ and $\lambda_s>0$,
\begin{align}
  D(\lambda_t, \lambda_s)
  \pto
   \sD(\lambda_t, \lambda_s)
  &=
 \kappa_s^2
 \big(
 q_{0,2} + b_s \tilde q_{0,2} - 2 \kappa_t (q_{1,2} + b_s \tilde q_{1,2}) + \kappa_t^2 
 (q_{2,2} + b_s \tilde q_{2,2})
 \nonumber \\
 &\hspace{3em}
 + b_t (\kappa_t^2 q_{2,0} + \sigma^2) \big(t_{2,2} + b_s \gamma_t \otr(\Sigma_t \Sigma_s G_t^2 G_s^2) \big)
 \big),
 \label{eq:fresh_D_DE}
\end{align}
where $b_s = \big(\gamma_s \otr(\Sigma_s \Sigma_t G_s^2) \big) / \big(1 - \gamma_s \otr(\Sigma_s^2 G_s^2) \big)$.
\end{lemma}

\begin{proof}[Proof of \Cref{lem:freshX_DE_delta_D_Rpd}]
We first note the exact identities
\[
  \beta_{\pd,\lambda_s}
  =
  (I_p - \lambda_s \widetilde Q_{\lambda_s}) \beta_{\lambda_t},
  \qquad
  \beta_{\lambda_t} -\beta_{\pd,\lambda_s}
  =
  \lambda_s \widetilde Q_{\lambda_s} \beta_{\lambda_t},
\]
where $\widetilde Q_{\lambda_s}:=(\widetilde\Sigma+\lambda_s I_p)^{-1}$ and
$\widetilde\Sigma=\widetilde X^\top \widetilde X/n_s$.
Writing $e_{\lambda_t}:=\beta_{\lambda_t} -\beta$, we have
\[
  D(\lambda_t, \lambda_s)
  =
  (\beta_{\lambda_t} -\beta_{\pd,\lambda_s})^\top \Sigma_t (\beta_{\lambda_t} -\beta_{\pd,\lambda_s})
  =
  \lambda_s^2\,\beta_{\lambda_t}^\top \widetilde Q_{\lambda_s} \Sigma_t \widetilde Q_{\lambda_s} \beta_{\lambda_t}.
\]

Applying the fresh-sample analogue of \Cref{lem:fresh_insertion_DE} with $A=\Sigma_t$, which commutes with $\Sigma_s$ by \Cref{def:dist}, gives
\[
  \lambda_s \widetilde Q_{\lambda_s} \asymp \kappa_s G_s,
  \qquad
  \lambda_s^2 \widetilde Q_{\lambda_s} \Sigma_t \widetilde Q_{\lambda_s}
  \asymp
  \kappa_s^2 G_s^2 \Sigma_t
  + \kappa_s^2 \frac{\gamma_s \otr(\Sigma_s \Sigma_t G_s^2)}{1 - \gamma_s \otr(\Sigma_s^2 G_s^2)} G_s \Sigma_s G_s.
\]
Hence,
\begin{align}
D(\lambda_t, \lambda_s)
  &=
  \beta_{\lambda_t}^\top A_2 \beta_{\lambda_t} + o_p(1), 
  \nonumber \\
A_2 &:= \kappa_s^2 G_s^2 \Sigma_t + \kappa_s^2 b_s G_s \Sigma_s G_s,
\end{align}
where we denote
\[
b_s := \frac{\gamma_s \otr(\Sigma_s \Sigma_t G_s^2)}{1 - \gamma_s \otr(\Sigma_s^2 G_s^2)}.
\]

Note that
\[
e_{\lambda_t} = \beta_{\lambda_t} - \beta = -\lambda_t Q_{\lambda_t} \beta + Q_{\lambda_t} X^{\top} \varepsilon / n_t.
\]

For any deterministic positive semidefinite insertion $A$ that commutes with $\Sigma_t$, \Cref{lem:fresh_insertion_DE} gives
\[
  e_{\lambda_t}^\top A e_{\lambda_t}
  \pto
  \kappa_t^2 \beta^\top G_t A G_t \beta
  +
  b_t \,\gamma_t\,\otr(A \Sigma_t G_t^2)\,\big(\kappa_t^2 q_{2,0} +\sigma^2\big),
\]
and
\[
  e_{\lambda_t}^\top A \beta
  \pto
  -\kappa_t \,\beta^\top G_t A \beta,
\]
where linear terms in $\varepsilon$ are $o_p(1)$ by the same leave-one-out concentration argument used in
Lemma B.3 of \citet{dang2026optimal}.

Using $\beta_{\lambda_t}=\beta+e_{\lambda_t}$, we therefore obtain
\[
  e_{\lambda_t}^\top A \beta_{\lambda_t}
  =
  e_{\lambda_t}^\top A \beta + e_{\lambda_t}^\top A e_{\lambda_t}
  \pto
  -\kappa_t \,\beta^\top G_t A \beta
  +
  \kappa_t^2 \beta^\top G_t A G_t \beta
  +
  b_t\,\gamma_t \,\otr(A\Sigma_t G_t^2)\,\big(\kappa_t^2 q_{2,0}+\sigma^2\big),
\]
and
\[
  \beta_{\lambda_t}^\top A \beta_{\lambda_t}
  =
  \beta^\top A \beta + 2 e_{\lambda_t}^\top A \beta + e_{\lambda_t}^\top A e_{\lambda_t}
\]
\[
  \pto
  \beta^\top A \beta
  -2\kappa_t \,\beta^\top G_t A \beta
  +
  \kappa_t^2 \beta^\top G_t A G_t \beta
  +
  b_t\,\gamma_t\,\otr(A\Sigma_t G_t^2)\,\big(\kappa_t^2 q_{2,0}+\sigma^2\big).
\]

Substituting $A=A_2$ yields
\[
  \beta^\top A_2 \beta=\kappa_s^2 q_{0,2} + \kappa_s^2 b_s \tilde q_{0,2},
  \qquad
  \beta^\top G_t A_2 \beta=\kappa_s^2 q_{1,2} + \kappa_s^2 b_s \tilde q_{1,2},
  \qquad
  \beta^\top G_t A_2 G_t \beta=\kappa_s^2 q_{2,2} + \kappa_s^2 b_s \tilde q_{2,2},
\]
and
\[
  \gamma_t \otr(A_2 \Sigma_t G_t^2)=\kappa_s^2  t_{2,2} + \kappa_s^2 b_s \gamma_t \otr(\Sigma_t \Sigma_s G_t^2 G_s^2).
\]

Finally,
\begin{align}
 \sD(\lambda_t, \lambda_s) &= \kappa_s^2
 \big(
 q_{0,2} + b_s \tilde q_{0,2} - 2 \kappa_t (q_{1,2} + b_s \tilde q_{1,2}) + \kappa_t^2 
 (q_{2,2} + b_s \tilde q_{2,2})
 \nonumber \\
 &\hspace{3em}
 + b_t (\kappa_t^2 q_{2,0} + \sigma^2) \big(t_{2,2} + b_s \gamma_t \otr(\Sigma_t \Sigma_s G_t^2 G_s^2) \big)
 \big).
 \nonumber
\end{align}
This proves \eqref{eq:fresh_D_DE}.
\end{proof}

\subsection{Background on deterministic equivalents}
\label{app:DE_Q_U-preliminaries}

We use an anisotropic (bilinear-form) notion of deterministic equivalent.

\begin{definition}[Deterministic equivalent]
\label{def:strong_DE}
Let $A=A_p$ be a (possibly random) $p\times p$ matrix and $\bar A=\bar A_p$ a deterministic $p\times p$ matrix.
We write $A\asymp \bar A$ if, for every pair of deterministic vectors $u=u_p$ and $v=v_p$ with $\|u\|_2,\|v\|_2=O(1)$,
\[
  u^\top(A-\bar A)v \pto 0.
\]
If $A(\theta),\bar A(\theta)$ depend on a parameter $\theta$ in an open set $\Theta$, we write
$A(\theta)\asymp \bar A(\theta)$ \emph{uniformly on compact subsets of $\Theta$} if the convergence above holds
uniformly over $\theta$ in any compact $K\subset \Theta$.
\end{definition}

This notion implies trace convergence whenever operator norms are uniformly bounded.
In particular, if $A\asymp \bar A$ and $\|A\|_{\mathrm{op}},\|\bar A\|_{\mathrm{op}}=O_{\PP}(1)$, then
$\otr(A)-\otr(\bar A)\pto 0$.

We will also use that uniform deterministic equivalents can be differentiated.

\begin{lemma}[Differentiate a deterministic equivalent]
\label{lem:diff_DE}
Let $A(z)$ be random and $\bar A(z)$ deterministic, both analytic in a complex neighborhood $U$ of $z_0$.
If $A(z)\asymp \bar A(z)$ locally uniformly on $U$, then
$A'(z_0)\asymp \bar A'(z_0)$.
\end{lemma}

\begin{proof}
For any deterministic $u$ and $v$, apply Cauchy's integral formula to
$u^\top(A(z)-\bar A(z))v$ on a sufficiently small circle around $z_0$ contained in $U$.
Local uniform convergence makes the resulting contour integral $o_p(1)$.
\end{proof}

For the next two lemmas, let $\widehat\Sigma=X^\top X/n$ be a generic sample covariance matrix with population covariance $\Sigma$ and aspect ratio $\gamma=p/n$, and let $Q_\lambda=(\widehat\Sigma+\lambda I_p)^{-1}$.
A standard anisotropic local law for sample-covariance resolvents gives the following deterministic equivalent; see, e.g., \citep{rubio_mestre_2011,knowles_yin_2017,patil2023generalized}.
\begin{lemma}[Scaled resolvent deterministic equivalent]
\label{lem:scaled_resolvent_DE}
Under \Cref{def:dist}, for each fixed $\lambda>0$,
\begin{equation}
\label{eq:scaled_resolvent_DE_app}
  \lambda Q_\lambda \ \asymp\ \kappa(\lambda)\,(\Sigma+\kappa(\lambda)I_p)^{-1} \ = \ \kappa G,
\end{equation}
where $\kappa=\kappa(\lambda)>0$ is the unique solution to
\begin{equation}
\label{eq:kappa_fp_app_restate}
  \kappa \;=\; \lambda + \gamma\,\kappa\,\otr\!\big(\Sigma(\Sigma+\kappa I_p)^{-1}\big),
\end{equation}
and $G=(\Sigma+\kappa I_p)^{-1}$.
The equivalence extends locally uniformly to a complex neighborhood of each $\lambda>0$ and, in particular, holds uniformly for $\lambda$ in compact subsets of $(0,\infty)$.
\end{lemma}

Define
\[
  t_2(\lambda):=\gamma \otr(\Sigma^2 G^2),
  \qquad
  b(\lambda):=\frac{1}{1-t_2(\lambda)}.
\]

\begin{lemma}[Derivative of the fixed-point solution]
\label{lem:kappa_prime_app}
The map $\lambda\mapsto \kappa(\lambda)$ is differentiable for $\lambda>0$ and
\[
  \kappa'(\lambda)=b(\lambda)=\frac{1}{1-t_2(\lambda)}.
\]
\end{lemma}

\begin{proof}
Differentiate \eqref{eq:kappa_fp_app_restate} and use $I_p-\kappa G=\Sigma G$ to obtain
\[
  \kappa'(\lambda)\bigl(1-\gamma\otr(\Sigma^2G^2)\bigr)=1.
\]
\end{proof}

\section{Proofs in Section~\ref{sec:strict_nonmonotonicity}}

\subsection{Proof of \Cref{prop:teacher_bad_set}}
\label{sec:proof_teacher_bad_set}

\begin{proof}[Proof of \Cref{prop:teacher_bad_set}]
For fixed $\lambda_t$, define
\[
H(\kappa_s):=\sum_{j=1}^{\tilde m}\frac{d_j}{\tilde\sigma_{(j)}+\kappa_s}.
\]
As $\lambda_s$ ranges over $(0,\infty)$, the corresponding fixed-point solution $\kappa_s$ ranges over an interval.
Thus, $\lambda_t\in\Lambda_t$ if and only if $H$ vanishes throughout this interval.
Clearing the denominators gives
\[
\sum_{j=1}^{\tilde m}d_j\prod_{k\ne j}(\tilde\sigma_{(k)}+\kappa_s)=0.
\]
The left-hand side is a polynomial that vanishes on an interval and is therefore identically zero.
For any $j\in[\tilde m]$, evaluating it at $\kappa_s=-\tilde\sigma_{(j)}$ yields
\[
d_j\prod_{k\ne j}(\tilde\sigma_{(k)}-\tilde\sigma_{(j)})=0.
\]
The eigenvalues $\tilde\sigma_{(1)},\ldots,\tilde\sigma_{(\tilde m)}$ are distinct, so $d_j=0$.
Conversely, if every $d_j=0$, then $H$ is identically zero and $\lambda_t\in\Lambda_t$.

We provide a crude upper bound on the cardinality of $\Lambda_t$ using only the weaker necessary condition,
\[
\sum_{j=1}^{\tilde m} d_j = \sum_{i=1}^{p} c_i = 0.
\]
This condition is equivalent to
\begin{align}
    &\sum_{i=1}^{p} \frac{ \kappa_t \sigma_i^2}{(\sigma_i + \kappa_t)^2} 
\biggl(\frac{\kappa_t^2 b_t q_{2,0} + \sigma^2 b_t}{\kappa_t n_t} - v_i \biggr) = 0,
\nonumber \\
&\Leftrightarrow \quad
\frac{b_t}{n_t} (\kappa_t^2 q_{2,0} + \sigma^2) \sum_{i=1}^{p} \frac{\sigma_i^2}{(\sigma_i + \kappa_t)^2} - \kappa_t \sum_{i=1}^{p} \frac{\sigma_i^2 v_i}{(\sigma_i + \kappa_t)^2} = 0.
\end{align}

For convenience, write $\kappa_t$ as $\kappa$.
Let $m$ be the number of distinct eigenvalues of $\Sigma_t$, denoted by $\{\sigma_{(1)},\ldots,\sigma_{(m)}\}$.
Define
\begin{align*}
    P(\kappa)&:=\prod_{j=1}^m (\sigma_{(j)}+\kappa), 
    \\
    Q(\kappa) &:=
\sum_{i=1}^p
\sigma_i^2
\prod_{j: \sigma_{(j)} \neq \sigma_i}
(\sigma_{(j)}+\kappa)^2,
\\
R(\kappa)
&:=
\sum_{i=1}^p
\sigma_i v_i
\prod_{j: \sigma_{(j)} \neq \sigma_i}
(\sigma_{(j)}+\kappa)^2,
\\
S(\kappa) &:= \sum_{i=1}^p
\sigma_i^2 v_i
\prod_{j: \sigma_{(j)} \neq \sigma_i}
(\sigma_{(j)}+\kappa)^2.
\end{align*}

Then
\begin{align*}
    \sum_{i=1}^{p} \frac{\sigma_i^2}{(\sigma_i + \kappa)^2}
    &= \frac{Q(\kappa)}{P(\kappa)^2},
    \quad
    \sum_{i=1}^{p} \frac{\sigma_i^2 v_i}{(\sigma_i + \kappa)^2}
    = \frac{S(\kappa)}{P(\kappa)^2},
    \quad
    q_{2,0} = \sum_{i=1}^{p}
    \frac{\sigma_i v_i}{(\sigma_i + \kappa)^2} = \frac{R(\kappa)}{P(\kappa)^2},
    \\
    &b_t = \frac{1}{1 - t_{2,0}}
    = \frac{1}{1 - \frac{1}{n_t} \sum_{i=1}^{p} \frac{\sigma_i^2}{(\sigma_i + \kappa)^2}} = \frac{P(\kappa)^2}{P(\kappa)^2 - Q(\kappa)/n_t}.
\end{align*}

Substituting these identities into the degeneracy condition and clearing the denominator gives
\begin{align}
    \frac{Q(\kappa)}{n_t} \bigl(\kappa^2 R(\kappa) + \sigma^2 P(\kappa)^2 \bigr)
    - \kappa S(\kappa) (P(\kappa)^2 - Q(\kappa)/n_t) = 0.
\end{align}

By construction, $\deg P=m$ and $\deg Q,\deg R,\deg S\leq 2m-2$, so this polynomial has degree at most $4m-1$.
When $\beta\ne0$, its leading coefficient is $-\sum_{i=1}^p\sigma_i^2v_i<0$.
If $\beta=0$ and $\sigma^2>0$, the polynomial reduces to $\sigma^2Q(\kappa)P(\kappa)^2/n_t$ and is again nonzero.
Thus, in the nontrivial case $r^2+\sigma^2>0$, the degeneracy condition has at most $4m-1$ solutions.
Because $\lambda_t\mapsto\kappa_t$ is one-to-one, the same bound applies to the corresponding values of $\lambda_t$.

\end{proof}

\subsection{Proof of \Cref{prop:strict_improv_ood}}
\label{sec:proof_of_strict_improvement}

\begin{proof}[Proof of \Cref{prop:strict_improv_ood}]
By \Cref{lem:fresh_num_DE_general},
\[
  \sR_{\te}(\lambda_t) - \sC(\lambda_t, \lambda_s)
  =
  \kappa_s\big(
    -\kappa_t q_{1,1}
    +\kappa_t^2 q_{2,1}
    +\kappa_t^2 b_t t_{2,1} q_{2,0}
    +\sigma^2 b_t t_{2,1}
  \big).
\]

The equation $\sR_{\te}(\lambda_t)-\sC(\lambda_t,\lambda_s)=0$ is therefore equivalent to
\begin{equation}
    \label{eq:si_condition}
    -\kappa_t q_{1,1}
    +\kappa_t^2 q_{2,1}
    +\kappa_t^2 b_t t_{2,1} q_{2,0}
    +\sigma^2 b_t t_{2,1} = 0.
\end{equation}

Recall the eigendecompositions $\Sigma_t=U^{\top}\operatorname{diag}(\sigma_1,\ldots,\sigma_p)U$ and $\Sigma_s=U^{\top}\operatorname{diag}(\tilde\sigma_1,\ldots,\tilde\sigma_p)U$, and the squared signal-covariance alignment $v_i=((U\beta)_i)^2$.
Equation \eqref{eq:si_condition} becomes
\begin{align}
    \label{eq:poly_equation_teacher}
    &\sum_{i=1}^{p}
    \biggl(\frac{- \kappa_t \sigma_i v_i}{(\sigma_i + \kappa_t)(\tilde \sigma_i + \kappa_s)} +
    \frac{\kappa_t^2 \sigma_i v_i}{(\sigma_i + \kappa_t)^2(\tilde \sigma_i + \kappa_s)} 
    + (\kappa_t^2 b_t q_{2,0} + \sigma^2 b_t) \frac{1}{n_t} \frac{\sigma_i^2}{(\sigma_i + \kappa_t)^2(\tilde \sigma_i + \kappa_s)} \biggr) = 0
    \nonumber
    \\
    &\Leftrightarrow \:
    \sum_{i=1}^{p}
    \frac{\sigma_i^2 \bigl[-\kappa_t v_i + (\kappa_t^2 b_t q_{2,0} + \sigma^2 b_t) / n_t\bigr]}{(\sigma_i + \kappa_t)^2}
    \frac{1}{\tilde \sigma_i + \kappa_s}= 0. 
\end{align}

For $i\in[p]$, define
$c_i := \frac{\kappa_t \sigma_i^2}{(\sigma_i + \kappa_t)^2}
\big[\frac{\kappa_t^2 b_t q_{2,0} + \sigma^2 b_t}{\kappa_t n_t} - v_i \big]$.
Then \eqref{eq:poly_equation_teacher} becomes
\begin{align}
    \label{eq:poly_c_i_teacher}
    H(\kappa_s) := \sum_{i=1}^{p} \frac{c_i}{\tilde \sigma_i + \kappa_s} =
\sum_{j=1}^{\tilde m} \frac{d_j
}{\tilde \sigma_{(j)} + \kappa_s} = 0,
\end{align}
where $d_j:=\sum_{k:\,\tilde\sigma_k=\tilde\sigma_{(j)}}c_k$.

Let $j_1<\cdots<j_K$ index the nonzero entries of $(d_1,\ldots,d_{\tilde m})$, and let $M$ be the number of sign changes in $(d_{j_1},\ldots,d_{j_K})$.
Because $\lambda_t\notin\Lambda_t$, \Cref{prop:teacher_bad_set} ensures that $K\ge1$.
Suppose, for contradiction, that $H$ has $r>M$ distinct positive roots $0<x_1<\cdots<x_r$.

Let $\mathcal I$ contain the indices $a$ at which $d_{j_a}$ and $d_{j_{a+1}}$ have opposite signs.
For each $a\in\mathcal I$, choose $\eta_a\in(\tilde\sigma_{(j_a)},\tilde\sigma_{(j_{a+1})})$ and set
\[
Q(t):=c_0\prod_{a\in\mathcal I}(t-\eta_a),
\]
where the sign of $c_0\ne0$ is chosen so that
\[
\operatorname{sign}Q(\tilde\sigma_{(j_a)})=\operatorname{sign}d_{j_a},
\qquad a=1,\ldots,K.
\]
This construction gives $\deg(Q)=M\le r-1$.
Define
\[
P(t):=\frac{Q(t)}{\prod_{\ell=1}^r(t+x_\ell)}.
\]
Since $\deg(Q)\le r-1$, $P$ has the partial fraction decomposition
\[
P(t)=\sum_{\ell=1}^r\frac{y_\ell}{t+x_\ell},
\qquad
y_\ell:=\frac{Q(-x_\ell)}{\prod_{k\ne\ell}(x_k-x_\ell)}.
\]
Indeed,
\[
Q(t)=\sum_{\ell=1}^r y_\ell\prod_{k\ne\ell}(t+x_k),
\]
because the two polynomials have degree at most $r-1$ and agree at $t=-x_1,\ldots,-x_r$.

Since $x_\ell>0$, the denominator of $P(\tilde\sigma_{(j)})$ is positive.
Thus, $d_jP(\tilde\sigma_{(j)})>0$ for every nonzero $d_j$, and
\[
\sum_{j=1}^{\tilde m}d_jP(\tilde\sigma_{(j)})>0.
\]
On the other hand,
\begin{align}
\sum_{j=1}^{\tilde m}d_jP(\tilde\sigma_{(j)})
&=\sum_{\ell=1}^r y_\ell\sum_{j=1}^{\tilde m}\frac{d_j}{\tilde\sigma_{(j)}+x_\ell}
=\sum_{\ell=1}^r y_\ell H(x_\ell)=0,
\end{align}
which is a contradiction.
Therefore, $H$ has at most $M$ positive roots.
\end{proof}

\subsection{Proof of \Cref{cor:strict_improv_special_cases}}

\begin{proof}[Proof of \Cref{cor:strict_improv_special_cases}]
We consider the following cases.
\begin{enumerate}[leftmargin=7mm]
    \item [(a)] Isotropic design: equation \eqref{eq:poly_c_i_teacher} has no solution if $\sum_{i=1}^{p} c_i \neq 0$. Consequently, $\sR_{\m}^{\star}(\lambda_t, \lambda_s) < \sR(\lambda_t)$ for all $\lambda_s > 0$.
    
    If $\sum_{i=1}^{p} c_i = 0$, then $H(\kappa_s) \equiv 0$ and $\sR_{\m}^{\star}(\lambda_t, \lambda_s) = \sR(\lambda_t)$ for all $\lambda_s > 0$. Since $\sum_i v_i = \|\beta\|_2^2=r^2$, this degeneracy condition is equivalent to
    \begin{align}
        \label{eq:degeneracy_isotropic}
        &-\kappa_t r^2 +  \gamma_t (\kappa_t^2 b_t q_{2,0} + \sigma^2 b_t) = 0,
    \end{align}
    where
    \begin{align}
        &q_{2,0} = \frac{r^2}{(1 + \kappa_t)^2}, 
        \quad b_t = \frac{(1 + \kappa_t)^2}{(1 + \kappa_t)^2 - \gamma_t},
        \\
        \label{eq:kappa_u_isotropic}
        &\kappa_t = \lambda_t + \gamma_t \frac{\kappa_t}{1 + \kappa_t}.
    \end{align}

    From \eqref{eq:kappa_u_isotropic},
    \begin{align}   
        \label{eq:lambda_t_iso}
        \lambda_t = \kappa_t - \gamma_t \frac{\kappa_t}{1 + \kappa_t} = \frac{\kappa_t(1 + \kappa_t - \gamma_t)}{1 + \kappa_t}.
    \end{align}
    
    Since $(1 + \kappa_t)^2 - \gamma_t > 0$ for $\lambda_t > 0$, equation \eqref{eq:degeneracy_isotropic} can be rewritten as
    \begin{align}
        &-\kappa_t r^2 + \gamma_t \frac{r^2 \kappa_t^2 + \sigma^2 (1 + \kappa_t)^2}{(1 + \kappa_t)^2  - \gamma_t}
        = 0
        \nonumber \\
        &\Leftrightarrow \:
        \kappa_t \big[(1 + \kappa_t)^2 - \gamma_t\big] = \gamma_t \biggl(\kappa_t^2 + \frac{\sigma^2}{r^2} (1 + \kappa_t)^2 \biggr)
        \nonumber \\
        &\Leftrightarrow \:
        \kappa_t (1 + \kappa_t)(1 + \kappa_t - \gamma_t) = \gamma_t \frac{\sigma^2}{r^2} (1 + \kappa_t)^2
        \nonumber \\
        &\Leftrightarrow \:
        \frac{\kappa_t (1 + \kappa_t - \gamma_t)}{1 + \kappa_t} = \gamma_t \frac{\sigma^2}{r^2}.
    \label{eq:kappa_u_degenerate_isotropic}
    \end{align}

    Comparing \eqref{eq:lambda_t_iso} and \eqref{eq:kappa_u_degenerate_isotropic} shows that the only value satisfying the degeneracy condition is the ridge-optimal $\lambda_t^{\star}:=\gamma_t\sigma^2/r^2$.

    \item [(b)] Spiked covariance model with $s$ spikes: here, $m=s+1$, so the sequence $\{d_1,\ldots,d_m\}$ has at most $m-1=s$ sign changes. The claim follows from \Cref{prop:strict_improv_ood}.

    \item [(c)] $\Sigma_t$ with $p$ distinct eigenvalues: we have
\[d_j = c_j = \frac{\kappa_t \sigma_j^2}{(\sigma_j + \kappa_t)^2} 
\Bigg( \underbrace{\frac{\kappa_t^2 b_t q_{2,0} + \sigma^2 b_t}{\kappa_t n_t}}_{v_0} - v_j \Bigg).
\]
Since $(v_1,\ldots,v_p)$ is monotone and at least one $v_i\ne v_0$, not all the $d_j$ vanish, so $\lambda_t\notin\Lambda_t$ by \Cref{prop:teacher_bad_set}.
Moreover, $(v_1,\ldots,v_p)$ can cross the threshold $v_0$ at most once (due to monotonicity), so $(d_1,\ldots,d_p)$ has at most one sign change after its zero entries are removed.
The claim follows from \Cref{prop:strict_improv_ood}.
\end{enumerate}

\end{proof}

\subsection{Proof of \Cref{cor:balanced_regime}}

\begin{proof}[Proof of \Cref{cor:balanced_regime}]

Since $\lambda_t = \lambda_s = \lambda$ and $\gamma_t = \gamma_s$, we have $\kappa_t = \kappa_s = \kappa$ and $G_t = G_s$. Thus, by \Cref{lem:fresh_num_DE_general},
\[
  \sR(\lambda) - \sC(\lambda, \lambda) =
  \kappa
  \big(
    -\kappa q_{2,0}
    +\kappa^2 q_{3,0}
    +b_t\,t_{3,0} \big(\kappa^2 q_{2,0}+\sigma^2\big)
  \big).
\]

Differentiating
\[
  \sR(\lambda)=\sigma^2 + \kappa^2 b_t\,q_{2,0} +\sigma^2 u_2,
  \qquad
  u_2=t_{2,0} b_t,
\]
and using the identities
\[
  \kappa'(\lambda)=b_t,\qquad
  q_{2,0}'(\lambda)=-2 b_t q_{3,0},\qquad
  t_{2,0}'(\lambda)=-2 b_t t_{3,0},\qquad
  b_t'(\lambda)=-2b_t^3 t_{3,0},\qquad
  u_2'(\lambda)=-2b_t^3 t_{3,0},
\]
we obtain
\begin{equation}
\label{eq:R_prime}
\sR'(\lambda)
  =
  -2b_t^2
  \big(
    -\kappa q_{2,0}
    +\kappa^2 q_{3,0}
    +b_t\,t_{3,0}\big(\kappa^2 q_{2,0}+\sigma^2\big)
  \big).
\end{equation}

Therefore,
\[
\mathcal R_{\te}(\lambda)-\mathcal C(\lambda, \lambda) =-\frac{\kappa}{2b_t^2}\,\mathcal R_{\te}'(\lambda).
\]
The strict-improvement and sign claims now follow from \Cref{prop:oracle_affine}.
\end{proof}

\subsection{Proof of \Cref{corr:strict_improv_isotropic_fresh}}
\label{sec:proof_of_strict_improv_isotropic_fresh}

\begin{proof}[Proof of \Cref{corr:strict_improv_isotropic_fresh}]
When $\Sigma_s=I_p$, all $\tilde\sigma_i$ equal $1$, so \eqref{eq:poly_c_i_teacher} becomes
    \[
    H(\kappa_s) = \frac{\sum_{i=1}^p c_i}{1 + \kappa_s}.
    \]

If the numerator is nonzero, $H(\kappa_s)=0$ has no solution; hence, strict improvement holds for every $\lambda_s>0$.
\end{proof}

\subsection{Proof of \Cref{prop:strict_improv_pd}}

\begin{proof}[Proof of \Cref{prop:strict_improv_pd}]
The oracle decomposition in \Cref{prop:oracle_affine} gives
\begin{align}
    \sR^{\star}_{\m}(\lambda_t, \lambda_s) 
    = \mathcal R_{\pd}(\lambda_s)
    -
    \frac{\bigl(\mathcal R_{\pd}(\lambda_s)-\mathcal C(\lambda_t, \lambda_s)\bigr)^2}{\mathcal R_{\te}(\lambda_t)+\mathcal R_{\pd}(\lambda_s) - 2 \mathcal C(\lambda_t, \lambda_s)}.
\end{align}
Thus, $\sR^{\star}_{\m}(\lambda_t,\lambda_s)\leq\mathcal R_{\pd}(\lambda_s)$, with strict inequality if and only if $\mathcal R_{\pd}(\lambda_s)-\mathcal C(\lambda_t,\lambda_s)\ne0$.

We now characterize this difference.
For every $\lambda_t,\lambda_s>0$,
\begin{equation}
\label{eq:Rpd_minus_C_equals_D_minus_RminusC}
  \sR_{\pd}(\lambda_s)-\sC(\lambda_t, \lambda_s)
  =
  \sD(\lambda_t, \lambda_s) -\big(\sR_{\te}(\lambda_t)-\sC(\lambda_t, \lambda_s)\big).
\end{equation}

Subtracting the deterministic equivalent for $R_{\te}-C$ in \Cref{lem:fresh_num_DE_general} from that for $D$ in \Cref{lem:freshX_DE_delta_D_Rpd} gives
\begin{align}
   \mathcal R_{\pd}(\lambda_s)-\mathcal C(\lambda_t, \lambda_s) &= 
   \kappa_s^2
 \big(
 q_{0,2} + b_s \tilde q_{0,2} - 2 \kappa_t (q_{1,2} + b_s \tilde q_{1,2}) + \kappa_t^2 
 (q_{2,2} + b_s \tilde q_{2,2})
 \nonumber \\
 &\hspace{3em}
 + b_t (\kappa_t^2 q_{2,0} + \sigma^2) \big(t_{2,2} + b_s \gamma_t \otr(\Sigma_t \Sigma_s G_t^2 G_s^2) \big)
 \big) \nonumber \\ 
 &\hspace{6em}
 -\kappa_s\big(
        -\kappa_t q_{1,1}
        +\kappa_t^2 q_{2,1}
        +\kappa_t^2 b_t t_{2,1} q_{2,0}
        +\sigma^2 b_t t_{2,1}
      \big).
\end{align}

When $\Sigma_s = \Sigma_t$, we have $\tilde q_{i,j} = q_{i,j}$ and $1 + b_s = 1/(1 - \gamma_s \otr(\Sigma_t^2 G_s^2))$. Hence,
\begin{align}
       \label{eq:R_pd-C}
   \mathcal R_{\pd}(\lambda_s)-\mathcal C(\lambda_t, \lambda_s) &= 
   \kappa_s^2 (1 + b_s)
 \big(
 q_{0,2} - 2 \kappa_t q_{1,2} + \kappa_t^2 
 q_{2,2} + 
 b_t (\kappa_t^2 q_{2,0} + \sigma^2) t_{2,2}
 \big) 
 \nonumber \\
 &\hspace{3em}
  -\kappa_s\big(
        -\kappa_t q_{1,1}
        +\kappa_t^2 q_{2,1}
        +\kappa_t^2 b_t t_{2,1} q_{2,0}
        +\sigma^2 b_t t_{2,1}
      \big).
\end{align}

Setting the right-hand side of \eqref{eq:R_pd-C} to zero and simplifying gives
\begin{align}
    \frac{\kappa_s}{1 - \frac{\gamma_s}{p} \sum_{i=1}^{p} \frac{\sigma_i^2}{(\sigma_i + \kappa_s)^2}} 
    \cdot \sum_{i=1}^{p} 
    \frac{\sigma_i^3 v_i + \frac{\gamma_t}{p} \sigma_i^2 b_t (\kappa_t^2 q_{2,0}+\sigma^2) }{(\sigma_i + \kappa_t)^2 (\sigma_i + \kappa_s)^2} 
    = \sum_{i=1}^{p} \frac{\sigma_i^2}{(\sigma_i + \kappa_t)^2 (\sigma_i + \kappa_s)} \biggl( -\kappa_t v_i +  \frac{\gamma_t}{p} b_t (\kappa_t^2 q_{2,0}+\sigma^2)\biggr).
    \nonumber
\end{align}

Under the isotropic design $\Sigma_t=\Sigma_s=I_p$, this simplifies to
\begin{align}
    &\frac{\kappa_s (1 + \kappa_s)^2}{(1 + \kappa_s)^2 - \gamma_s} 
    \cdot \frac{r^2 + \gamma_t b_t (\kappa_t^2 q_{2,0} + \sigma^2)}{(1 + \kappa_t)^2 (1 + \kappa_s)^2}
    = \frac{-\kappa_t r^2 + \gamma_t b_t (\kappa_t^2 q_{2,0} + \sigma^2)}{(1 + \kappa_t)^2 (1 + \kappa_s)} \nonumber 
    \\
    &\Leftrightarrow \:
    \big[r^2 + \gamma_t b_t (\kappa_t^2 q_{2,0} + \sigma^2)\big] \kappa_s (1 + \kappa_s) = \big[-\kappa_t r^2 + \gamma_t b_t (\kappa_t^2 q_{2,0} + \sigma^2)\big] \big[(1 + \kappa_s)^2 - \gamma_s \big]
    \nonumber 
    \\ 
    &\Leftrightarrow \:
    (1+\kappa_t)r^2\kappa_s^2
    +
    \bigl[
    (1+2\kappa_t)r^2
    -
    \gamma_t b_t\bigl(\kappa_t^2 q_{2,0}+\sigma^2\bigr)
    \bigr]\kappa_s
    +
    (1-\gamma_s)
    \bigl[
    \kappa_t r^2
    -
    \gamma_t b_t\bigl(\kappa_t^2 q_{2,0}+\sigma^2\bigr)
    \bigr]
    =0.
    \nonumber
\end{align}

The left-hand side is a polynomial of degree at most two in $\kappa_s$.
In the nontrivial case $r^2+\sigma^2>0$, it is not identically zero: if $r^2>0$, its quadratic coefficient is $(1+\kappa_t)r^2>0$, while if $r^2=0$ and $\sigma^2>0$, its linear coefficient is $-\gamma_t b_t\sigma^2<0$.
Therefore, for any fixed $\lambda_t>0$, there are at most two values of $\kappa_s$, and hence at most two corresponding choices of $\lambda_s$, for which $\mathcal R_{\pd}(\lambda_s)-\mathcal C(\lambda_t,\lambda_s)=0$.

\end{proof}

\section{Proofs in Section~\ref{sec:tuning_ridge}}

\subsection{Proof of \Cref{thm:calibration_consistency}}
\label{sec:appendix_tuning_calibration}

We prove a slightly more general finite-candidate version of \Cref{thm:calibration_consistency}.
This version is useful when the same calibration set is used to choose among finitely many student regularization parameters $\lambda_s$, unlabeled sample sizes $n_s$, or other finite hyperparameter choices.

The argument is straightforward: once the teacher and fresh-$X$ pure-distilled student are fixed, calibration estimates only the risks and residual correlation of two fixed predictors.
Thus, no high-dimensional condition on $p/n_{\rm cal}$ is needed.

Let $\mathcal G$ denote the sigma-field generated by the deployed teacher, the fresh unlabeled covariates, and all fitted pure-distilled students under consideration.
Conditionally on $\mathcal G$, the teacher and student predictors are fixed functions.
For a single candidate $f_{\pd}$, let $(X_i^{\rm cal},Y_i^{\rm cal})_{i=1}^{n_{\rm cal}}$ be an independent labeled calibration sample drawn from the target test distribution.
Define
\[
  e_i^{\te}:=Y_i^{\rm cal}-f^{\te}(X_i^{\rm cal}),
  \qquad
  e_i^{\pd}:=Y_i^{\rm cal}-f_{\pd}(X_i^{\rm cal}).
\]
The oracle quantities are
\[
  R_{\te}:=\E[(e^{\te})^2\mid \mathcal G],
  \qquad
  R_{\pd}:=\E[(e^{\pd})^2\mid \mathcal G],
  \qquad
  C:=\E[e^{\te}e^{\pd}\mid \mathcal G],
\]
where $(e^{\te},e^{\pd})$ denotes an independent test residual pair. 
The calibration estimates are
\[
  \widehat R_{\te}^{\rm cal}
  :=
  \frac1{n_{\rm cal}}\sum_{i=1}^{n_{\rm cal}}(e_i^{\te})^2,
  \qquad
  \widehat R_{\pd}^{\rm cal}
  :=
  \frac1{n_{\rm cal}}\sum_{i=1}^{n_{\rm cal}}(e_i^{\pd})^2,
  \qquad
  \widehat C^{\rm cal}
  :=
  \frac1{n_{\rm cal}}\sum_{i=1}^{n_{\rm cal}}e_i^{\te}e_i^{\pd}.
\]

We first prove the following finite-candidate result.

\begin{theoremsupp}[Calibration consistency, finite-candidate version]
\label{thm:calibration_consistency_finite_candidates}
Fix a finite collection of candidate pure-distilled students
\[
  f_{\pd}^{(1)},\ldots,f_{\pd}^{(J)}
\]
with \(J<\infty\), and condition on $\mathcal G$.
For each \(j\in[J]\), define
\[
  R_{\pd,j}:=\E[(e_j^{\pd})^2\mid\mathcal G],
  \qquad
  C_j:=\E[e^{\te}e_j^{\pd}\mid\mathcal G],
  \qquad
  D_j:=R_{\te}+R_{\pd,j}-2C_j,
\]
where \(e_j^{\pd}:=Y-f_{\pd}^{(j)}(X)\) for an independent target-distribution pair $(X,Y)$.
For the calibration observations, let
\[
e_{i,j}^{\pd}:=Y_i^{\rm cal}-f_{\pd}^{(j)}(X_i^{\rm cal}),
\]
and define
\[
\widehat R_{\pd,j}^{\rm cal}
:=\frac{1}{n_{\rm cal}}\sum_{i=1}^{n_{\rm cal}}(e_{i,j}^{\pd})^2,
\qquad
\widehat C_j^{\rm cal}
:=\frac{1}{n_{\rm cal}}\sum_{i=1}^{n_{\rm cal}}e_i^{\te}e_{i,j}^{\pd}.
\]
Assume that there exists \(d_0>0\) such that
\[
  \min_{j\in[J]} D_j \ge d_0,
\]
and that the conditional variances of
\[
  (e^{\te})^2,\qquad (e_j^{\pd})^2,\qquad e^{\te}e_j^{\pd},
  \qquad j\in[J],
\]
are bounded uniformly over \(j\in[J]\).
Then, as \(n_{\rm cal}\to\infty\),
\[
  \max_{j\in[J]}
  \bigl|
    \widehat \xi_{{\rm cal},j}^{\star}-\xi_j^\star
  \bigr|
  \pto 0,
  \qquad
  \max_{j\in[J]}
  \bigl|
    \widehat R_{\m,{\rm cal},j}^{\star}
    -
    R_{\m,j}^{\star}
  \bigr|
  \pto 0,
\]
where
\[
  \xi_j^\star
  :=
  \frac{R_{\te}-C_j}{D_j},
  \qquad
  R_{\m,j}^{\star}
  :=
  R_{\te}
  -
  \frac{(R_{\te}-C_j)^2}{D_j},
\]
and \(\widehat \xi_{{\rm cal},j}^{\star}\), \(\widehat R_{\m,{\rm cal},j}^{\star}\) are the corresponding plug-in calibration estimates.
In particular, \Cref{thm:calibration_consistency} follows by taking \(J=1\).
\end{theoremsupp}

\begin{proof}
We first establish consistency of the three empirical quantities.
For the teacher risk estimate, conditional on \(\mathcal G\),
\[
  \E[\widehat R_{\te}^{\rm cal}\mid\mathcal G]=R_{\te},
  \qquad
  \Var(\widehat R_{\te}^{\rm cal}\mid\mathcal G)
  =
  \frac{1}{n_{\rm cal}}\Var((e^{\te})^2\mid\mathcal G).
\]
The assumed variance bound and Chebyshev's inequality give
\[
  \widehat R_{\te}^{\rm cal}-R_{\te}\pto 0.
\]
The same argument, followed by a union bound over the fixed candidate set, gives
\[
  \max_{j\in[J]}
  \bigl|
    \widehat R_{\pd,j}^{\rm cal}-R_{\pd,j}
  \bigr|
  \pto 0,
  \qquad
  \max_{j\in[J]}
  \bigl|
    \widehat C_j^{\rm cal}-C_j
  \bigr|
  \pto 0.
\]
Therefore,
\[
  \max_{j\in[J]}
  \bigl|
    \widehat D_j^{\rm cal}-D_j
  \bigr|
  \pto 0,
  \qquad
  \widehat D_j^{\rm cal}
  :=
  \widehat R_{\te}^{\rm cal}
  +
  \widehat R_{\pd,j}^{\rm cal}
  -
  2\widehat C_j^{\rm cal}.
\]
Since \(\min_j D_j\ge d_0\), we have
\[
  \min_{j\in[J]}\widehat D_j^{\rm cal}\ge d_0/2
\]
with probability tending to one.

Thus, with probability tending to one, every plug-in denominator is bounded away from zero.
Applying the continuous mapping theorem to each candidate and taking the maximum over the fixed finite set yields
\[
  \max_{j\in[J]}
  \bigl|
    \widehat \xi_{{\rm cal},j}^{\star}-\xi_j^\star
  \bigr|
  \pto 0,
  \qquad
  \max_{j\in[J]}
  \bigl|
    \widehat R_{\m,{\rm cal},j}^{\star}
    -
    R_{\m,j}^{\star}
  \bigr|
  \pto 0.
\]
\end{proof}

This finite-candidate consistency statement immediately implies that calibration can also be used for finite hyperparameter selection.

\begin{corollary}[Selecting among finitely many candidates]
\label{cor:calibration_selection_finite_candidates}
Under the assumptions of \Cref{thm:calibration_consistency_finite_candidates}, let
\[
  \widehat j
  \in
  \argmin_{j\in[J]}
  \widehat R_{\m,{\rm cal},j}^{\star}.
\]
Then
\[
  R_{\m,\widehat j}^{\star}
  -
  \min_{j\in[J]} R_{\m,j}^{\star}
  \pto 0.
\]
If the oracle minimizer \(j^\star\in\argmin_j R_{\m,j}^{\star}\) is unique and separated by a positive gap, then
\[
  \PP(\widehat j=j^\star)\to 1.
\]
\end{corollary}

\begin{proof}
Let \(\Delta_{\rm cal}:=\max_{j\in[J]}|\widehat R_{\m,{\rm cal},j}^{\star}-R_{\m,j}^{\star}|\). 
By \Cref{thm:calibration_consistency_finite_candidates}, \(\Delta_{\rm cal}\pto0\). 
By definition of \(\widehat j\),
\[
  \widehat R_{\m,{\rm cal},\widehat j}^{\star}
  \le
  \widehat R_{\m,{\rm cal},j}^{\star}
  \qquad\text{for every }j\in[J].
\]
Let \(j^\star\in\argmin_j R_{\m,j}^{\star}\). Then
\[
  R_{\m,\widehat j}^{\star}
  \le
  \widehat R_{\m,{\rm cal},\widehat j}^{\star}+\Delta_{\rm cal}
  \le
  \widehat R_{\m,{\rm cal},j^\star}^{\star}+\Delta_{\rm cal}
  \le
  R_{\m,j^\star}^{\star}+2\Delta_{\rm cal}.
\]
Thus
\[
  0
  \le
  R_{\m,\widehat j}^{\star}
  -
  \min_{j\in[J]} R_{\m,j}^{\star}
  \le
  2\Delta_{\rm cal}
  \pto0.
\]
If the minimizer is unique, let
\[
g:=\min_{j\ne j^\star}\bigl(R_{\m,j}^{\star}-R_{\m,j^\star}^{\star}\bigr)>0
\]
denote its separation gap.
The event $\Delta_{\rm cal}<g/2$ then forces \(\widehat j=j^\star\).
This event has probability tending to one.
\end{proof}

\subsection{Vector-valued predictors}
\label{sec:appendix_tuning_vector_valued}

The same calibration method extends directly to vector-valued prediction, which is used in our squared-loss linear-probing experiments.
Suppose \(Y\in\mathbb R^K\) and predictors \(f:\mathcal X\to\mathbb R^K\) are evaluated by squared risk
\[
  R(f):=\E\big[\|Y-f(X)\|_2^2\big].
\]
For two fixed predictors \(f^{\te}\) and \(f_{\pd}\), define vector residuals
\[
  e^{\te}:=Y-f^{\te}(X),
  \qquad
  e^{\pd}:=Y-f_{\pd}(X),
\]
and set
\[
  R_{\te}:=\E[\|e^{\te}\|_2^2],
  \qquad
  R_{\pd}:=\E[\|e^{\pd}\|_2^2],
  \qquad
  C:=\E[\langle e^{\te},e^{\pd}\rangle].
\]
Then the same affine-risk expansion gives
\[
  \xi^\star
  =
  \frac{R_{\te}-C}{R_{\te}+R_{\pd}-2C},
  \qquad
  R_{\m}^{\star}
  =
  R_{\te}
  -
  \frac{(R_{\te}-C)^2}{R_{\te}+R_{\pd}-2C}.
\]
The calibration estimates are obtained by replacing scalar squares and products with squared Euclidean norms and inner products:
\[
  \widehat R_{\te}^{\rm cal}
  =
  \frac1{n_{\rm cal}}\sum_{i=1}^{n_{\rm cal}}
  \|e_i^{\te}\|_2^2,
  \quad
  \widehat R_{\pd}^{\rm cal}
  =
  \frac1{n_{\rm cal}}\sum_{i=1}^{n_{\rm cal}}
  \|e_i^{\pd}\|_2^2,
  \quad
  \widehat C^{\rm cal}
  =
  \frac1{n_{\rm cal}}\sum_{i=1}^{n_{\rm cal}}
  \langle e_i^{\te},e_i^{\pd}\rangle.
\]
The proof of \Cref{thm:calibration_consistency_finite_candidates} applies directly under the corresponding finite-variance assumptions on
\[
  \|e^{\te}\|_2^2,
  \qquad
  \|e^{\pd}\|_2^2,
  \qquad
  \langle e^{\te},e^{\pd}\rangle.
\]

A note on the computational cost:
The calibration estimator is ``one-shot'' in the following sense.
For a fixed candidate student \(f_{\pd}\), after training \(f_{\pd}\) once on fresh pseudo-labels, tuning \(\xi\) requires only
\[
  \big(f^{\te}(x_i^{\rm cal}), f_{\pd}(x_i^{\rm cal}), y_i^{\rm cal}\big)_{i=1}^{n_{\rm cal}},
\]
followed by three empirical averages and the closed-form plug-in formula.
There is no retraining over candidate mixing weights \(\xi\).
If one considers finitely many candidate students, such as different student penalties \(\lambda_s\) or different fresh unlabeled sample sizes \(n_s\), each student must be trained once, but the same calibration set can evaluate all candidates by \Cref{cor:calibration_selection_finite_candidates}.

\section{Proofs in Section~\ref{sec:logistic}}

\subsection{Preliminaries}

For readability, we write $\rho$ in the finite sample systems below.
When $n\rho$ is not an integer, this notation denotes $\rho_n:=\lfloor n\rho\rfloor/n$; since $\rho_n\to\rho$, this does not affect the limiting conclusions.

Recall the training objective for the student classifier based on the teacher's hard labels:
\begin{align}
    \label{eq:class_training_loss_appendix}
    \theta_{\pd} := \argmin_{\theta}
    \underbrace{\frac{1}{2n} \sum_{i=1}^{2n}
    \text{CE}(\hat y_i, \sigma( \langle \theta, \phi(x_i) \rangle)) + \frac{\lambda}{2} \| \theta \|^2}_{\mathcal{L}_1(\theta)}.
\end{align}

The loss can be rewritten as
\begin{align}
    \mathcal{L}_1(\theta) 
    &= \frac{1}{2n} \sum_{i=1}^{2n}
    \Bigl( \log \big(1 + e^{\langle \theta, \phi(x_i) \rangle} \big) - \hat y_i \langle \theta, \phi(x_i) \rangle \Bigr) +  \frac{\lambda}{2} \| \theta \|^2,
    \\
    \nabla \mathcal{L}_1(\theta) 
    &= \frac{1}{2n} \sum_{i=1}^{2n} 
    \Big( \sigma(\langle \theta, \phi(x_i) \rangle) - \hat y_i  \Big) \phi(x_i) + \lambda \theta.
\end{align}

Setting $\nabla \mathcal{L}_1(\theta)=0$ gives
\begin{equation}
    \theta_{\pd} = \sum_{i=1}^{2n} \underbrace{\frac{1}{2n \lambda} \Big( \hat{y}_i - \sigma ( \langle \theta_{\pd}, \phi(x_i) \rangle ) \Big) }_{:=\alpha(x_i, \hat y_i)} \phi(x_i) = \sum_{i=1}^{2n} \alpha(x_i, \hat y_i) \phi(x_i).
\end{equation}

It follows that
\begin{align}
    \boxed{y_{\pd}(x_i) = \sigma ( \langle \theta_{\pd}, \phi(x_i) \rangle ) = \hat y_i - 2 n \lambda \alpha(x_i, \hat y_i).}
\end{align}

The following result characterizes these coefficients.

\begin{lemma}[Adapted from Lemma 2 of \citet{das2023understanding}]
\label{lem:alpha_form}
Under \Cref{assp:feature_corr}, let $\hl_n:=2n\lambda$. Then
\[
\alpha(x_i, \hat y_i) =
\begin{cases}
-\ha_n & \text{for } i \in \mathcal{S}_{1,\mathrm{bad}}, \\
\alpha_n & \text{for } i \in \mathcal{S}_{1,\mathrm{good}}, \\
\ha_n & \text{for } i \in \mathcal{S}_{0,\mathrm{bad}}, \\
-\alpha_n & \text{for } i \in \mathcal{S}_{0,\mathrm{good}}.
\end{cases}.
\]

Here, $\alpha_n\ge0$ and $\ha_n\ge0$ solve
\begin{align}
    \label{eq:equations_alpha_n}
\left\{
\begin{aligned}
\sigma \Big( cn(\alpha_n - (\alpha_n + \ha_n)\rho) - (1-c) \ha_n \Big) &= \hl_n \ha_n \\
\sigma \Big( cn(\alpha_n - (\alpha_n + \ha_n)\rho) + (1-c)\alpha_n \Big) &= 1 - \hl_n \alpha_n
\end{aligned}
\right. . 
\end{align}
\end{lemma}

By \Cref{lem:alpha_form}, the student predictions on the training set are
\begin{equation}
    \label{eq:pd_prediction}
   y_{\pd}(x_i) =
\begin{cases}
\hl_n \ha_n & \text{for } i \in \mathcal{S}_{1,\mathrm{bad}}, \\
1 - \hl_n \alpha_n & \text{for } i \in \mathcal{S}_{1,\mathrm{good}}, \\
1 - \hl_n \ha_n & \text{for } i \in \mathcal{S}_{0,\mathrm{bad}}, \\
\hl_n \alpha_n & \text{for } i \in \mathcal{S}_{0,\mathrm{good}}.
\end{cases} 
\end{equation}

For $y_{\pd}$ to achieve $100\%$ training accuracy with respect to the ground-truth labels, we require $\hl_n\ha_n>0.5$ and $\hl_n\alpha_n<0.5$.

In \Cref{sec:extensions_logistic}, we train a second student on the first student's soft labels $y_{\pd}(\cdot)$ over the same unlabeled set $\cD_{\rm unlab}$. Its training objective is
\begin{align}
    \label{eq:class_training_loss_2_appendix}
    \theta_{\pd}^{(2)} &:= \argmin_{\theta}
    \underbrace{\frac{1}{2n} \sum_{i=1}^{2n}
    \text{CE}(y_{\pd}(x_i), \sigma( \langle \theta, \phi(x_i) \rangle)) + \frac{\lambda}{2} \| \theta \|^2}_{\mathcal{L}_2(\theta)}.
\end{align}

Setting $\nabla\mathcal L_2(\theta)=0$ gives

\begin{equation}
    \theta_{\pd}^{(2)} = \sum_{i=1}^{2n} \underbrace{\frac{1}{2n \lambda} \Big( y_{\pd}(x_i) - \sigma ( \langle \theta_{\pd}^{(2)}, \phi(x_i) \rangle ) \Big) }_{:=\beta_i} \phi(x_i) = \sum_{i=1}^{2n} \beta_i \phi(x_i).
\end{equation}

The following lemma characterizes the coefficients $\beta_i$.

\begin{lemma}[Adapted from Lemma 3 of \citet{das2023understanding}]
    \label{lem:beta_form}
    Under \Cref{assp:feature_corr}, let $\hl_n:=2n\lambda$. Then
    \begin{equation}
        \beta_i =
        \begin{cases}
        -\hb_n & \text{for } i \in \mathcal{S}_{1,\mathrm{bad}}, \\
        \beta_n & \text{for } i \in \mathcal{S}_{1,\mathrm{good}}, \\
        \hb_n & \text{for } i \in \mathcal{S}_{0,\mathrm{bad}}, \\
        -\beta_n & \text{for } i \in \mathcal{S}_{0,\mathrm{good}}
        \end{cases}.
    \end{equation}
    Here, $\beta_n$ and $\hb_n$ solve
    \begin{align}
    \label{eq:equations_beta_n}
     \left\{
    \begin{aligned}
    \sigma\!\bigl(cn(\beta_n - (\beta_n + \hb_n)\rho) - (1-c)\hb_n\bigr) &= \hl_n \ha_n + \hl_n \hb_n
    \\
     \sigma\!\bigl(cn(\beta_n - (\beta_n + \hb_n)\rho) + (1-c)\beta_n \bigr) &= 1 - \hl_n \alpha_n - \hl_n \beta_n,
    \end{aligned}
    \right.
    \end{align}
    where $(\alpha_n,\ha_n)$ solve \eqref{eq:equations_alpha_n}. If $\hl_n\ha_n<0.5$ and $\hl_n\alpha_n<0.5$, then $\beta_n,\hb_n\geq0$.
\end{lemma}

By \Cref{lem:beta_form}, the second student's predictions on the training set are
\begin{equation}
    \label{eq:pd_2_prediction}
    y_{\pd}^{(2)}(x_i) =
    \begin{cases}
    \hl_n \ha_n + \hl_n \hb_n & \text{for } i \in \mathcal{S}_{1,\text{bad}}, \\
    1 - \hl_n \alpha_n - \hl_n \beta_n & \text{for } i \in \mathcal{S}_{1,\text{good}}, \\
    1 - \hl_n \ha_n - \hl_n \hb_n & \text{for } i \in \mathcal{S}_{0,\text{bad}}, \\
    \hl_n \alpha_n + \hl_n \beta_n & \text{for } i \in \mathcal{S}_{0,\text{good}}.
    \end{cases}
\end{equation}

For $y_{\pd}^{(2)}$ to achieve $100\%$ training accuracy, we require $\hl_n(\ha_n+\hb_n)>0.5$ and $\hl_n(\alpha_n+\beta_n)<0.5$.

To distinguish the affine scores from the thresholded classifiers in \eqref{eq:logistic_affine} and \eqref{eq:logistic_affine_2}, define
\[
s_{\m,\xi}(x):=(1-\xi)\hat y(x)+\xi y_{\pd}(x),
\qquad
s_{\m,\xi}^{(2)}(x):=(1-\xi)y_{\pd}(x)+\xi y_{\pd}^{(2)}(x).
\]

\subsection{Proof of \Cref{thm:class_mixing}}

\begin{proof}[Proof of \Cref{thm:class_mixing}]
We first show that when $\hl_n\to\hl\in(0,\infty)$, there exists $\rho_0\in(0,0.5)$ such that $y_{\pd}$ has $100(1-\rho)\%$ population accuracy for every $\rho\in(\rho_0,0.5)$, matching the teacher.

By \Cref{lem:alpha_form},
    \begin{align}
    \label{eq:equations_alpha_n_1}
\left\{
\begin{aligned}
\sigma \Big( cn(\alpha_n - (\alpha_n + \ha_n)\rho) - (1-c) \ha_n \Big) &= \hl_n \ha_n \\
\sigma \Big( cn(\alpha_n - (\alpha_n + \ha_n)\rho) + (1-c)\alpha_n \Big) &= 1 - \hl_n \alpha_n
\end{aligned}
\right.
\end{align}

The regularized logistic objective is strictly convex, and the feature spans for the two classes are orthogonal, so this system has a unique solution $(\alpha_n,\ha_n)$.
Because the sigmoid takes values in $(0,1)$, both $\alpha_n$ and $\ha_n$ are strictly positive.

Let $T_n:=\alpha_n-(\alpha_n+\ha_n)\rho$.
Suppose, for contradiction, that $\ha_n\leq\alpha_n$.
Using $\sigma(z)=1-\sigma(-z)$ and the monotonicity of the sigmoid,
\begin{align}
    &\sigma(cn T_n - (1-c)\ha_n) = \hl_n \ha_n \leq \hl_n \alpha_n = \sigma(- cn T_n - (1-c) \alpha_n) 
    \nonumber \\ 
    &\Rightarrow \quad
    cn T_n - (1 - c) \ha_n \leq - cn T_n - (1 - c) \alpha_n
    \nonumber \\
    &\Rightarrow \quad
    2 c n T_n \leq (1 - c)(\ha_n - \alpha_n) \leq 0.
\end{align}
This is impossible because $\rho<0.5$ and $\ha_n\leq\alpha_n$ imply $T_n>0$.
Hence, $\ha_n>\alpha_n$.
Moreover,
\[
\hl_n(\ha_n+\alpha_n)
=\sigma(cnT_n-(1-c)\ha_n)+1-\sigma(cnT_n+(1-c)\alpha_n)
\leq1,
\]
because $\alpha_n,\ha_n\geq0$.

Since $\hl_n\to\hl>0$ and $\hl_n\alpha_n,\hl_n\ha_n\in(0,1)$, the sequence $(\alpha_n,\ha_n)$ is bounded.
Fix $\rho\in(0,0.5)$ and let
\[
r:=\frac{\rho}{1-\rho}\in(0,1).
\]
Pass to an arbitrary convergent subsequence, still indexed by $n$, with limit $(\alpha,\ha)$.
Then $0\leq\alpha\leq\ha$, $\hl(\alpha+\ha)\leq1$, and
\[
T:=\lim_nT_n=\alpha-(\alpha+\ha)\rho.
\]
We claim that $T=0$.
\begin{itemize}[leftmargin=7mm]
    \item If $T>0$, both sigmoid arguments in \eqref{eq:equations_alpha_n_1} diverge to $+\infty$. Hence, $\hl\ha=1-\hl\alpha=1$, so $\ha>0$ and $\alpha=0$, contradicting $T>0$.

    \item If $T<0$, both arguments diverge to $-\infty$. Hence, $\hl\ha=1-\hl\alpha=0$, so $\ha=0$ and $\alpha=1/\hl>0$, contradicting $0 \leq \alpha\leq\ha$.

\end{itemize}
Therefore, $T=0$, and every subsequential limit satisfies
\begin{equation}
\label{eq:alpha_hat_alpha_relation}
\alpha=r\ha.
\end{equation}

The sequence $nT_n$ is also bounded.
Otherwise, pass to a subsequence along which it diverges to $+\infty$ or $-\infty$ and then to a further subsequence along which $(\alpha_n,\ha_n)$ converges.
In the first case, \eqref{eq:equations_alpha_n_1} gives $\hl\ha=1$ and $\hl\alpha=0$; in the second, it gives $\hl\ha=0$ and $\hl\alpha=1$.
Both conclusions contradict \eqref{eq:alpha_hat_alpha_relation} because $r>0$.

Consequently, every subsequence of $(\alpha_n,\ha_n,nT_n)$ has a further convergent subsequence.
Let $(\alpha,\ha,N)$ be any such subsequential limit.
Taking limits in \eqref{eq:equations_alpha_n_1} gives
\begin{align}
    \label{eq:equations_alpha_1}
\left\{
\begin{aligned}
\sigma \Big( cN - (1-c) \ha \Big) &= \hl \ha \\
\sigma \Big( cN + (1-c)\alpha \Big) &= 1 - \hl \alpha
\end{aligned}
\right.
\quad
\Leftrightarrow
\quad
\left\{
\begin{aligned}
\sigma \Big( cN - (1-c) \ha \Big) &= \hl \ha \\
\sigma \Big( -cN - (1-c)\alpha \Big) &= \hl \alpha.
\end{aligned}
\right.
\end{align}

We now show that this subsequential limit is unique.
Let $\operatorname{logit}(u):=\log\{u/(1-u)\}$.
The first equation in \eqref{eq:equations_alpha_1} implies $0<\hl\ha<1$, so the logit below is well defined.
The first equation in \eqref{eq:equations_alpha_1} gives
\[
cN=\operatorname{logit}(\hl\ha)+(1-c)\ha.
\]
Substituting this identity and $\alpha=r\ha$ into the second equation gives
\begin{equation}
\label{eq:alpha_scalar_limit}
g_r(\ha):=
\sigma\!\bigl(
-\operatorname{logit}(\hl\ha)-(1-c)(1+r)\ha
\bigr)
-\hl r\ha
=0.
\end{equation}
On $h\in(0,1/\hl)$, the function $g_r(h)$ is strictly decreasing, with
\[
\lim_{h\downarrow0}g_r(h)=1,
\qquad
\lim_{h\uparrow1/\hl}g_r(h)=-r.
\]
Thus, \eqref{eq:alpha_scalar_limit} has a unique solution, denoted $h_r$.
It follows that every subsequential limit is the same:
\[
\ha=h_r,
\qquad
\alpha=r h_r,
\qquad
N=N_r:=\frac{\operatorname{logit}(\hl h_r)+(1-c)h_r}{c}.
\]
Therefore, the full sequence $(\alpha_n,\ha_n,nT_n)$ converges to $(r h_r,h_r,N_r)$.

It remains to study the limit as $\rho\uparrow0.5$, or equivalently as $r\uparrow1$.
The unique zero $h_r$ of $g_r$ depends continuously on $r\in(0,1]$, because $g_r(h)$ is continuous in $(r,h)$ and strictly decreasing in $h$.
Indeed, if $r_k\to r$, the endpoint limits above ensure that every subsequential limit of $h_{r_k}$ lies in $(0,1/\hl)$ and solves $g_r(h)=0$; uniqueness then forces that limit to equal $h_r$.
At $r=1$, the continuous extension of the limiting equations in \eqref{eq:equations_alpha_1} becomes
\[
\sigma(cN_1-(1-c)h_1)=\hl h_1,
\qquad
\sigma(-cN_1-(1-c)h_1)=\hl h_1.
\]
The strict monotonicity of the sigmoid implies $N_1=0$, and the first equation then gives
\[
0<\hl h_1=\sigma(-(1-c)h_1)<0.5.
\]
Consequently,
\[
N_r \to 0,
\qquad
h_r \to h_1>0
\qquad\text{as }\rho\uparrow0.5.
\]
Hence, there exists $\rho_0\in(0,0.5)$ such that, for every $\rho\in(\rho_0,0.5)$,
\[
cN_r-(1-c)h_r<0,
\qquad
\hl h_r<0.5.
\]
Since $\alpha=r h_r<h_r$, we also have $\hl\alpha<0.5$.
Equation \eqref{eq:pd_prediction} therefore shows that $y_{\pd}$ misclassifies exactly the samples in $\mathcal{S}_{1,\mathrm{bad}}$ and $\mathcal{S}_{0,\mathrm{bad}}$, and hence has $100(1-\rho)\%$ population accuracy.

\textbf{Mixing achieves $100\%$ accuracy for all $\rho < 0.5$.} We now prove the second statement of the theorem. Under the assumed regime $\hl_n\to\hl\in(0,\infty)$, we can choose $\xi$ so that thresholding the affine score $(1-\xi)\hat y(x)+\xi y_{\pd}(x)$ yields $100\%$ population accuracy as $n\to\infty$. The limiting affine scores are
\[
s_{\m, \xi}(x_i) = (1 - \xi) \hat y_i + \xi y_{\pd}(x_i) =
\begin{cases}
\xi \hl\ha & \text{for } i \in \mathcal{S}_{1,\mathrm{bad}}, \\
1 - \xi \hl\alpha & \text{for } i \in \mathcal{S}_{1,\mathrm{good}}, \\
1 - \xi \hl \ha & \text{for } i \in \mathcal{S}_{0,\mathrm{bad}}, \\
\xi \hl\alpha & \text{for } i \in \mathcal{S}_{0,\mathrm{good}}.
\end{cases}
\]

Since $\ha>\alpha$, there are three algebraic cases:
\begin{itemize}[leftmargin=7mm]
    \item If $\hl \ha > 0.5 > \hl \alpha$, then \eqref{eq:pd_prediction} shows that the PD student already achieves $100\%$ accuracy as $n\to\infty$, so we choose $\xi=1$.
    \item If $0.5 \geq \hl \ha > \hl \alpha$, choose $\xi>1$ such that $\xi\hl\ha>0.5$ and $\xi\hl\alpha<0.5$. Such a choice exists because $\ha>\alpha$.
    \item The case $\hl \ha > \hl \alpha \geq 0.5$ is impossible because $\hl(\ha+\alpha)\leq1$.
\end{itemize}
\end{proof}

\subsection{Proof of \Cref{thm:class_mixing_2}}

\begin{proof}[Proof of \Cref{thm:class_mixing_2}]
As established in the proof of \Cref{thm:class_mixing},
\[
\alpha=r\ha,
\qquad
r:=\frac{\rho}{1-\rho}\in(0,1),
\qquad
\hl(\ha+\alpha)\leq1.
\]
In particular, $\ha>\alpha>0$.
Write
\[
P:=\hl\ha,
\qquad
Q:=\hl\alpha=rP.
\]

If $P>0.5$, then $Q<0.5$ because $P+Q\leq1$.
Thus, $y_{\pd}$ already achieves $100\%$ accuracy, and choosing $\xi=0$ recovers the classifier induced by $y_{\pd}$.
It remains to consider
\begin{equation}
\label{eq:second_student_nonperfect_case}
0<Q<P\leq0.5.
\end{equation}
This includes the boundary case $P=0.5$.

The equations in \eqref{eq:equations_beta_n} imply
\[
0<\hl_n(\ha_n+\hb_n)<1,
\qquad
0<\hl_n(\alpha_n+\beta_n)<1.
\]
Hence, $(\beta_n,\hb_n)$ is bounded.
Let $(\beta,\hb)$ be an arbitrary subsequential limit, and define
\[
W_n:=\beta_n-(\beta_n+\hb_n)\rho,
\qquad
W:=\lim_nW_n=(1-\rho)\beta-\rho\hb.
\]
We claim that $W=0$.
If $W>0$, both sigmoid arguments in \eqref{eq:equations_beta_n} diverge to $+\infty$, giving
\[
P+\hl\hb=1,
\qquad
Q+\hl\beta=0.
\]
Thus, $\hb>0$ and $\beta<0$, contradicting $W>0$.
If $W<0$, both arguments diverge to $-\infty$, giving
\[
P+\hl\hb=0,
\qquad
Q+\hl\beta=1.
\]
Thus, $\hb<0$ and $\beta>0$, contradicting $W<0$.
Therefore, every subsequential limit satisfies
\begin{equation}
\label{eq:beta_hat_beta_relation}
\beta=r\hb.
\end{equation}

The sequence $nW_n$ is bounded as well.
Otherwise, pass to a subsequence along which it diverges to $+\infty$ or $-\infty$ and then to a further subsequence along which $(\beta_n,\hb_n)$ converges.
The limiting equations would again force $\beta$ and $\hb$ to have opposite signs, contradicting \eqref{eq:beta_hat_beta_relation}.
Consequently, every subsequence of $(\beta_n,\hb_n,nW_n)$ has a further convergent subsequence.

Let $(\beta,\hb,M)$ be any such subsequential limit.
Using \eqref{eq:alpha_hat_alpha_relation} and \eqref{eq:beta_hat_beta_relation}, the limiting equations are
\begin{align}
\label{eq:beta_limiting_system}
\left\{
\begin{aligned}
\sigma(cM-(1-c)\hb)&=\hl(\ha+\hb),\\
\sigma(-cM-(1-c)r\hb)&=\hl r(\ha+\hb).
\end{aligned}
\right.
\end{align}
In particular, $0<\hl(\ha+\hb)<1$, so the logit below is well defined.
The first equation implies
\[
cM=\operatorname{logit}\!\big(\hl(\ha+\hb)\big)+(1-c)\hb.
\]
For $h\in(-\ha,1/\hl-\ha)$, define
\begin{equation}
\label{eq:beta_scalar_limit}
k_r(h):=
\sigma\!\bigl(
-\operatorname{logit}\!\big(\hl(\ha+h)\big)
-(1-c)(1+r)h
\bigr)
-\hl r(\ha+h).
\end{equation}
Substituting the expression for $M$ into the second equation in \eqref{eq:beta_limiting_system} shows that $k_r(\hb)=0$.
The function $k_r$ in \eqref{eq:beta_scalar_limit} is strictly decreasing, with
\[
\lim_{h\downarrow-\ha}k_r(h)=1,
\qquad
\lim_{h\uparrow1/\hl-\ha}k_r(h)=-r.
\]
Thus, $k_r(h)=0$ has a unique solution, denoted $b_r$.
Moreover,
\[
k_r(0)=1-P-Q>0,
\]
where the strict inequality follows from \eqref{eq:second_student_nonperfect_case} and $Q=rP<P$.
Therefore,
\[
b_r>0,
\qquad
\hb=b_r,
\qquad
\beta=r b_r<b_r.
\]
The value of $M$ is then uniquely determined by the first equation in \eqref{eq:beta_limiting_system}.
Hence, the full sequence $(\beta_n,\hb_n,nW_n)$ converges to this unique limit.

The limiting affine scores are
\[
s_{\m,\xi}^{(2)}(x_i)=
\begin{cases}
P+\xi\hl b_r & \text{for } i \in \mathcal{S}_{1,\mathrm{bad}}, \\
1-Q-\xi\hl r b_r & \text{for } i \in \mathcal{S}_{1,\mathrm{good}}, \\
1-P-\xi\hl b_r & \text{for } i \in \mathcal{S}_{0,\mathrm{bad}}, \\
Q+\xi\hl r b_r & \text{for } i \in \mathcal{S}_{0,\mathrm{good}}.
\end{cases}
\]
Choose $\xi$ such that
\[
\frac{0.5-P}{\hl b_r}
<\xi<
\frac{0.5-Q}{\hl r b_r}.
\]
This interval is nonempty because $P>Q$ and $b_r>r b_r>0$; its lower endpoint is zero when $P=0.5$.
This choice makes the first two scores exceed $0.5$ and the last two fall below $0.5$, proving the claim.
\end{proof}

\subsection{Proof of \Cref{thm:class_mixing_large_p}}

\begin{proof}[Proof of \Cref{thm:class_mixing_large_p}]
    For this proof, write $\lambda_n$ for the regularization parameter at sample size $n$, so that $\hl_n=2n\lambda_n$ and $\lambda_n=\Theta(1)$.
    Let $T_n:=\alpha_n-(\alpha_n+\ha_n)\rho$. The equations in \eqref{eq:equations_alpha_n} force $\alpha_n,\ha_n>0$. We first show that $\alpha_n>\ha_n$ when $\rho>0.5$. Suppose instead that $\alpha_n \leq \ha_n$. Using $\sigma(z) = 1 - \sigma(-z)$ and the monotonicity of the sigmoid, we have
    \begin{align}
        &\sigma(- cn T_n - (1-c)\alpha_n) = \hl_n \alpha_n \leq \hl_n \ha_n = \sigma(cn T_n - (1-c) \ha_n)
        \nonumber \\ 
        &\Rightarrow \quad
        -cn T_n - (1 - c) \alpha_n \leq  cn T_n - (1 - c) \ha_n
        \nonumber \\
        &\Rightarrow \quad
        2 c n T_n \geq (1 - c)(\ha_n - \alpha_n) \geq 0,
        \nonumber
    \end{align}
    which is a contradiction because $T_n = (1 - \rho) \alpha_n - \rho \ha_n < 0$ (recall that $\rho > 0.5$). Thus, we must have $\boxed{\alpha_n > \ha_n}$.

    Because $\lambda_n=\Theta(1)$, we have $\hl_n\to\infty$.
    Define the scaled coefficients
    \[
    A_n:=\hl_n\alpha_n,
    \qquad
    \widehat A_n:=\hl_n\ha_n,
    \qquad
    N_n:=nT_n.
    \]
    The equations in \eqref{eq:equations_alpha_n} imply $A_n,\widehat A_n\in(0,1)$, so $\alpha_n,\ha_n\to0$.
    Moreover,
    \begin{equation}
    \label{eq:large_alpha_scaled_relation}
    2\lambda_nN_n=(1-\rho)A_n-\rho\widehat A_n.
    \end{equation}
    Hence, $(A_n,\widehat A_n,N_n)$ is bounded.

    Starting from an arbitrary subsequence, pass to a further subsequence along which
    \[
    (A_n,\widehat A_n,N_n,\lambda_n)
    \to
    (A,\widehat A,N,\lambda_{\star}),
    \qquad
    \lambda_{\star}\in(0,\infty).
    \]
    Taking limits in \eqref{eq:equations_alpha_n} gives
    \begin{equation}
    \label{eq:equations_alpha_2}
    \widehat A=\sigma(cN),
    \qquad
    A=\sigma(-cN),
    \qquad
    A+\widehat A=1.
    \end{equation}
    Since $\alpha_n>\ha_n$, we have $A\geq\widehat A$.
    If $N>0$, then \eqref{eq:equations_alpha_2} gives $\widehat A>A$, a contradiction.
    If $N=0$, then $A=\widehat A=0.5$, whereas \eqref{eq:large_alpha_scaled_relation} gives
    \[
    0=\frac{1-2\rho}{2},
    \]
    again a contradiction because $\rho>0.5$.
    Therefore,
    \begin{equation}
    \label{eq:large_alpha_signs}
    N<0,
    \qquad
    \widehat A<0.5<A.
    \end{equation}

    The original subsequence was arbitrary, so every subsequential limit satisfies \eqref{eq:large_alpha_signs}.
    Consequently, for all sufficiently large $n$, $\widehat A_n<0.5<A_n$.
    By \eqref{eq:pd_prediction}, the first PD student therefore has asymptotic $0\%$ population accuracy.

    For the second student, define
    \[
    B_n:=\hl_n\beta_n,
    \qquad
    \widehat B_n:=\hl_n\hb_n,
    \qquad
    M_n:=n\big((1-\rho)\beta_n-\rho\hb_n\big).
    \]
    Equation \eqref{eq:equations_beta_n} implies
    \[
    0<\widehat A_n+\widehat B_n<1,
    \qquad
    0<A_n+B_n<1,
    \]
    so $(B_n,\widehat B_n)$ is bounded and $\beta_n,\hb_n\to0$.
    We also have the exact identity
    \begin{equation}
    \label{eq:large_beta_scaled_relation}
    2\lambda_nM_n=(1-\rho)B_n-\rho\widehat B_n,
    \end{equation}
    which shows that $M_n$ is bounded.

    Passing to a further subsequence if necessary, suppose that
    \[
    (B_n,\widehat B_n,M_n)\to(B,\widehat B,M).
    \]
    Taking limits in \eqref{eq:equations_beta_n} gives
    \begin{equation}
    \label{eq:equations_beta}
    \widehat A+\widehat B=\sigma(cM),
    \qquad
    A+B=\sigma(-cM).
    \end{equation}
    Combining these equations with $A+\widehat A=1$ yields
    \begin{equation}
    \label{eq:large_beta_sum_zero}
    B+\widehat B=0.
    \end{equation}
    Taking limits in \eqref{eq:large_beta_scaled_relation} and using \eqref{eq:large_beta_sum_zero} gives
    \begin{equation}
    \label{eq:large_M_relation}
    2\lambda_{\star}M=B.
    \end{equation}

    We claim that $B<0<\widehat B$.
    If $B>0$, then \eqref{eq:large_M_relation} gives $M>0$.
    On the other hand, $A+B>A>0.5$, so \eqref{eq:equations_beta} gives $M<0$, a contradiction.
    If $B=0$, then $\widehat B=M=0$, and \eqref{eq:equations_beta} gives $A=\widehat A=0.5$, contradicting \eqref{eq:large_alpha_signs}.
    Therefore, $B<0<\widehat B$, and \eqref{eq:large_M_relation} gives $M<0$.
    It follows from \eqref{eq:equations_beta} that
    \begin{equation}
    \label{eq:large_second_student_signs}
    \widehat A+\widehat B<0.5<A+B.
    \end{equation}
    Because the original subsequence was arbitrary, for all sufficiently large $n$,
    $\widehat A_n+\widehat B_n<0.5<A_n+B_n$.
    Equation \eqref{eq:pd_2_prediction} therefore shows that the second PD student also has asymptotic $0\%$ population accuracy.

    At any subsequential limit, the affine scores are
    \[
    s_{\m,\xi}^{(2)}(x_i)=
    \begin{cases}
    \widehat A+\xi\widehat B & \text{for } i \in \mathcal{S}_{1,\mathrm{bad}}, \\
    1-A-\xi B & \text{for } i \in \mathcal{S}_{1,\mathrm{good}}, \\
    1-\widehat A-\xi\widehat B & \text{for } i \in \mathcal{S}_{0,\mathrm{bad}}, \\
    A+\xi B & \text{for } i \in \mathcal{S}_{0,\mathrm{good}}.
    \end{cases}
    \]
    Thus, it suffices to choose
    \begin{equation}
    \label{eq:large_mixing_lower_bound}
    \xi>
    \max\biggl\{
    \frac{0.5-\widehat A}{\widehat B},
    \frac{A-0.5}{-B}
    \biggr\}.
    \end{equation}
    Both ratios in \eqref{eq:large_mixing_lower_bound} exceed $1$ because
    \[
    \widehat A+\widehat B<0.5<A+B.
    \]

    Finally, let $\mathcal K$ denote the compact set of all subsequential limits of $(A_n,\widehat A_n,B_n,\widehat B_n)$.
    On $\mathcal K$, we have $\widehat B>0$ and $-B>0$, so the two ratios in \eqref{eq:large_mixing_lower_bound} are continuous and have a finite maximum.
    Choosing $\xi$ larger than this maximum gives $\xi>1$ and makes all four groups classify correctly for all sufficiently large $n$.
    Hence, the PMSD student has asymptotic $100\%$ population accuracy.
\end{proof}

\section{Proof in \Cref{sec:extensions}}

\subsection{Proof of \Cref{prop:optimize_Sigma_s}}

\begin{proof}[Proof of \Cref{prop:optimize_Sigma_s}]
Let
\[
    W := \Sigma_s^{-1}
    =
    U^{\top}\operatorname{diag}(w_1,\ldots,w_p)U.
\]

For \(a\in\{t,s\}\), the fixed-point equation can be rewritten as
\[
    \lambda_a
    =
    \kappa_a
    \bigl(
        1
        -
        \gamma_a
        \otr\!\bigl(
            \Sigma_a(\Sigma_a+\kappa_a I_p)^{-1}
        \bigr)
    \bigr).
\]
Because
\[
0\leq
\otr\!\bigl(\Sigma_a(\Sigma_a+\kappa_a I_p)^{-1}\bigr)
\leq1
\]
and $\gamma_a<1$, the factor in parentheses is bounded below by $1-\gamma_a>0$. Hence,
\[
0<\kappa_a\leq\frac{\lambda_a}{1-\gamma_a},
\]
so
\[
    \kappa_t,\kappa_s\to 0
    \qquad
    \text{as }
    \lambda_t,\lambda_s\to0.
\]
It follows that
\[
    G_t\to\Sigma_t^{-1},
    \qquad
    G_s\to\Sigma_s^{-1}=W,
    \qquad
    b_t\to\frac{1}{1-\gamma_t}.
\]
Moreover,
\begin{align}
    b_s
    &=
    \frac{
        \gamma_s\otr(\Sigma_s\Sigma_tG_s^2)
    }{
        1-\gamma_s\otr(\Sigma_s^2G_s^2)
    } 
    \to
    \frac{\gamma_s}{1-\gamma_s}
    \otr(\Sigma_tW).
\end{align}

Recall that the risk reduction has the form $(\sR_{\te}-\sC)^2/\sD$. Substituting these limits into the expressions for
\(\sR_{\te}-\sC\) and $\sD$
in \Cref{thm:general_ridge}, we obtain
\begin{align}
    \label{eq:optimize_Sigma_s_R-C}
        \lim_{\lambda_t,\lambda_s\to0}
    \frac{\sR_{\te}-\sC}{\kappa_s}
    =
    \frac{\sigma^2 \gamma_t}{1 - \gamma_t}
    \otr(W),
\end{align}
and
\begin{align}
    \label{eq:optimize_Sigma_s_D}
    \lim_{\lambda_t,\lambda_s\to0}
    \frac{\sD}{\kappa_s^2}
    &=
    \beta^{\top}W\Sigma_tW\beta
    +
    \frac{\gamma_s}{1 - \gamma_s} \otr(\Sigma_tW)\beta^{\top}W\beta \nonumber \\
    &\quad +
    \frac{\sigma^2 \gamma_t}{1 - \gamma_t}
    \Bigl\{
        \otr(W^2)
        +
        \frac{\gamma_s}{1 - \gamma_s}
        \otr(\Sigma_tW)
        \otr(\Sigma_t^{-1}W)
    \Bigr\}.
\end{align}

After canceling the common factor $\kappa_s^2$, the limiting risk reduction is the square of the right-hand side of \eqref{eq:optimize_Sigma_s_R-C} divided by the right-hand side of \eqref{eq:optimize_Sigma_s_D}.

We next express the right-hand side of \eqref{eq:optimize_Sigma_s_D} in terms of \(w\). In the common eigenbasis of \(\Sigma_t\) and \(\Sigma_s\), we have
\[
    \beta^{\top}W\Sigma_tW\beta
    =
    \sum_{i=1}^p \sigma_i v_i w_i^2,
    \qquad
    \beta^{\top}W\beta
    =
    \sum_{i=1}^p v_iw_i,
\]
and
\[
    \otr(\Sigma_tW)
    =
    \frac{1}{p}\sum_{i=1}^p\sigma_iw_i,
    \qquad
    \otr(W^2)
    =
    \frac{1}{p}\sum_{i=1}^pw_i^2,
    \qquad 
    \otr(\Sigma_t^{-1}W)
    =
    \frac{1}{p}\sum_{i=1}^p\frac{w_i}{\sigma_i}.
\]

Substituting these identities into \eqref{eq:optimize_Sigma_s_D} gives
\[
\begin{aligned}
    &\sum_{i=1}^p
    \biggl(
        \sigma_iv_i+\frac{\sigma^2 \gamma_t}{p(1 - \gamma_t)}
    \biggr)w_i^2
    +
    \frac{\gamma_s}{p(1-\gamma_s)}
    \biggl(
        \sum_{i=1}^p\sigma_iw_i
    \biggr)
    \biggl(
        \sum_{i=1}^pv_iw_i
    \biggr)
    +
    \frac{\sigma^2 \gamma_t}{p(1 - \gamma_t)} 
    \cdot
    \frac{\gamma_s}{p(1-\gamma_s)}
    \biggl(
        \sum_{i=1}^p\sigma_iw_i
    \biggr)
    \biggl(
        \sum_{i=1}^p\frac{w_i}{\sigma_i}
    \biggr).
\end{aligned}
\]

By the definitions of \(S\), \(t\), and \(\theta\), this expression equals
\[
w^{\top}
        \bigl(
            S+t\theta^{\top}+\theta t^{\top}
        \bigr)w.
\]

The squared numerator determined by \eqref{eq:optimize_Sigma_s_R-C} and the denominator in \eqref{eq:optimize_Sigma_s_D} are both homogeneous of degree two in
\(w\). We may therefore impose the normalization $\mathbf 1_{p}^{\top} w = p$, under which $\otr(W)=1$ and the squared numerator equals the positive constant $\tau_t^2$. Hence, maximizing the limiting risk reduction is equivalent to
\[
\begin{aligned}
    \underset{w\in\mathbb R_{+}^p}{\operatorname{minimize}}
    \quad&
    w^{\top}
    \bigl(
        S+t\theta^{\top}+\theta t^{\top}
    \bigr)w,\\
    \operatorname{subject\ to}
    \quad&
    \mathbf 1_p^{\top}w=p.
\end{aligned}
\]
\end{proof}

\section{Additional theoretical results}
\label{sec:additional_theory}

\subsection{Non-monotonicity in the fresh unlabeled sample size}

\label{sec:appendix_monotonicity}

In this subsection, we consider the isotropic and same-$\lambda$ setting, i.e., $\lambda_t = \lambda_s = \lambda$. We demonstrate that the optimal fresh-$X$ PMSD risk is \emph{not} globally monotone in $\gamma_s$. Instead, it has a simple one-dimensional structure and is unimodal in $\gamma_s$.

Assume $\Sigma_t = \Sigma_s = I_p$ and $\beta\sim\cN(0,(r^2/p)I_p)$, where $r^2>0$, and let
\[
  \lambda^\star:=\gamma_t\frac{\sigma^2}{r^2}.
\]
For the labeled sample, let $\kappa_t=\kappa_t(\lambda)$ be the unique positive solution to
\[
  \kappa_t
  =
  \lambda+\frac{\gamma_t\kappa_t}{1+\kappa_t},
\]
and set
\[
  b_t:=\Big(1-\frac{\gamma_t}{(1+\kappa_t)^2}\Big)^{-1}.
\]
For the fresh unlabeled sample, let $\kappa_s=\kappa_s(\lambda)$ be the unique positive solution to
\[
  \kappa_s
  =
  \lambda+\frac{\gamma_s\kappa_s}{1+\kappa_s},
\]
and define
\[
  c_s:=\Big(1-\frac{\gamma_s}{(1+\kappa_s)^2}\Big)^{-1}.
\]

Let $\bar \sR_{\te}(\lambda):=\sR_{\te}(\lambda)-\sigma^2$ denote the excess teacher risk. Define the excess PD-student risk $\bar \sR_{\pd}$ and the excess correlation $\bar \sC$ similarly.

In the isotropic model, the limiting teacher risk in \eqref{eq:teacher_risk} simplifies to
\begin{equation}
\label{eq:barR_teacher_iso_again}
  \bar \sR_{\te}(\lambda)
  =
  \frac{b_t}{(1+\kappa_t)^2}
  \big(r^2\kappa_t^2+\sigma^2\gamma_t\big).
\end{equation}
Define also
\begin{equation}
\label{eq:Theta_T_labeled}
  \Theta_t(\lambda)
  :=
  \bar \sR_{\te}(\lambda)-r^2\frac{\kappa_t}{1+\kappa_t},
  \qquad
  T_t(\lambda)
  :=
  \bar \sR_{\te}(\lambda)+r^2\frac{1-\kappa_t}{1+\kappa_t}.
\end{equation}

We are now ready to state the key result in this subsection.

\begin{proposition}[Dependence of optimal PMSD risk on $\gamma_s$]
\label{prop:freshX_unlabeled_effect_iso}
In the isotropic setting, for each fixed $\lambda>0$,
\begin{equation}
\label{eq:Rsd_fresh_gamma_u_formula}
  \bar \sR_{\m}^{\star}(\lambda;\gamma_s)
  :=
  \sR_{\m}^{\star}(\lambda;\gamma_s) - \sigma^2
  =
  \bar \sR_{\te}(\lambda)
  -
  \frac{\Theta_t(\lambda)^2}{T_t(\lambda)}\,
  \frac{1}{c_s(\lambda)}.
\end{equation}
Equivalently,
\begin{equation}
\label{eq:Rsd_fresh_gamma_u_formula_K}
  \bar \sR_{\m}^{\star}(\lambda;\gamma_s)
  =
  \bar \sR_{\te}(\lambda)-K_t(\lambda)\,\frac{1}{c_s(\lambda)},
  \qquad
  K_t(\lambda):=\frac{\Theta_t(\lambda)^2}{T_t(\lambda)}\ge 0.
\end{equation}

Moreover:
\begin{enumerate}[label=(\alph*),leftmargin=7mm]
  \item if $\lambda=\lambda^\star$, then $K_t(\lambda)=0$ and
  \[
    \sR_{\m}^{\star}(\lambda;\gamma_s) \equiv \sR_{\te}(\lambda)
    \qquad\text{for all }\gamma_s>0;
  \]
  \item if $\lambda\neq \lambda^\star$, then
  \begin{equation}
  \label{eq:dRsd_dgamma_u}
    \frac{\partial}{\partial \gamma_s} \bar \sR_{\m}^{\star}(\lambda;\gamma_s)
    =
    -K_t(\lambda)\,
    \frac{
      4(\gamma_s-\lambda-1)
    }{\big(\gamma_s+\lambda+1+\sqrt{(\gamma_s+\lambda-1)^2+4\lambda}\big)^2
      \sqrt{(\gamma_s+\lambda-1)^2+4\lambda}
    }.
  \end{equation}
  Hence $\gamma_s\mapsto \bar \sR_{\m}^{\star}(\lambda;\gamma_s)$ is
  strictly increasing on $(0,1+\lambda)$, strictly decreasing on $(1+\lambda,\infty)$,
  and has a unique maximum at
  \[
    \gamma_s=1+\lambda.
  \]
\end{enumerate}
In particular, the optimal PMSD risk is \emph{not} globally monotone in the amount of unlabeled data.
\end{proposition}

Since smaller $\gamma_s=p/n_s$ corresponds to more unlabeled data, \Cref{prop:freshX_unlabeled_effect_iso}
shows that adding more fresh unlabeled data is \emph{not} uniformly beneficial.
For fixed $\lambda\neq\lambda^\star$:
\begin{itemize}[leftmargin=7mm]
  \item if $\gamma_s<1+\lambda$, then decreasing $\gamma_s$ (adding more unlabeled data) decreases the optimal PMSD risk;
  \item if $\gamma_s>1+\lambda$, then decreasing $\gamma_s$ initially \emph{increases} the optimal PMSD risk, until the turning point $\gamma_s=1+\lambda$ is reached.
\end{itemize}
Thus, the fresh-$X$ PMSD risk is unimodal, not monotone, in the amount of unlabeled data.

\begin{proof}[Proof of \Cref{prop:freshX_unlabeled_effect_iso}]
Let
\[
  \Delta_\lambda:=\beta_\lambda^{\te}-\beta,
  \qquad
  \Delta_{\pd,\lambda}:=\beta_{\pd,\lambda} - \beta.
\]
All limits below are taken in the proportional asymptotic regime with $\gamma_t$ and $\gamma_s$ fixed.
In the isotropic in-distribution setting,
\[
  \bar \sR_{\te}(\lambda)=\lim \E\|\Delta_\lambda\|_2^2,
  \qquad
  \bar \sC(\lambda)=\lim \E\langle \Delta_\lambda,\Delta_{\pd,\lambda}\rangle,
  \qquad
  \bar \sR_{\pd}(\lambda)=\lim \E\|\Delta_{\pd,\lambda}\|_2^2.
\]

Let $M_u:=\widetilde\Sigma(\widetilde\Sigma+\lambda I_p)^{-1}$. Then
\[
\beta_{\pd, \lambda}=M_u \beta_\lambda^{\te},
  \qquad
  \Delta_{\pd,\lambda}
  =
  M_u\Delta_\lambda+(M_u-I_p)\beta.
\]
Because $\widetilde\Sigma$ is independent of the labeled sample, $M_u$ is independent of
$(\Delta_\lambda,\beta)$.
Under isotropy, $M_u$ is orthogonally invariant, and therefore
\[
  \Delta_\lambda^\top M_u\Delta_\lambda
  \to s_u\,\|\Delta_\lambda\|_2^2,
  \qquad
  \Delta_\lambda^\top M_u^2\Delta_\lambda
  \to s_{2,u}\,\|\Delta_\lambda\|_2^2,
\]
\[
  \Delta_\lambda^\top M_u\beta
  \to s_u\,\langle \Delta_\lambda,\beta\rangle,
  \qquad
  \beta^\top(M_u-I_p)^2\beta
  \to (1-2s_u+s_{2,u})\,\|\beta\|_2^2,
\]
where
\[
  s_u=\lim\frac1p\tr(M_u)=\frac{1}{1+\kappa_s},
  \quad
  s_{2,u} = \lim \frac1p\tr(M_u^2),
\]
and
\[
  1-2s_u+s_{2,u}
  =
  \lim \frac1p\tr\!\big((I_p-M_u)^2\big)
  =
  \frac{\kappa_s^2 c_s}{(1+\kappa_s)^2}.
\]

Similarly, on the labeled side,
\[
  s_t=\frac{1}{1+\kappa_t},
  \qquad
  \langle \Delta_\lambda,\beta\rangle
  \to
  -r^2(1-s_t)
  =
  -r^2\frac{\kappa_t}{1+\kappa_t}.
\]
Consequently,
\[
  T_t(\lambda)
  =
  \lim \E\|\beta_\lambda^{\te}\|_2^2
  >0.
\]

Expanding the fresh-$X$ correlation term gives
\[
  \bar \sC(\lambda)
  =
  s_u\,\bar \sR_{\te}(\lambda)
  +(s_u-1)\lim \langle \Delta_\lambda,\beta\rangle
  =
  s_u\,\bar \sR_{\te}(\lambda)+r^2(1-s_u)(1-s_t).
\]
Therefore,
\begin{align}
\label{eq:A_minus_C_fresh_unlabeled}
  \bar \sR_{\te}(\lambda)-\bar \sC(\lambda)
  &=
  (1-s_u)(\bar \sR_{\te}(\lambda)-r^2(1-s_t))
  \nonumber\\
  &=
  (1-s_u)\,\Theta_t(\lambda).
\end{align}

Next, expanding the fresh-$X$ PD risk gives
\[
  \bar \sR_{\pd}(\lambda)
  =
  s_{2,u}\bar \sR_{\te}(\lambda)
  +r^2\big(1-s_{2,u}+2s_t(s_{2,u}-s_u)\big).
\]
Hence the discrepancy denominator is
\begin{align*}
  \bar \sD(\lambda)
  &:=
  \bar \sR_{\te}(\lambda)+\bar \sR_{\pd}(\lambda)-2\bar \sC(\lambda)
  \\
  &=
  (1-2s_u+s_{2,u})
  (\bar \sR_{\te}(\lambda)+r^2(2s_t-1))
  \\
  &=
  (1-2s_u+s_{2,u})\,T_t(\lambda).
\end{align*}
Using
\[
  1-s_u=\frac{\kappa_s}{1+\kappa_s},
  \qquad
  1-2s_u+s_{2,u}=\frac{\kappa_s^2 c_s}{(1+\kappa_s)^2},
\]
we obtain
\[
  \frac{(\bar \sR_{\te}(\lambda)-\bar \sC(\lambda))^2}{\bar \sD(\lambda)}
  =
  \frac{\Theta_t(\lambda)^2}{T_t(\lambda)}\cdot \frac{1}{c_s(\lambda)}.
\]
Substituting this into the oracle decomposition
\[
  \bar \sR_{\m}^{\star}(\lambda)
  =
  \bar \sR_{\te}(\lambda)
  -
  \frac{(\bar \sR_{\te}(\lambda)-\bar \sC(\lambda))^2}{\bar \sD(\lambda)}
\]
proves \eqref{eq:Rsd_fresh_gamma_u_formula}.

To identify the flat case, note that
\[
  \Theta_t(\lambda)
  =
  \bar \sR_{\te}(\lambda)-r^2(1-s_t)
  =
  r^2\lambda^2 m_{2,t}(\lambda)
  +\sigma^2\gamma_t\big(m_{1,t}(\lambda)-\lambda m_{2,t}(\lambda)\big)
  -r^2\lambda m_{1,t}(\lambda),
\]
where
\[
  m_{k,t}(\lambda)
  :=
  \int \frac{1}{(u+\lambda)^k}\,dF_{\mathrm{MP},\gamma_t}(u),
  \qquad k\in\{1,2\},
\]
are the labeled resolvent moments. Rearranging,
\[
  \Theta_t(\lambda)
  =
  (\sigma^2\gamma_t-r^2\lambda)\big(m_{1,t}(\lambda)-\lambda m_{2,t}(\lambda)\big).
\]
Since
\[
  m_{1,t}(\lambda)-\lambda m_{2,t}(\lambda)
  =
  \int \frac{u}{(u+\lambda)^2}\,dF_{\mathrm{MP},\gamma_t}(u)
  >0,
\]
we have $\Theta_t(\lambda)=0$ if and only if
\[
  \lambda=\lambda^\star:=\gamma_t \frac{\sigma^2}{r^2}.
\]
This proves part (a).

For part (b), when $\lambda\neq\lambda^\star$ we have $K_t(\lambda)>0$, so it remains to differentiate
$1/c_s$.
Using the explicit isotropic root
\[
  \kappa_s
  =
  \frac{\lambda+\gamma_s-1+\sqrt{(\lambda+\gamma_s-1)^2+4\lambda}}{2},
\]
we obtain
\[
  \frac{1}{c_s}
  =
  1-\frac{\gamma_s}{(1+\kappa_s)^2}
  =
  1-\frac{4\gamma_s}{
    \big(\gamma_s+\lambda+1+\sqrt{(\gamma_s+\lambda-1)^2+4\lambda}\big)^2
  }.
\]
Direct differentiation yields
\[
  \frac{\partial}{\partial\gamma_s}\Big(\frac{1}{c_s}\Big)
  =
  \frac{
    4(\gamma_s-\lambda-1)
  }{
    \big(\gamma_s+\lambda+1+\sqrt{(\gamma_s+\lambda-1)^2+4\lambda}\big)^2
    \sqrt{(\gamma_s+\lambda-1)^2+4\lambda}
  }.
\]
Substituting into
\[
  \frac{\partial}{\partial\gamma_s}\bar \sR_{\m}^{\star}(\lambda;\gamma_s)
  =
  -K_t(\lambda)\frac{\partial}{\partial\gamma_s}\Big(\frac{1}{c_s}\Big)
\]
proves \eqref{eq:dRsd_dgamma_u} and the sign characterization.
\end{proof}

\subsection{Non-identifiability of tuning without labeled data}

\label{sec:appendix_tuning}

Recall that, for squared prediction risk with test covariance $\Sigma\succ 0$, the oracle weight is
\[
  \xi^\star
  =
  \frac{R_{\te}-C}{D}
  =
  \frac{\langle \beta^{\te}_{\lambda_t}-\beta,\ \beta^{\te}_{\lambda_t}-\beta_{\pd, \lambda_s}\rangle_{\Sigma}}
       {\|\beta^{\te}_{\lambda_t}-\beta_{\pd, \lambda_s}\|_{\Sigma}^2},
\]
provided $D(\lambda_t, \lambda_s)>0$.
Thus, for known $\Sigma$, the denominator depends only on the observable pair
$(\beta^{\te}_{\lambda_t},\beta_{\pd, \lambda_s})$, whereas the numerator still depends on the unknown target
parameter $\beta$.

\begin{proposition}[Non-identifiability of $\xi^\star$ without labels]
\label{prop:no_label_tuning_impossible}
Fix $\lambda_t,\lambda_s>0$, a teacher ridge predictor $\beta^{\te}_{\lambda_t}\in\R^p$, and a fresh unlabeled design
$\tilde X\in\R^{n_s \times p}$. Let
\[
  \beta_{\pd, \lambda_s}
  := (\tilde X^\top \tilde X / n_s + \lambda_s I_p)^{-1}(\tilde X^\top \tilde X / n_s )\beta^{\te}_{\lambda_t},
  \qquad
  \delta:=\beta^{\te}_{\lambda_t} -\beta_{\pd, \lambda_s}.
\]
Assume $\|\delta\|_{\Sigma}^2>0$.
For any $\eta\in\R$, define a test model by
\[
  x_0\sim P_x \text{ with } \E[x_0 x_0^\top]=\Sigma,
  \qquad
  y_0 = x_0^\top \beta_{\eta} + \varepsilon_0,
  \qquad
  \beta_{\eta}:=\beta^{\te}_{\lambda_t}-\eta\,\delta,
\]
where $\varepsilon_0$ is independent of $x_0$ with mean $0$ and variance $\sigma^2$.
Then, even if $\Sigma$ is known, the observable pair $(\beta^{\te}_{\lambda_t},\tilde X)$ is the same for every $\eta$, whereas
\[
  \xi^\star_{\eta}(\lambda_t,\lambda_s)=\eta,
  \qquad
  R_{\m,\eta}^{\star}=\sigma^2.
\]
In particular, $\xi^\star(\lambda_t, \lambda_s)$ is not identifiable from $(\beta^{\te}_{\lambda_t},\tilde X)$.
\end{proposition}

As an immediate corollary, there is no uniformly consistent estimator of $\xi^\star$ based only on the
teacher and unlabeled sample. Indeed, any measurable rule
$\hat\xi=\hat\xi(\beta^{\te}_{\lambda_t},\tilde X)$ takes the same value for the setting in
\Cref{prop:no_label_tuning_impossible}, whereas the target $\xi^\star_{\eta}(\lambda_t,\lambda_s)=\eta$ can be changed
arbitrarily.

\begin{proof}[Proof of \Cref{prop:no_label_tuning_impossible}]
    Since $\beta^{\te}_{\lambda_t}$ and $\tilde X$ are fixed, the induced pure-distilled coefficient
$\beta_{\pd, \lambda_s}$ is also fixed, so the observables do not depend on $\eta$.
Now note that
\[
  \beta^{\te}_{\lambda_t} -\beta_{\eta}=\eta\,\delta,
  \qquad
  \beta_{\pd, \lambda_s} -\beta_{\eta}=(\eta-1)\delta.
\]
Therefore,
\[
  R_{\te,\eta}(\lambda_t)-\sigma^2 = \eta^2\|\delta\|_{\Sigma}^2,
  \qquad
  C_{\eta}(\lambda_t, \lambda_s)-\sigma^2 = \eta(\eta-1)\|\delta\|_{\Sigma}^2,
\]
and hence
\[
   R_{\te,\eta}(\lambda_t) -C_{\eta}(\lambda_t, \lambda_s)
  = \eta\|\delta\|_{\Sigma}^2,
  \qquad
  D_{\eta}(\lambda_t, \lambda_s) = \|\delta\|_{\Sigma}^2.
\]
This gives $\xi^\star_{\eta}(\lambda_t,\lambda_s)=\eta$. Finally, the optimally mixed coefficient satisfies
\[
  (1-\xi_\eta^\star)\beta^{\te}_{\lambda_t}
  +\xi_\eta^\star\beta_{\pd,\lambda_s}
  =(1-\eta)\beta^{\te}_{\lambda_t}
  +\eta \beta_{\pd, \lambda_s}
  = \beta_{\eta},
\]
so at $\xi=\eta$ the predictor matches the truth exactly and the risk is the irreducible noise level $\sigma^2$.
\end{proof}

\clearpage

\section{Additional experiments}
\label{sec:additional_exp}

Code for reproducing the results is available at \url{https://github.com/hhd357/prediction_only_mixing_distillation}.

\vspace{-0.5em}
\subsection{Real-world regression tasks with different student regularization levels}
\label{sec:additional_exp_regression}

We report additional ridge regression results on Communities and Crime and Blog Feedback. For each dataset, we compare fresh covariates drawn from the teacher's training distribution with fresh covariates drawn from an isotropic Gaussian distribution, fixing or tuning $\lambda_s$ as in the captions.
\vspace{-1em}
\subsubsection{Communities and Crime}

\begin{figure*}[!ht]
   \centering
    \begin{subfigure}[t]{0.48\textwidth}
    \centering
    \includegraphics[width=\textwidth]{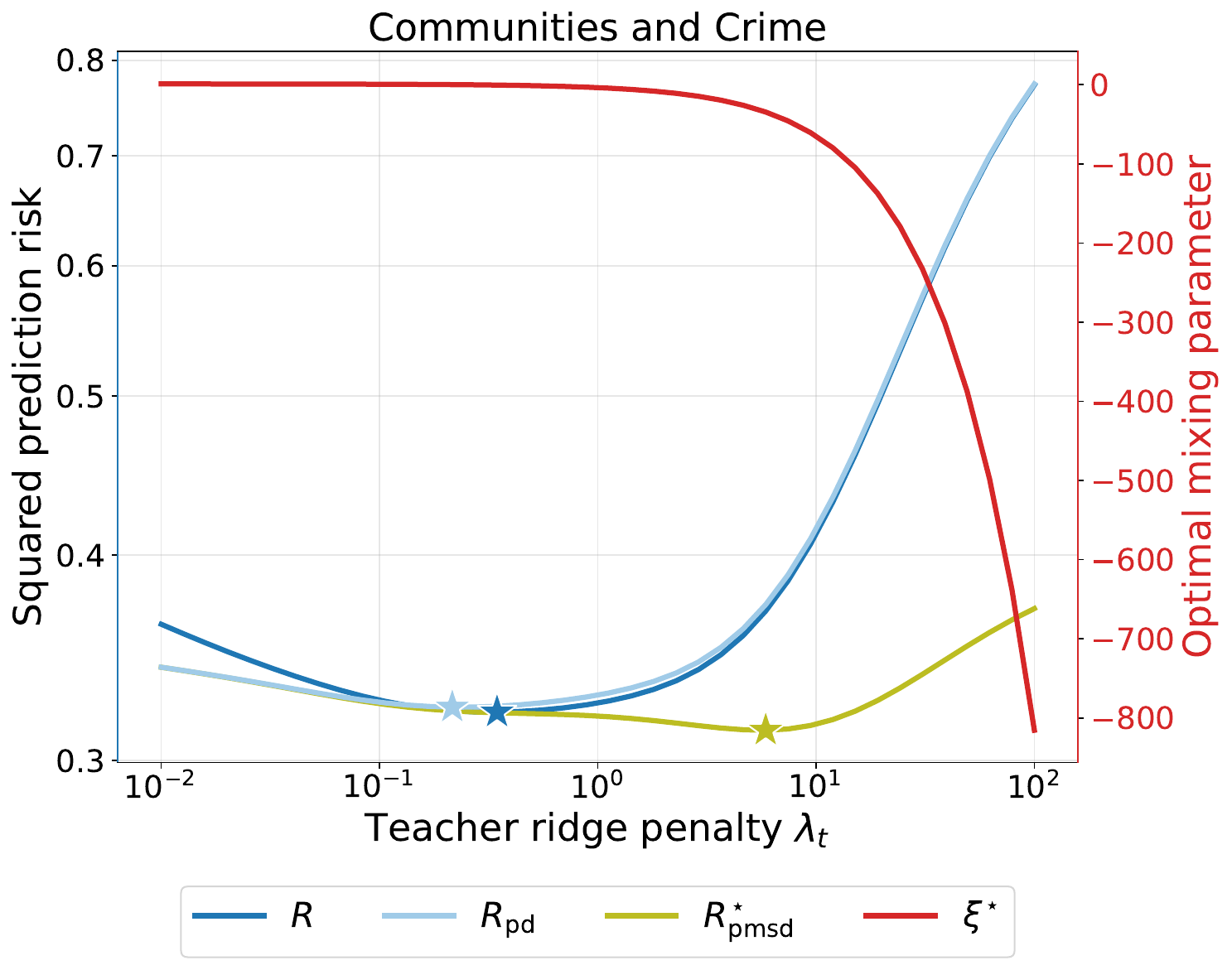}
    \caption{Fresh unlabeled covariates are drawn from the same distribution as the teacher's training covariates.}
 
  \end{subfigure}
  \quad
    \begin{subfigure}[t]{0.48\textwidth}
    \centering
    \includegraphics[width=\textwidth]{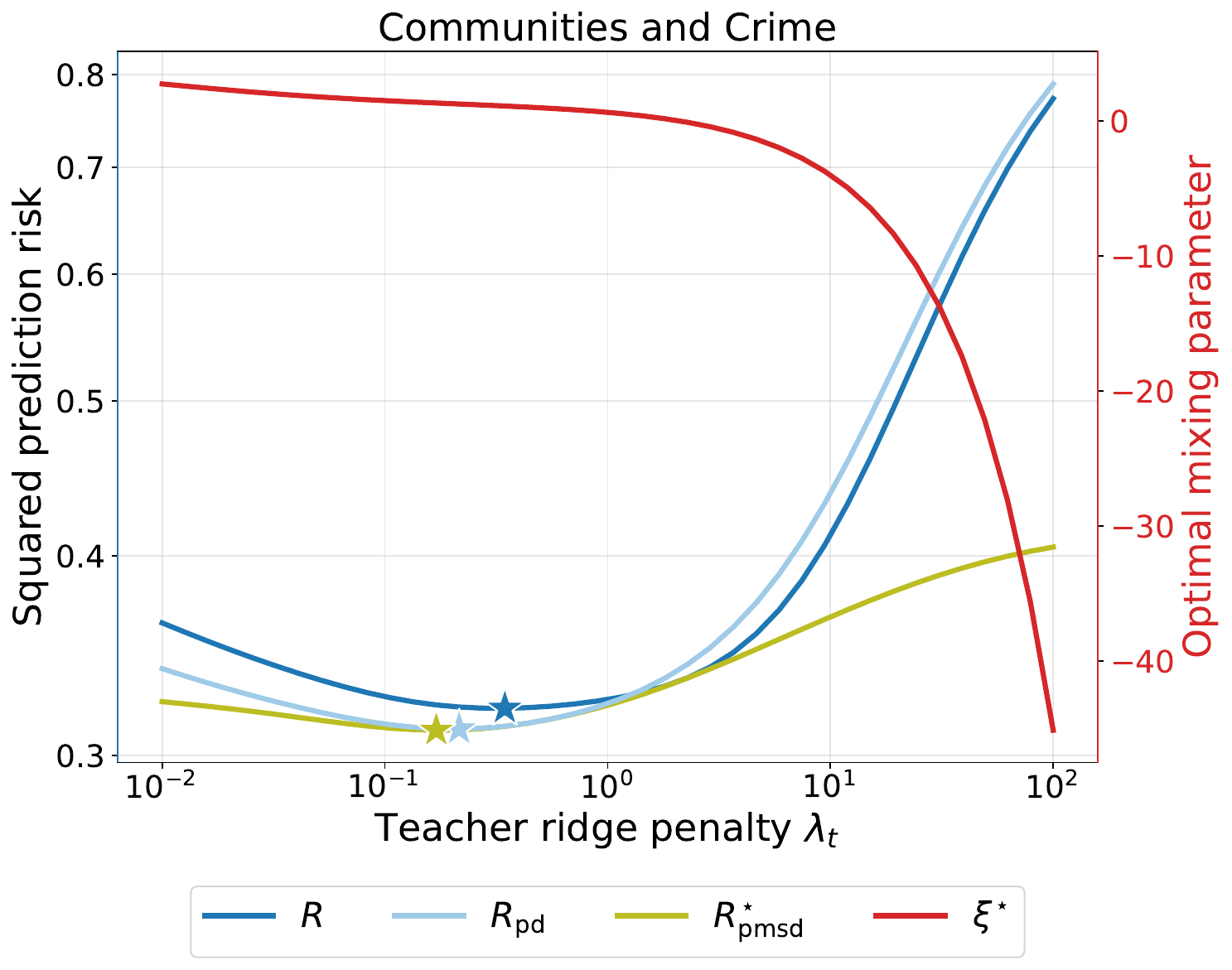}
    \caption{Fresh unlabeled covariates are drawn i.i.d.\ from $\mathcal{N}(0, I_p)$.}
  \end{subfigure}
    \caption{Communities and Crime with $n_t = 400$, $n_s = 800$, and $p = 99$. The student regularization is fixed at $\lambda_s = 0.1$ for all values of $\lambda_t$.}
    \vspace{-1em}
\end{figure*}

\begin{figure*}[!ht]
   \centering
    \begin{subfigure}[t]{0.48\textwidth}
    \centering
    \includegraphics[width=\textwidth]{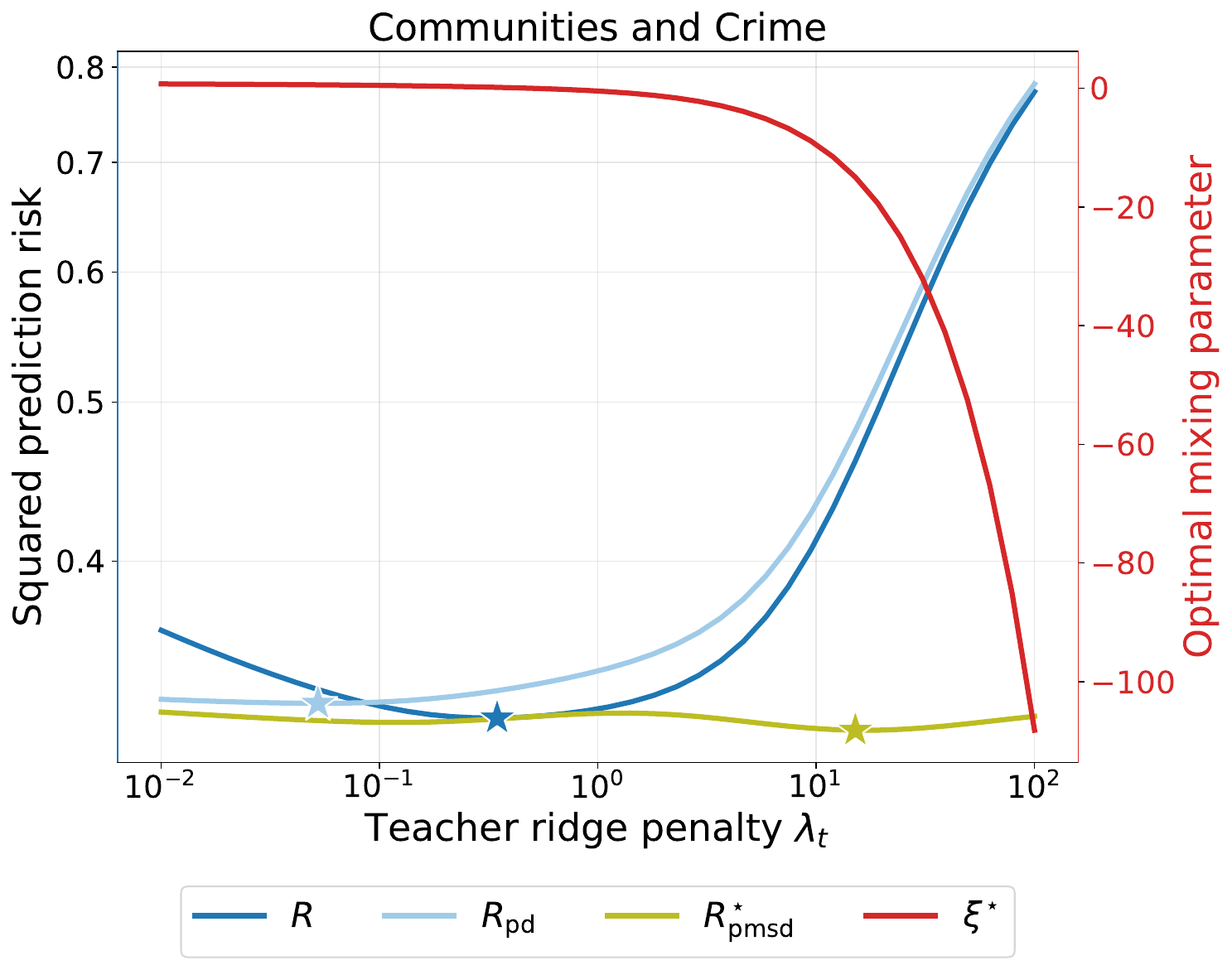}
    \caption{Fresh unlabeled covariates are drawn from the same distribution as the teacher's training covariates.}
  \end{subfigure}
  \quad
    \begin{subfigure}[t]{0.48\textwidth}
    \centering
    \includegraphics[width=\textwidth]{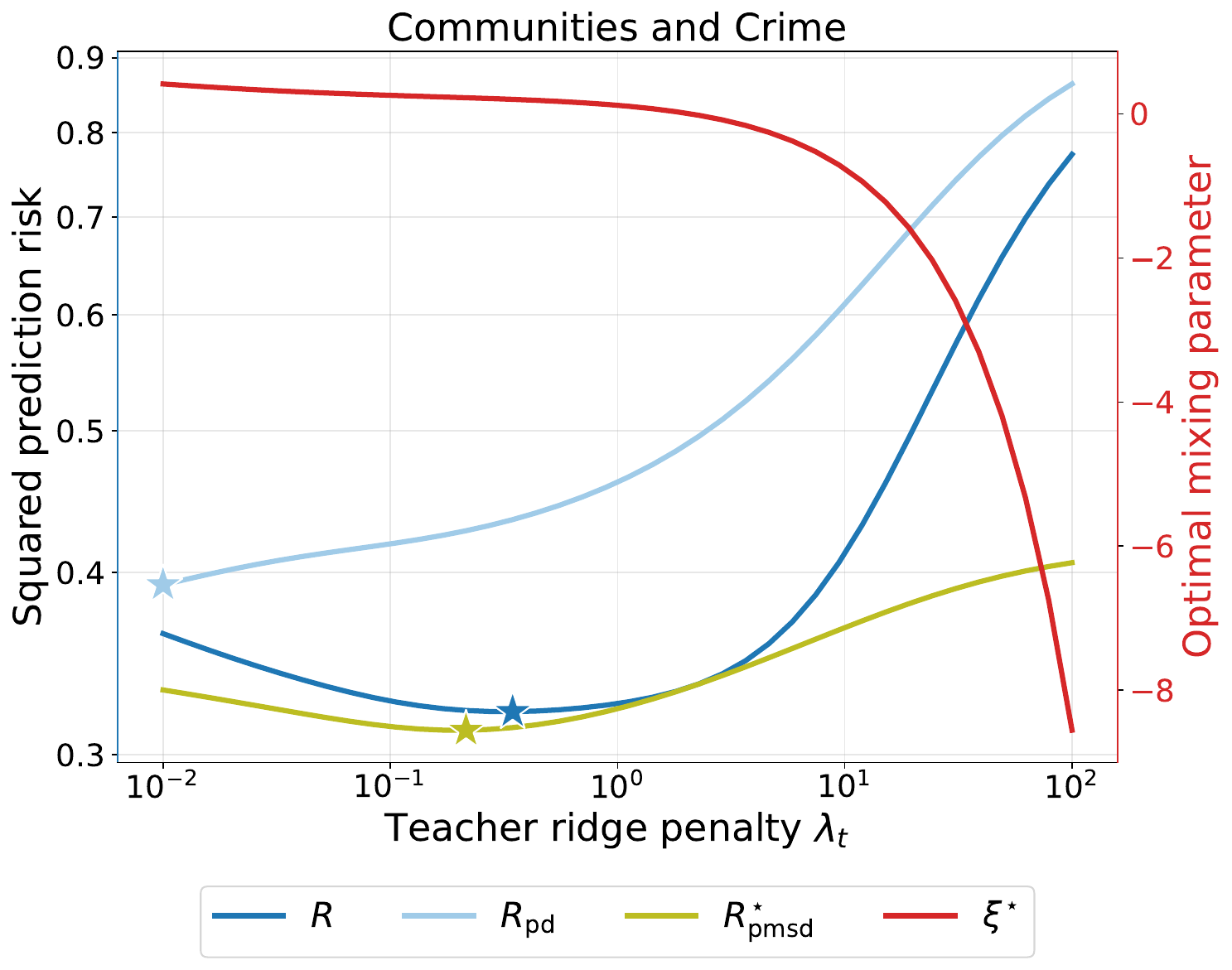}
    \caption{Fresh unlabeled covariates are drawn i.i.d.\ from $\mathcal{N}(0, I_p)$.}
  \end{subfigure}
    \caption{Communities and Crime with $n_t = 400$, $n_s = 800$, and $p = 99$. The student regularization is fixed at $\lambda_s = 1$ for all values of $\lambda_t$.}
\end{figure*}

\begin{figure}[!ht]
   \centering
    \begin{subfigure}[t]{0.48\textwidth}
    \centering
    \includegraphics[width=\textwidth]{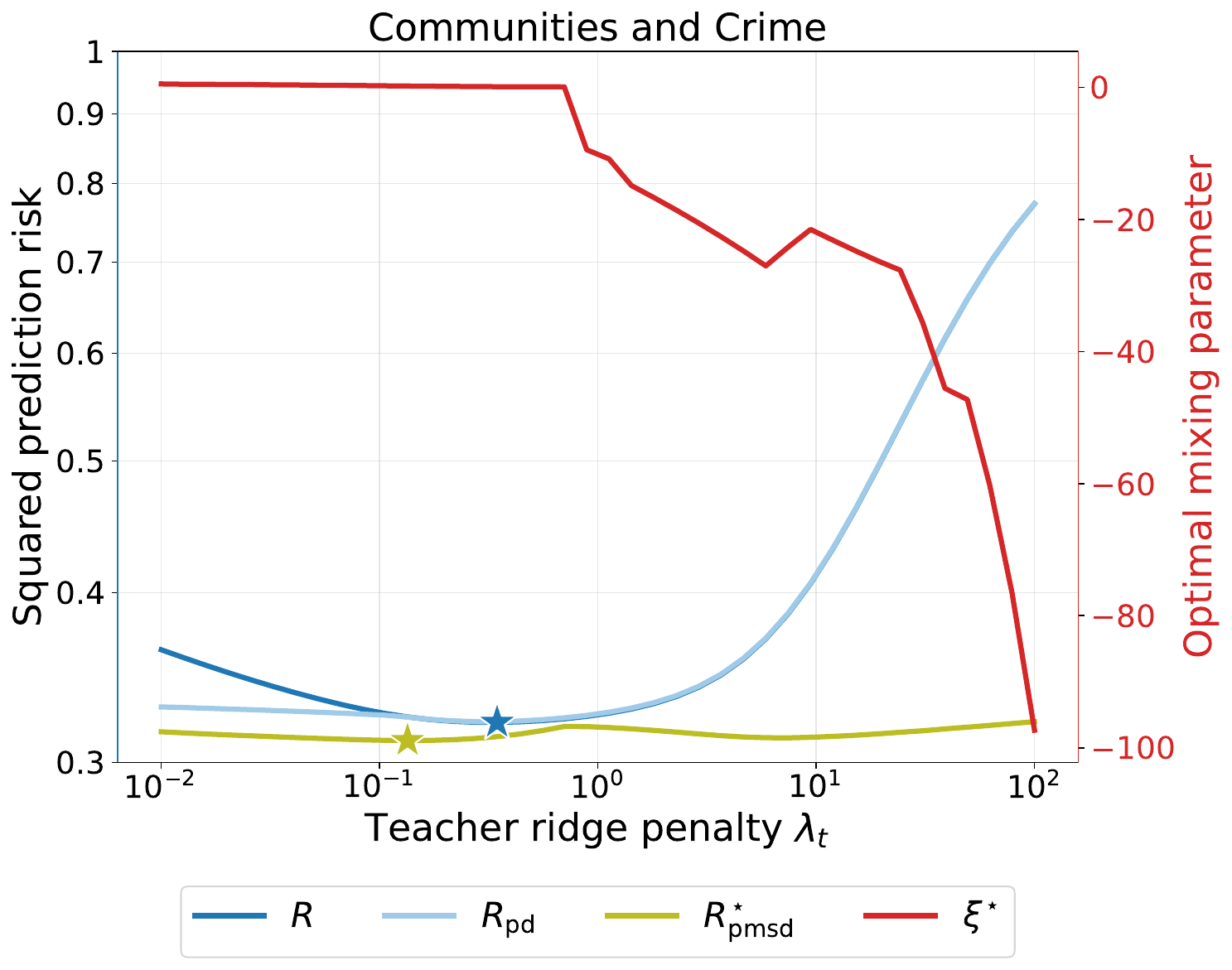}
    \caption{Fresh unlabeled covariates are drawn from the same distribution as the teacher's training covariates.}
  \end{subfigure}
  \quad
    \begin{subfigure}[t]{0.48\textwidth}
    \centering
    \includegraphics[width=\textwidth]{figures/arxiv_new/CC_0.2_0.4_isotropic_fresh_tuned_positive.pdf}
    \caption{Fresh unlabeled covariates are drawn i.i.d.\ from $\mathcal{N}(0, I_p)$.}
  \end{subfigure}
    \caption{Communities and Crime with $n_t = 400$, $n_s = 800$, and $p = 99$. At each value of $\lambda_t$, the student regularization $\lambda_s$ is tuned separately over a grid for the PD and PMSD students.}
\end{figure}

\vspace{4em}

\subsubsection{Blog Feedback}

\begin{figure*}[!ht]
   \centering
    \begin{subfigure}[t]{0.48\textwidth}
    \centering
    \includegraphics[width=\textwidth]{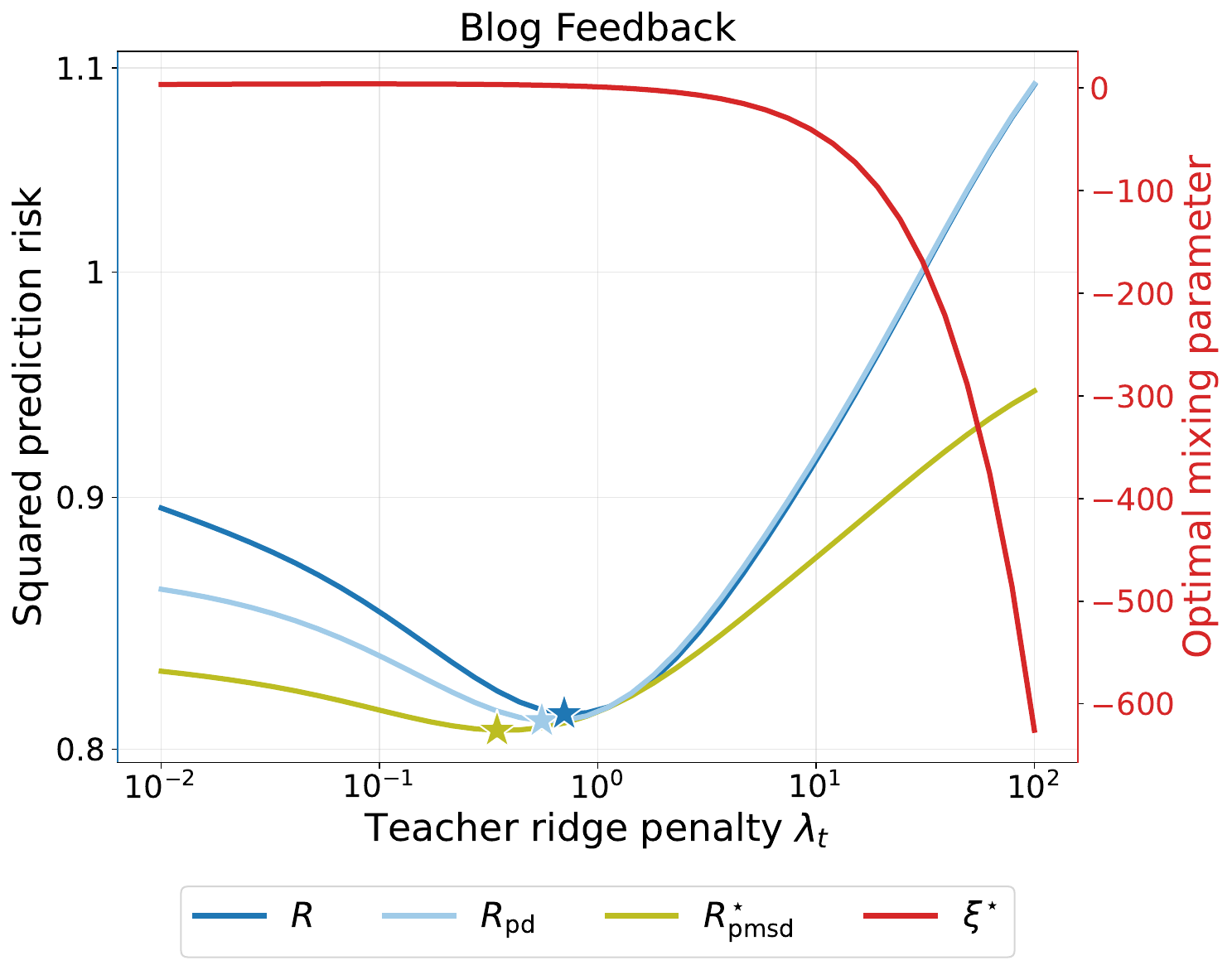}
    \caption{Fresh unlabeled covariates are drawn from the same distribution as the teacher's training covariates.}
  \end{subfigure}
  \quad
    \begin{subfigure}[t]{0.48\textwidth}
    \centering
    \includegraphics[width=\textwidth]{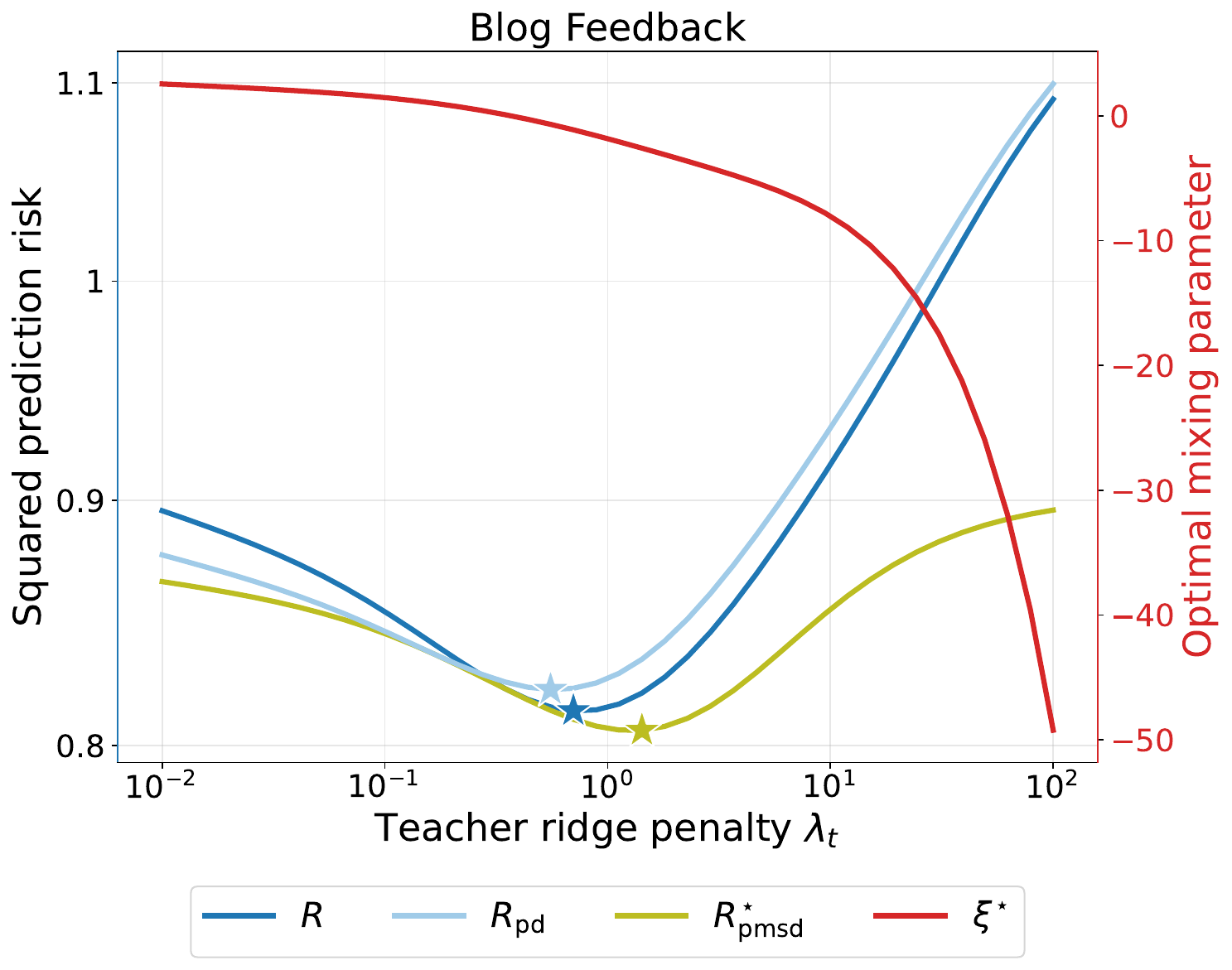}
    \caption{Fresh unlabeled covariates are drawn i.i.d.\ from $\mathcal{N}(0, I_p)$.}
  \end{subfigure}
    \caption{Blog Feedback with $n_t = 2619$, $n_s = 5240$, and $p = 280$. The student regularization is fixed at $\lambda_s = 0.1$ for all values of $\lambda_t$.}
\end{figure*}

\begin{figure*}[!ht]
   \centering
    \begin{subfigure}[t]{0.48\textwidth}
    \centering
    \includegraphics[width=\textwidth]{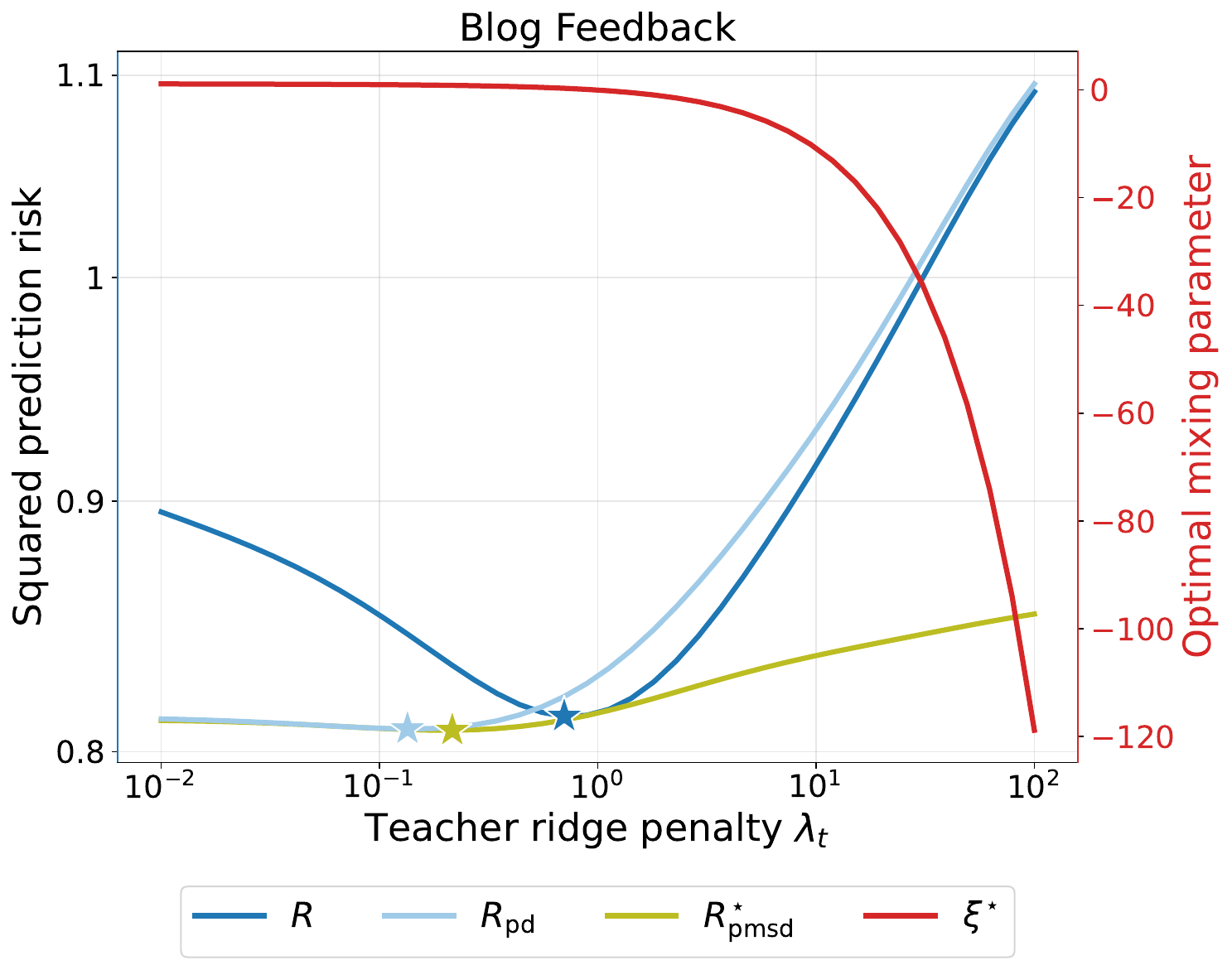}
    \caption{Fresh unlabeled covariates are drawn from the same distribution as the teacher's training covariates.}
  \end{subfigure}
  \quad
    \begin{subfigure}[t]{0.48\textwidth}
    \centering
    \includegraphics[width=\textwidth]{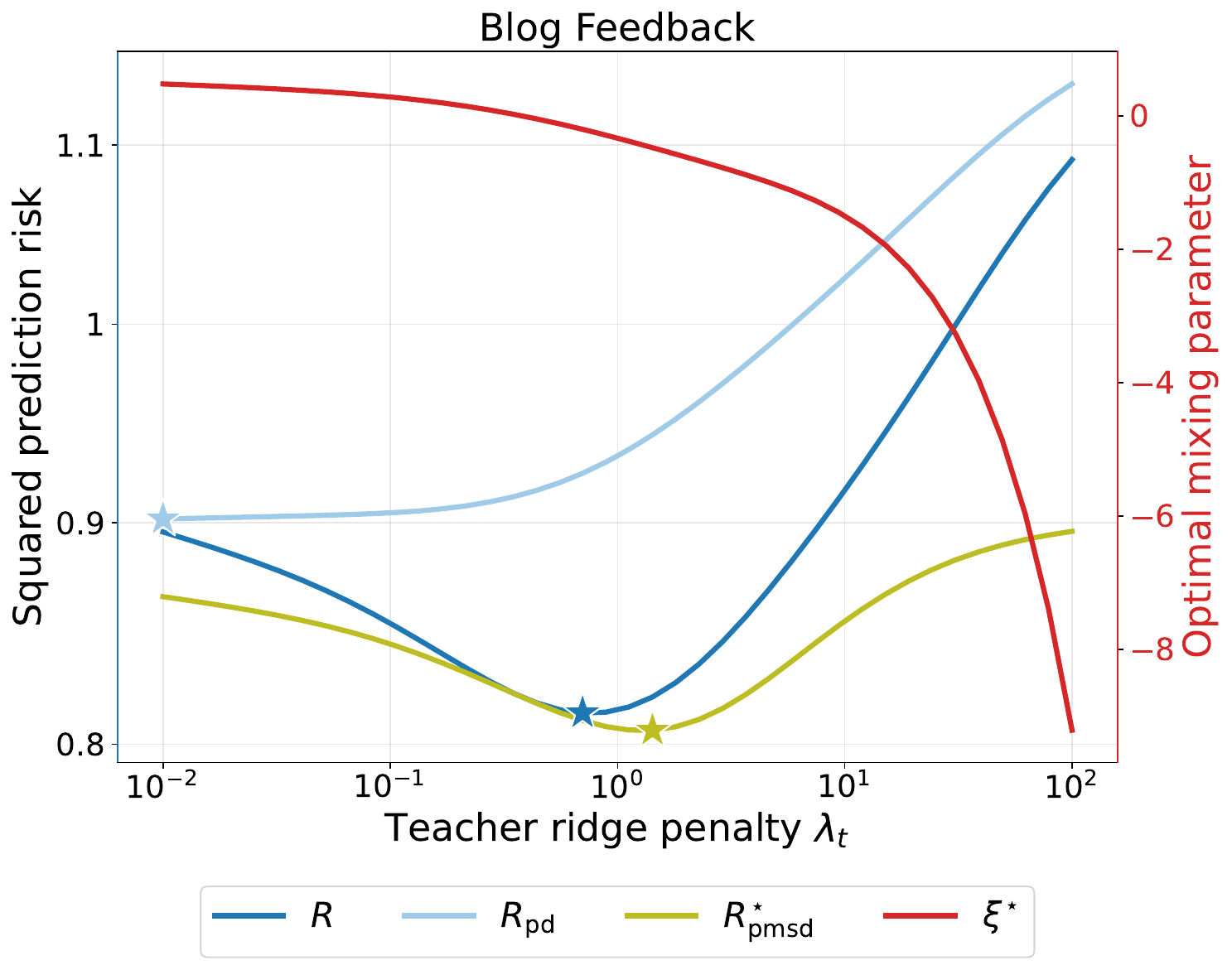}
    \caption{Fresh unlabeled covariates are drawn i.i.d.\ from $\mathcal{N}(0, I_p)$.}
  \end{subfigure}
    \caption{Blog Feedback with $n_t = 2619$, $n_s = 5240$, and $p = 280$. The student regularization is fixed at $\lambda_s = 1$ for all values of $\lambda_t$.}
\end{figure*}

\begin{figure*}[!ht]
   \centering
    \begin{subfigure}[t]{0.48\textwidth}
    \centering
    \includegraphics[width=\textwidth]{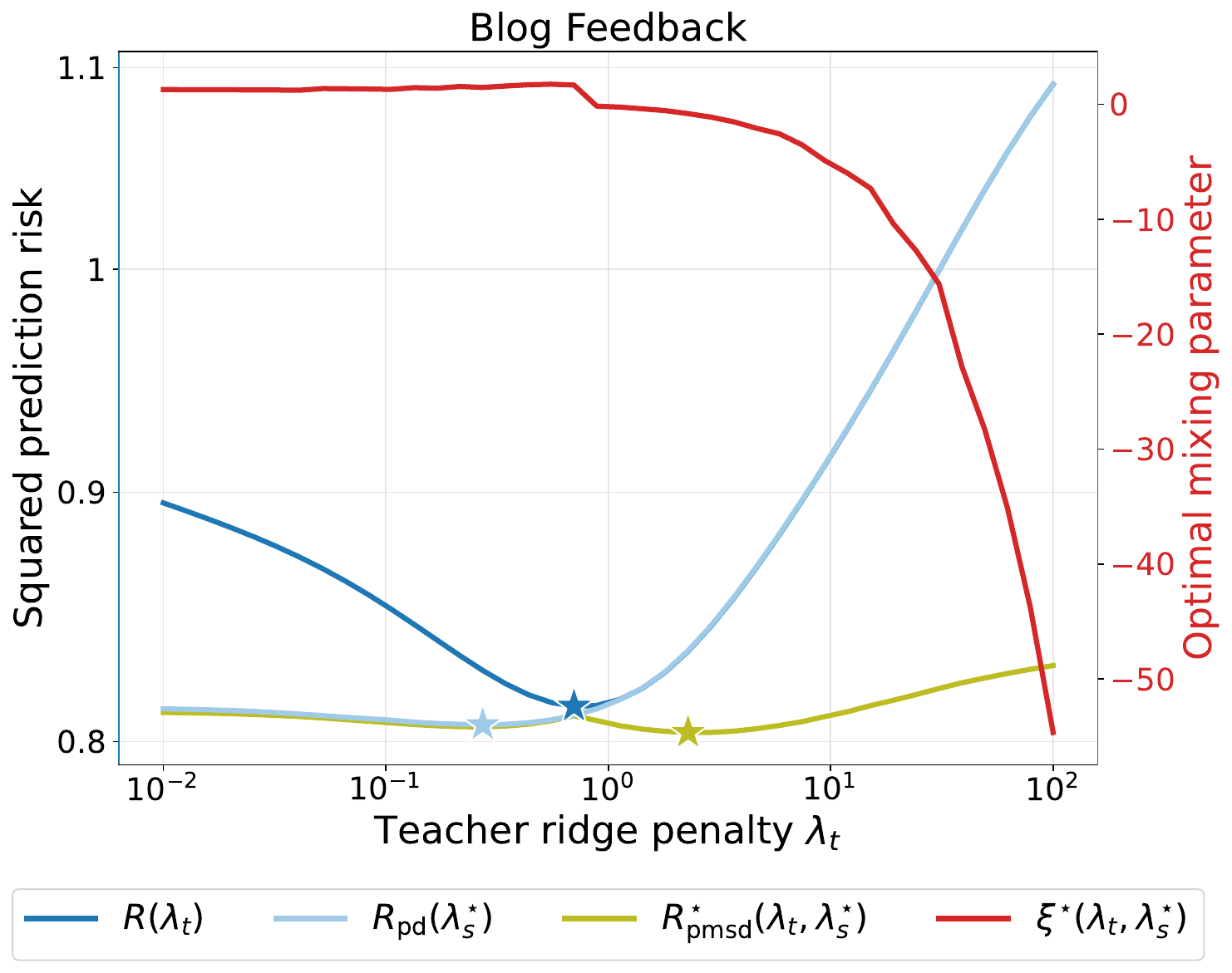}
    \caption{Fresh unlabeled covariates are drawn from the same distribution as the teacher's training covariates.}
  \end{subfigure}
  \quad
    \begin{subfigure}[t]{0.48\textwidth}
    \centering
    \includegraphics[width=\textwidth]{figures/arxiv_new/BF_0.05_0.1_isotropic_fresh_tuned_positive.pdf}
    \caption{Fresh unlabeled covariates are drawn i.i.d.\ from $\mathcal{N}(0, I_p)$.}
  \end{subfigure}
    \caption{Blog Feedback with $n_t = 2619$, $n_s = 5240$, and $p = 280$. At each value of $\lambda_t$, the student regularization $\lambda_s$ is tuned separately over a grid for the PD and PMSD students.}
\end{figure*}

\clearpage

\subsection{Linear probing experiments on real datasets}
\label{sec:linear_probing_tables}

The following tables report test classification accuracy and the estimated optimal mixing weight across regularization levels. In every experiment, the mixing weight is selected using a calibration set over the grid $[-20, 20]$; see \Cref{sec:experiment_details} for the dataset and training details.

\subsubsection{Caltech-101}

\begin{table}[!ht]
\caption{Test classification accuracy on Caltech-101 with random corruption rate $\rho = 0.2$.}
\label{tab:caltech101_0.2}
\centering
\resizebox{\textwidth}{!}{%
\begin{tabular}{lccccccccc}
\toprule
Regularization level $\lambda$ & $10^{-4}$ & $10^{-3.5}$ & $10^{-3}$ & $10^{-2.5}$ & $10^{-2}$ & $10^{-1.5}$ & $10^{-1}$ & $10^{-0.5}$ & $10^{0}$\\
\midrule
Teacher (\%) & $68.7$ & $68.7$ & $68.8$ & $68.9$ & $69.5$ & $70.9$ & $72.8$ & $68.4$ & $46.0$ \\
PD student (\%) & $68.9$ & $69.7$ & $71.4$ & $72.4$ & $73.6$ & $71.9$ & $63.5$ & $42.7$ & $27.0$ \\
\midrule
Optimal PMSD student (\%) & $68.9$ & $70.8$ & $71.5$ & $72.4$ & $73.4$ & $72.5$ & $72.8$ & $76.1$ & $63.0$  \\
Estimated optimal mixing weight $\hat\xi^{\star}$ & $5.5$ & $2.5$ & $1.1$ & $0.9$ & $0.9$ & $0.7$ & $-0.1$ & $-20.0$ & $-17.8$
\\
\bottomrule
\end{tabular}%
}
\end{table}

\begin{table}[ht]
\caption{Test classification accuracy on Caltech-101 with random corruption rate $\rho = 0.4$.}
\label{tab:caltech101_0.4}
\centering
\resizebox{\textwidth}{!}{%
\begin{tabular}{lccccccccc}
\toprule
Regularization level $\lambda$ & $10^{-4}$ & $10^{-3.5}$ & $10^{-3}$ & $10^{-2.5}$ & $10^{-2}$ & $10^{-1.5}$ & $10^{-1}$ & $10^{-0.5}$ & $10^{0}$\\
\midrule
Teacher (\%) & $62.5$ & $62.6$ & $62.6$ & $62.9$ & $63.5$ & $65.3$ & $69.3$ & $67.4$ & $45.6$ \\
PD student (\%) & $62.7$ & $63.3$ & $65.2$ & $67.7$ & $69.4$ & $68.4$ & $61.6$ & $41.3$ & $26.2$ \\
\midrule
Optimal PMSD student (\%) & $63.1$ & $65.6$ & $66.8$ & $68.3$ & $69.6$ & $68.9$ & $69.2$ & $72.7$ & $63.0$ \\
Estimated optimal mixing weight $\hat\xi^{\star}$ & $4.9$ & $5.5$ & $2.2$ & $1.4$ & $1.1$ & $0.9$ & $-0.1$ & $-4.2$ & $-18.6$
\\
\bottomrule
\end{tabular}%
}
\end{table}

\begin{table}[ht]
\caption{Test classification accuracy on Caltech-101 with random corruption rate $\rho = 0.6$.}
\label{tab:caltech101_0.6}
\centering
\resizebox{\textwidth}{!}{%
\begin{tabular}{lccccccccc}
\toprule
Regularization level $\lambda$ & $10^{-4}$ & $10^{-3.5}$ & $10^{-3}$ & $10^{-2.5}$ & $10^{-2}$ & $10^{-1.5}$ & $10^{-1}$ & $10^{-0.5}$ & $10^{0}$\\
\midrule
Teacher (\%) & $49.3$ & $49.3$ & $49.3$ & $49.6$ & $50.1$ & $52.2$ & $57.1$ & $59.8$ & $36.2$ \\
PD student (\%) & $49.6$ & $50.0$ & $51.6$ & $54.1$ & $57.4$ & $58.3$ & $52.5$ & $33.4$ & $23.5$ \\
\midrule
Optimal PMSD student (\%) & $49.5$ & $52.4$ & $54.3$ & $55.9$ & $57.0$ & $58.2$ & $57.2$ & $65.5$ & $60.4$ \\
Estimated optimal mixing weight $\hat\xi^{\star}$ & $0.8$ & $8.4$ & $3.2$ & $1.6$ & $1.7$ & $1.0$ & $0.2$ & $-10.6$ & $-20.0$
\\
\bottomrule
\end{tabular}%
}

\end{table}

\clearpage
\subsubsection{Caltech-256}

\begin{table}[!ht]
\caption{Test classification accuracy on Caltech-256 with random corruption rate $\rho = 0.2$.}
\label{tab:caltech256_0.2}
\centering
\resizebox{\textwidth}{!}{%
\begin{tabular}{lccccccccc}
\toprule
Regularization level $\lambda$ & $10^{-4}$ & $10^{-3.5}$ & $10^{-3}$ & $10^{-2.5}$ & $10^{-2}$ & $10^{-1.5}$ & $10^{-1}$ & $10^{-0.5}$ & $10^{0}$\\
\midrule
Teacher (\%) & $65.0$ & $65.1$ & $65.2$ & $65.5$ & $66.6$ & $67.5$ & $61.5$ & $35.8$ & $14.1$ \\
PD student (\%) & $65.7$ & $66.5$ & $68.0$ & $68.6$ & $67.8$ & $58.4$ & $28.4$ & $11.4$ & $6.0$ \\
\midrule
Optimal PMSD student (\%) & $66.9$ & $67.7$ & $68.2$ & $68.5$ & $68.1$ & $68.0$ & $67.6$ & $50.8$ & $25.6$ \\
Estimated optimal mixing weight $\hat\xi^{\star}$ & $5.6$ & $2.7$ & $1.2$ & $0.9$ & $0.5$ & $-1.9$ & $-20.0$ & $-20.0$ & $-20.0$
\\
\bottomrule
\end{tabular}%
}
\end{table}

\begin{table}[!ht]
\caption{Test classification accuracy on Caltech-256 with random corruption rate $\rho = 0.4$.}
\label{tab:caltech256_0.4}
\centering
\resizebox{\textwidth}{!}{%
\begin{tabular}{lccccccccc}
\toprule
Regularization level $\lambda$ & $10^{-4}$ & $10^{-3.5}$ & $10^{-3}$ & $10^{-2.5}$ & $10^{-2}$ & $10^{-1.5}$ & $10^{-1}$ & $10^{-0.5}$ & $10^{0}$\\
\midrule
Teacher (\%) & $56.8$ & $56.8$ & $57.0$ & $57.5$ & $58.4$ & $60.3$ & $55.0$ & $33.4$ & $13.0$ \\
PD student (\%) & $57.4$ & $58.2$ & $59.7$ & $60.9$ & $60.3$ & $49.3$ & $25.5$ & $10.4$ & $5.4$ \\
\midrule
Optimal PMSD student (\%) & $59.0$ & $59.5$ & $60.3$ & $61.2$ & $60.4$ & $60.6$ & $61.5$ & $45.7$ & $23.5$ \\
Estimated optimal mixing weight $\hat\xi^{\star}$ & $7.4$ & $2.8$ & $1.6$ & $1.2$ & $0.6$ & $-0.8$ & $-20.0$ & $-20.0$ & $-20.0$
\\
\bottomrule
\end{tabular}%
}
\end{table}

\begin{table}[!ht]
\caption{Test classification accuracy on Caltech-256 with random corruption rate $\rho = 0.6$.}
\label{tab:caltech256_0.6}
\centering
\resizebox{\textwidth}{!}{%
\begin{tabular}{lccccccccc}
\toprule
Regularization level $\lambda$ & $10^{-4}$ & $10^{-3.5}$ & $10^{-3}$ & $10^{-2.5}$ & $10^{-2}$ & $10^{-1.5}$ & $10^{-1}$ & $10^{-0.5}$ & $10^{0}$\\
\midrule
Teacher (\%) & $43.5$ & $43.6$ & $43.9$ & $44.4$ & $45.6$ & $48.0$ & $45.1$ & $27.5$ & $11.6$ \\
PD student (\%) & $44.1$ & $44.9$ & $46.7$ & $48.3$ & $48.0$ & $39.4$ & $19.6$ & $9.5$ & $3.8$ \\
\midrule
Optimal PMSD student (\%) & $45.8$ & $46.5$ & $47.5$ & $48.6$ & $47.6$ & $48.3$ & $51.5$ & $39.6$ & $21.5$ \\
Estimated optimal mixing weight $\hat\xi^{\star}$ & $6.0$ & $2.5$ & $1.9$ & $1.1$ & $0.6$ & $-2.1$ & $-20.0$ & $-20.0$ & $-19.4$
\\
\bottomrule
\end{tabular}%
}
\end{table}

\clearpage

\subsubsection{CIFAR-100 (random corruption)}

\begin{table}[ht]
\caption{Test classification accuracy on CIFAR-100 with random corruption rate $\rho = 0.2$.}
\label{tab:cifar100_0.2}
\centering
\resizebox{\textwidth}{!}{%
\begin{tabular}{lccccccccc}
\toprule
Regularization level $\lambda$ & $10^{-4}$ & $10^{-3.5}$ & $10^{-3}$ & $10^{-2.5}$ & $10^{-2}$ & $10^{-1.5}$ & $10^{-1}$ & $10^{-0.5}$ & $10^{0}$\\
\midrule
Teacher (\%) & $53.8$ & $53.9$ & $54.0$ & $54.4$ & $55.2$ & $57.1$ & $58.2$ & $50.6$ & $35.9$ \\
PD student (\%) & $54.0$ & $54.6$ & $55.4$ & $56.8$ & $57.9$ & $56.9$ & $46.4$ & $18.4$ & $2.4$ \\
\midrule
Optimal PMSD student (\%) & $55.0$ & $55.3$ & $55.9$ & $56.9$ & $57.9$ & $57.6$ & $59.7$ & $54.8$ & $44.2$ \\
Estimated optimal mixing weight $\hat\xi^{\star}$ & $6.6$ & $2.4$ & $1.5$ & $1.2$ & $0.9$ & $0.7$ & $-20.0$ & $-20.0$ & $-20.0$
\\
\bottomrule
\end{tabular}%
}
\end{table}

\begin{table}[ht]
\caption{Test classification accuracy on CIFAR-100 with random corruption rate $\rho = 0.4$.}
\label{tab:cifar100_0.4}
\centering
\resizebox{\textwidth}{!}{%
\begin{tabular}{lccccccccc}
\toprule
Regularization level $\lambda$ & $10^{-4}$ & $10^{-3.5}$ & $10^{-3}$ & $10^{-2.5}$ & $10^{-2}$ & $10^{-1.5}$ & $10^{-1}$ & $10^{-0.5}$ & $10^{0}$\\
\midrule
Teacher (\%) & $47.8$ & $47.9$ & $48.0$ & $48.5$ & $49.8$ & $52.7$ & $54.2$ & $45.4$ & $31.7$ \\
PD student (\%) & $48.3$ & $49.3$ & $50.6$ & $52.3$ & $54.0$ & $52.6$ & $39.9$ & $16.0$ & $2.3$ \\
\midrule
Optimal PMSD student (\%) & $50.5$ & $50.6$ & $51.5$ & $53.3$ & $54.0$ & $53.5$ & $56.1$ & $49.5$ & $40.3$ \\
Estimated optimal mixing weight $\hat\xi^{\star}$ & $6.9$ & $4.4$ & $2.1$ & $1.5$ & $1.0$ & $0.8$ & $-20.0$ & $-14.4$ & $-20.0$
\\
\bottomrule
\end{tabular}%
}
\end{table}

\begin{table}[!ht]
\caption{Test classification accuracy on CIFAR-100 with random corruption rate $\rho = 0.6$.}
\label{tab:cifar100_0.6}
\centering
\resizebox{\textwidth}{!}{%
\begin{tabular}{lccccccccc}
\toprule
Regularization level $\lambda$ & $10^{-4}$ & $10^{-3.5}$ & $10^{-3}$ & $10^{-2.5}$ & $10^{-2}$ & $10^{-1.5}$ & $10^{-1}$ & $10^{-0.5}$ & $10^{0}$\\
\midrule
Teacher (\%) & $38.0$ & $38.0$ & $38.1$ & $38.6$ & $40.1$ & $43.9$ & $46.3$ & $37.3$ & $24.3$ \\
PD student (\%) & $38.3$ & $39.1$ & $41.0$ & $43.6$ & $46.4$ & $44.7$ & $30.5$ & $13.6$ & $3.3$ \\
\midrule
Optimal PMSD student (\%) & $40.9$ & $41.6$ & $43.0$ & $45.4$ & $46.9$ & $44.9$ & $48.2$ & $43.1$ & $34.0$ \\
Estimated optimal mixing weight $\hat\xi^{\star}$ & $10.2$ & $4.8$ & $2.5$ & $1.7$ & $1.3$ & $0.5$ & $-1.7$ & $-20.0$ & $-20.0$
\\
\bottomrule
\end{tabular}%
}
\end{table}

\clearpage
\subsubsection{CIFAR-100 (hierarchical corruption)}

\begin{table}[ht]
\caption{Test classification accuracy on CIFAR-100 with hierarchical corruption rate $\rho = 0.4$.}
\label{tab:cifar100_0.4_hierarchical}
\centering
\resizebox{\textwidth}{!}{%
\begin{tabular}{lccccccccc}
\toprule
Regularization level $\lambda$ & $10^{-4}$ & $10^{-3.5}$ & $10^{-3}$ & $10^{-2.5}$ & $10^{-2}$ & $10^{-1.5}$ & $10^{-1}$ & $10^{-0.5}$ & $10^{0}$\\
\midrule
Teacher (\%) & $44.9$ & $45.0$ & $45.0$ & $45.3$ & $46.3$ & $49.5$ & $51.1$ & $42.7$ & $28.4$ \\
PD student (\%) & $45.2$ & $46.1$ & $47.6$ & $50.0$ & $51.6$ & $50.9$ & $39.4$ & $17.5$ & $4.4$ \\
\midrule
Optimal PMSD student (\%) & $47.0$ & $47.6$ & $48.9$ & $49.9$ & $51.6$ & $50.9$ & $53.5$ & $47.0$ & $35.9$ \\
Estimated optimal mixing weight $\hat\xi^{\star}$ & $13.0$ & $5.7$ & $2.5$ & $2.1$ & $1.3$ & $0.9$ & $-15.6$ & $-20.0$ & $-20.0$
\\
\bottomrule
\end{tabular}%
}
\end{table}

\begin{table}[ht]
\caption{Test classification accuracy on CIFAR-100 with hierarchical corruption rate $\rho = 0.6$.}
\label{tab:cifar100_0.6_hierarchical}
\centering
\resizebox{\textwidth}{!}{%
\begin{tabular}{lccccccccc}
\toprule
Regularization level $\lambda$ & $10^{-4}$ & $10^{-3.5}$ & $10^{-3}$ & $10^{-2.5}$ & $10^{-2}$ & $10^{-1.5}$ & $10^{-1}$ & $10^{-0.5}$ & $10^{0}$\\
\midrule
Teacher (\%) & $31.3$ & $31.3$ & $31.5$ & $31.8$ & $32.5$ & $35.8$ & $39.0$ & $32.5$ & $22.4$ \\
PD student (\%) & $31.4$ & $32.2$ & $33.4$ & $35.4$ & $38.4$ & $39.1$ & $29.5$ & $15.3$ & $3.8$ \\
\midrule
Optimal PMSD student (\%) & $33.6$ & $34.0$ & $35.3$ & $37.0$ & $38.7$ & $39.2$ & $40.3$ & $36.9$ & $27.8$ \\
Estimated optimal mixing weight $\hat\xi^{\star}$ & $13.0$ & $4.8$ & $2.4$ & $1.8$ & $1.5$ & $1.0$ & $-4.3$ & $-15.6$ & $-20.0$
\\
\bottomrule
\end{tabular}%
}
\end{table}

\clearpage
\subsection{Synthetic experiments for ridge regression}
\label{sec:additional_exps_synthetic_ridge}

We next compare the empirical performance of PMSD with its theoretical predictions under several covariance and regularization settings.

\subsubsection{Empirical versus theoretical predictions}

\begin{figure*}[!ht]
  \centering
    \begin{subfigure}[t]{0.55\textwidth}
    \centering
    \includegraphics[width=\textwidth]{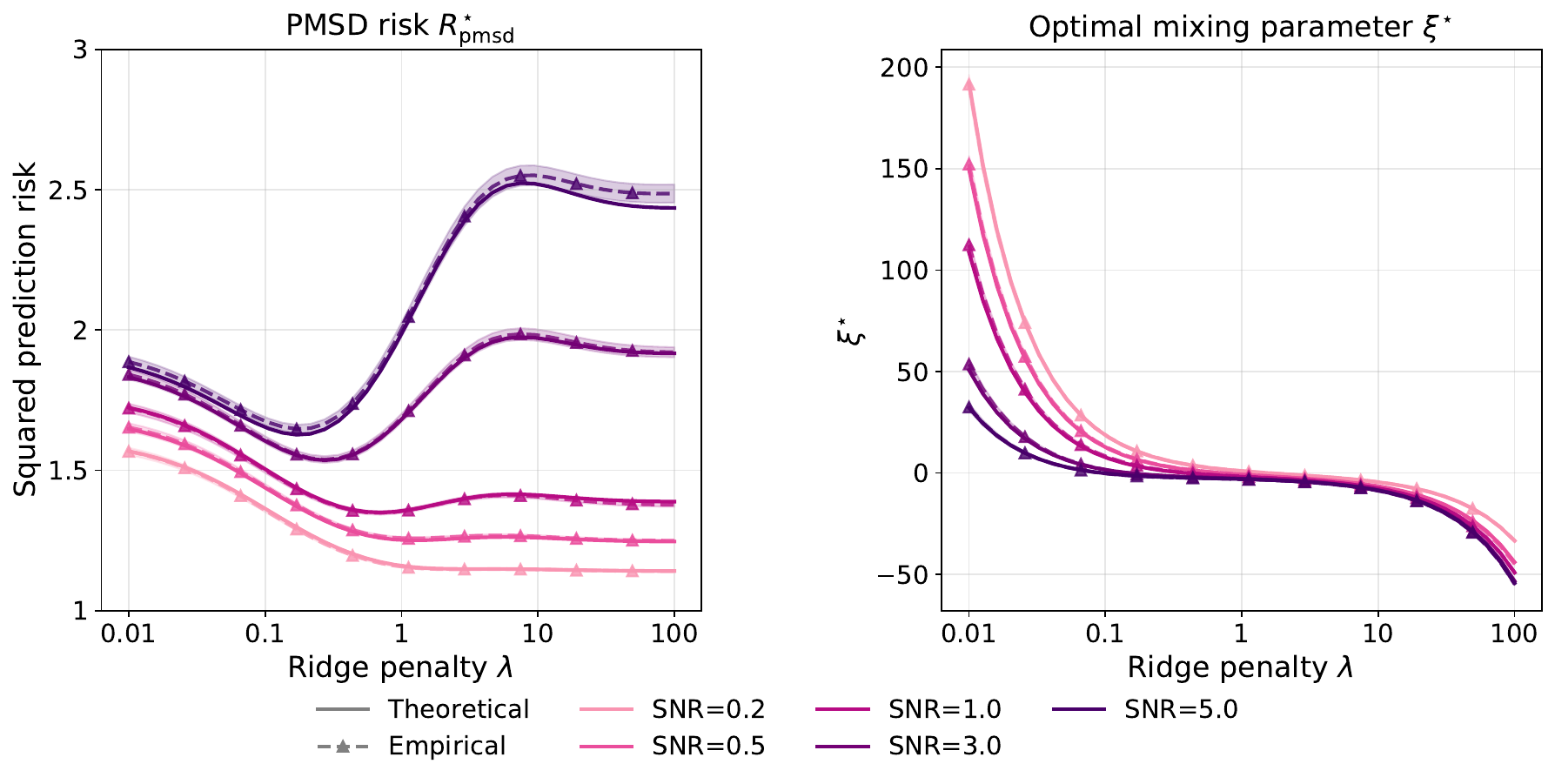}
  \end{subfigure}
  \hfill
    \begin{subfigure}[t]{0.35\textwidth}
    \centering
    \includegraphics[width=\textwidth]{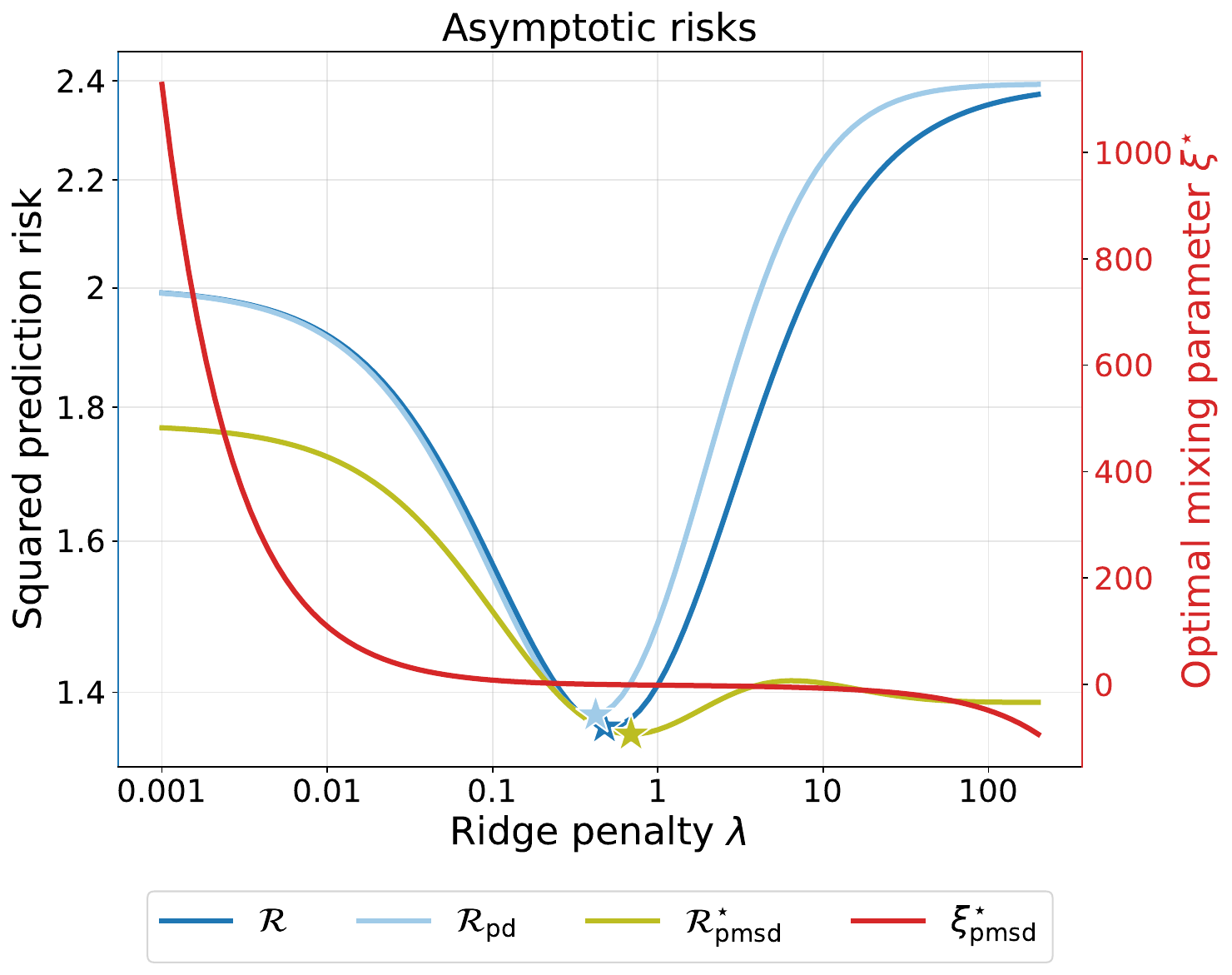}
  \end{subfigure}
  \caption{Empirical and theoretical results when $\Sigma_t$ follows an AR1 covariance model, $\Sigma_s = 10I_p$, and $\lambda_t = \lambda_s$. The signal $\beta$ aligns with the top eigenvectors of $\Sigma_t$. Both panels use $p = 200$, $n_t = n_s = 400$, and $\sigma^2 = 1$; the right panel has $r^2 = 1$.}
\end{figure*}

\begin{figure*}[!ht]
  \centering
    \begin{subfigure}[t]{0.55\textwidth}
    \centering
    \includegraphics[width=\textwidth]{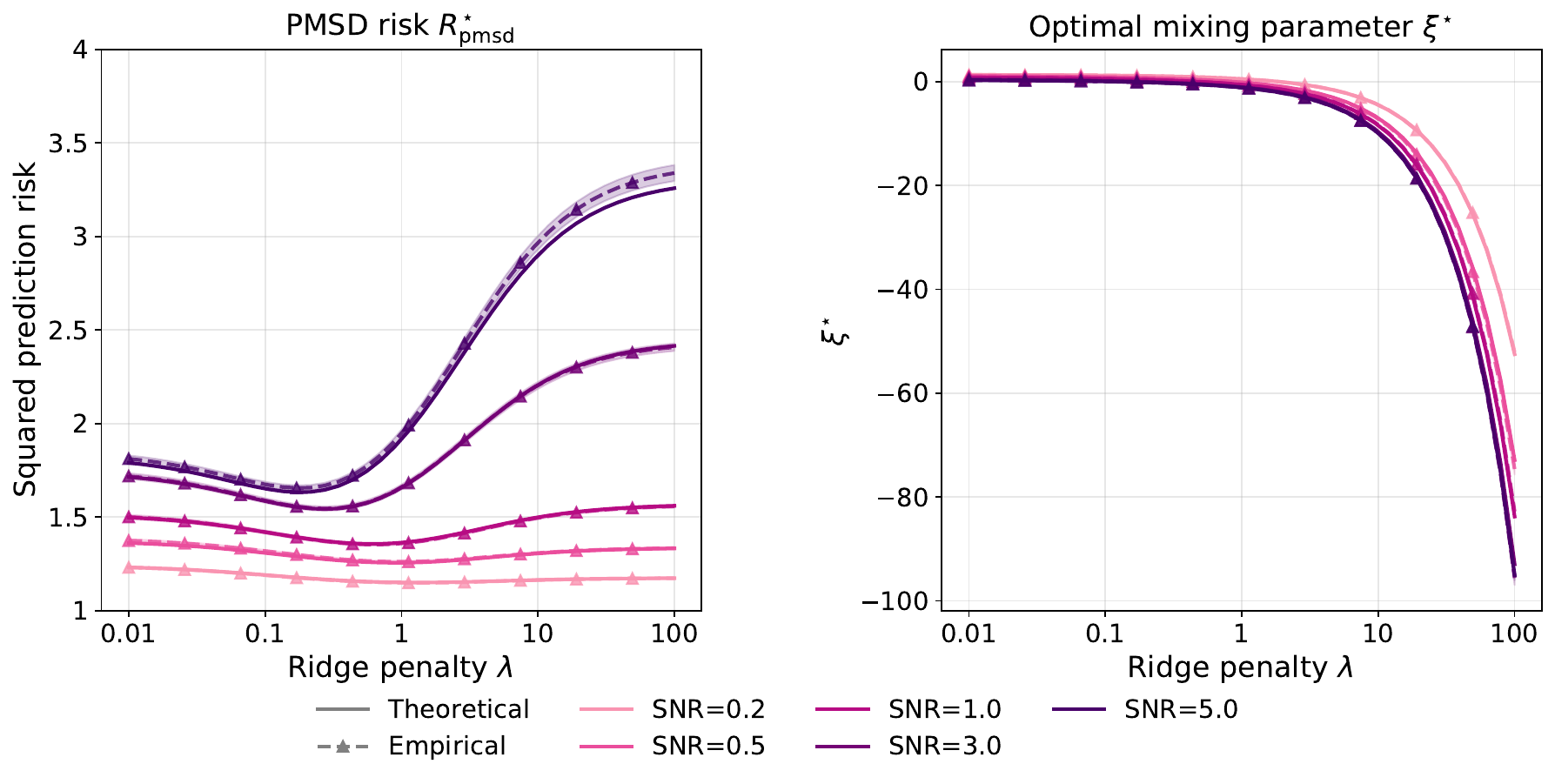}
  \end{subfigure}
  \hfill
    \begin{subfigure}[t]{0.35\textwidth}
    \centering
    \includegraphics[width=\textwidth]{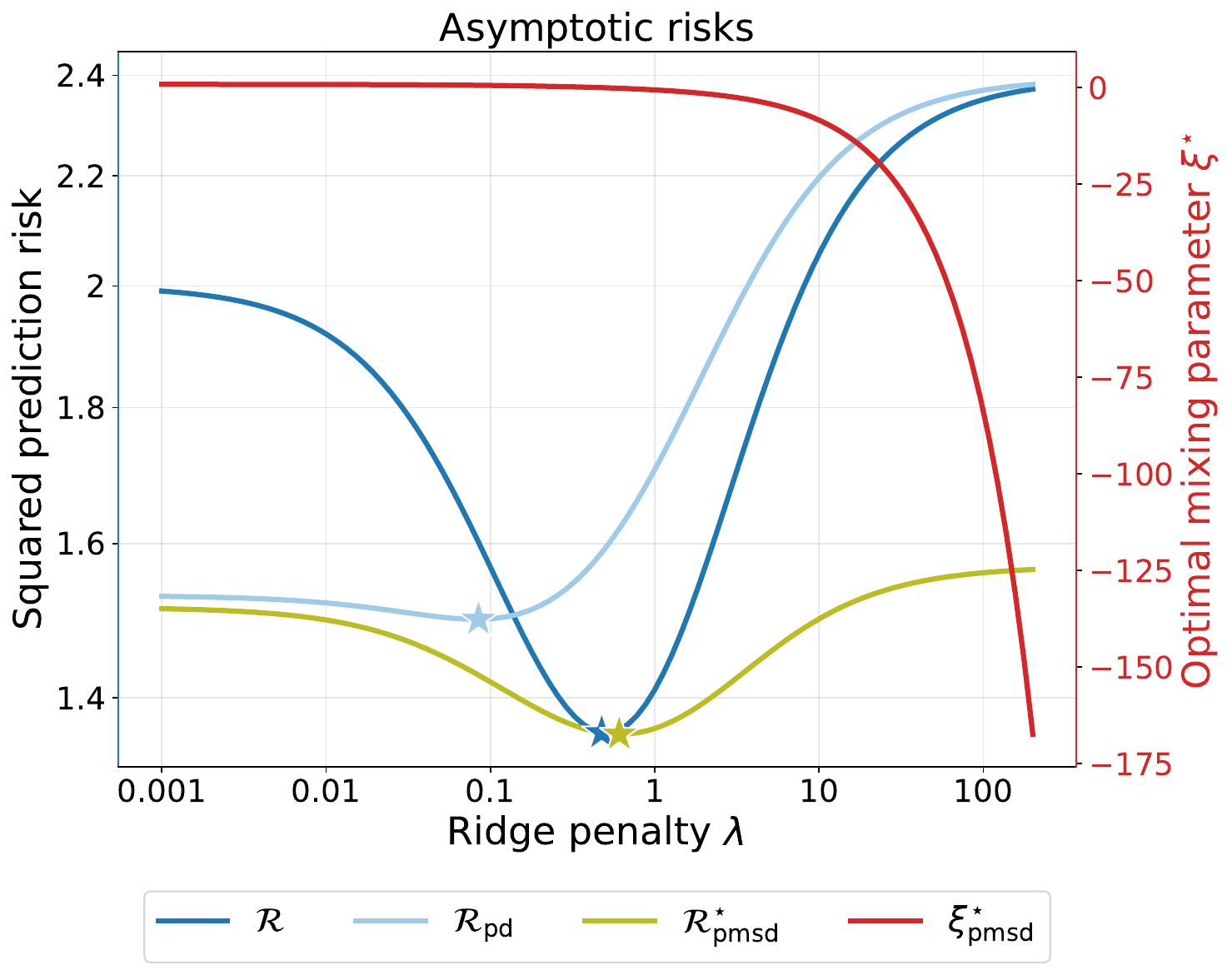}
  \end{subfigure}
  \caption{Empirical and theoretical results when $\Sigma_t = \Sigma_s$ follow the same AR1 covariance model and $\lambda_s = 1$. The signal $\beta$ aligns with the top eigenvectors of $\Sigma_t$. Both panels use $p = 200$, $n_t = n_s = 400$, and $\sigma^2 = 1$; the right panel has $r^2 = 1$.}
\end{figure*}

\begin{figure}[!ht]
  \centering
    \begin{subfigure}[t]{0.55\textwidth}
    \centering
    \includegraphics[width=\textwidth]{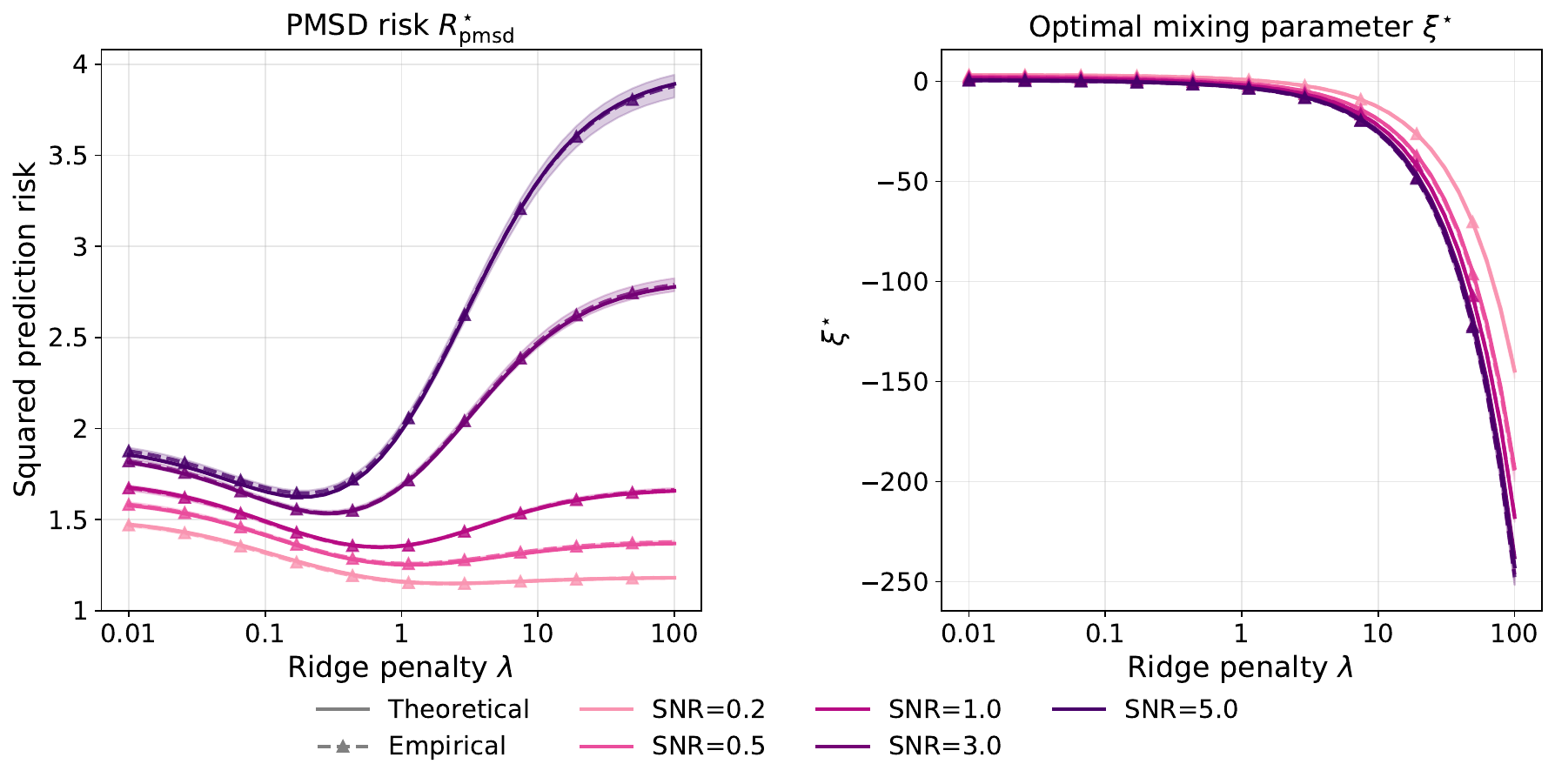}
  \end{subfigure}
  \hfill
    \begin{subfigure}[t]{0.35\textwidth}
    \centering
    \includegraphics[width=\textwidth]{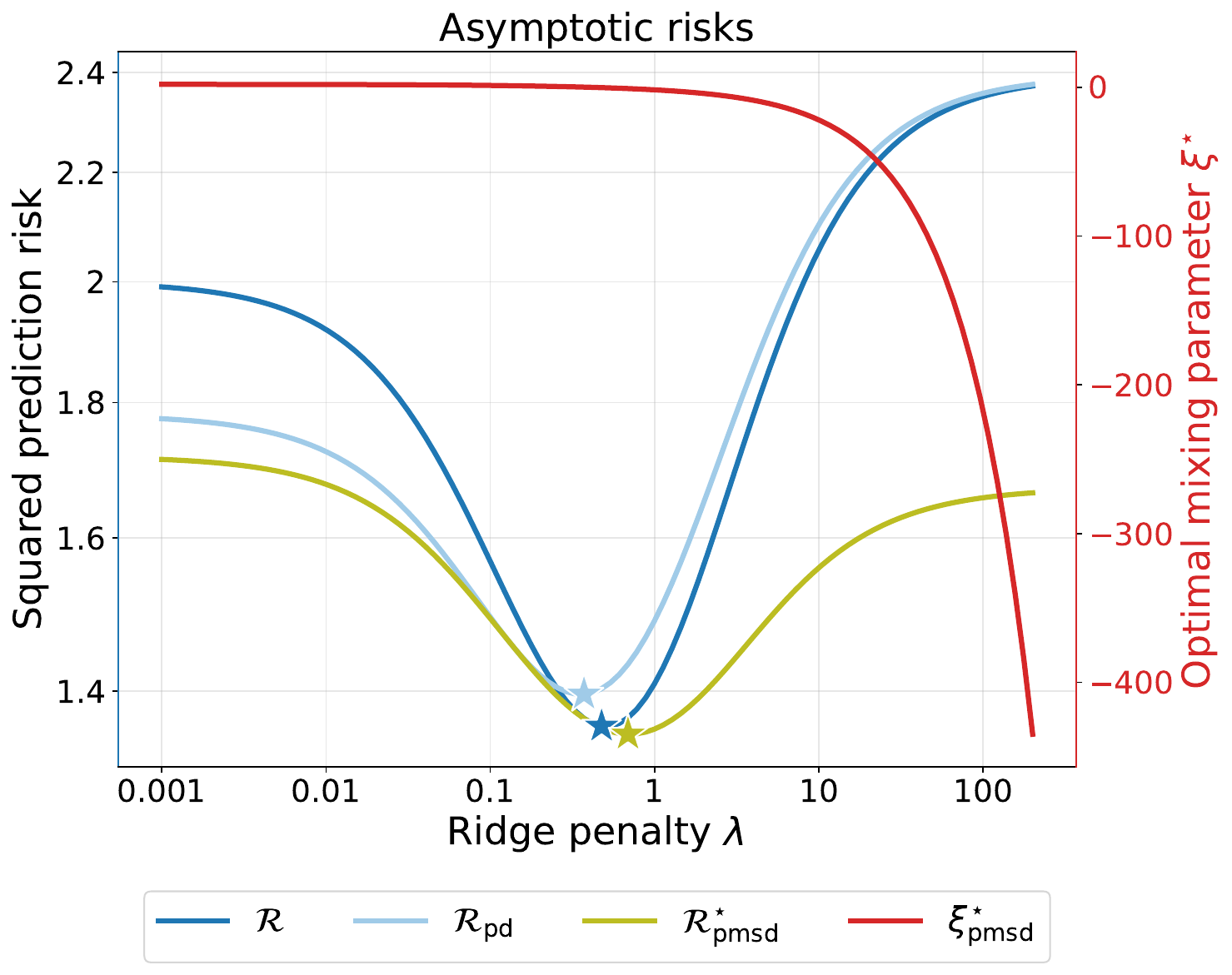}
  \end{subfigure}
  \caption{Empirical and theoretical results when $\Sigma_t$ follows an AR1 covariance model, $\Sigma_s = 10I_p$, and $\lambda_s = 1$. The signal $\beta$ aligns with the top eigenvectors of $\Sigma_t$. Both panels use $p = 200$, $n_t = n_s = 400$, and $\sigma^2 = 1$; the right panel has $r^2 = 1$.}
\end{figure}

\FloatBarrier

\subsubsection{Comparison with the same-$X$ SD student}

\begin{figure}[!ht]
    \centering
    \includegraphics[width=0.65\linewidth]{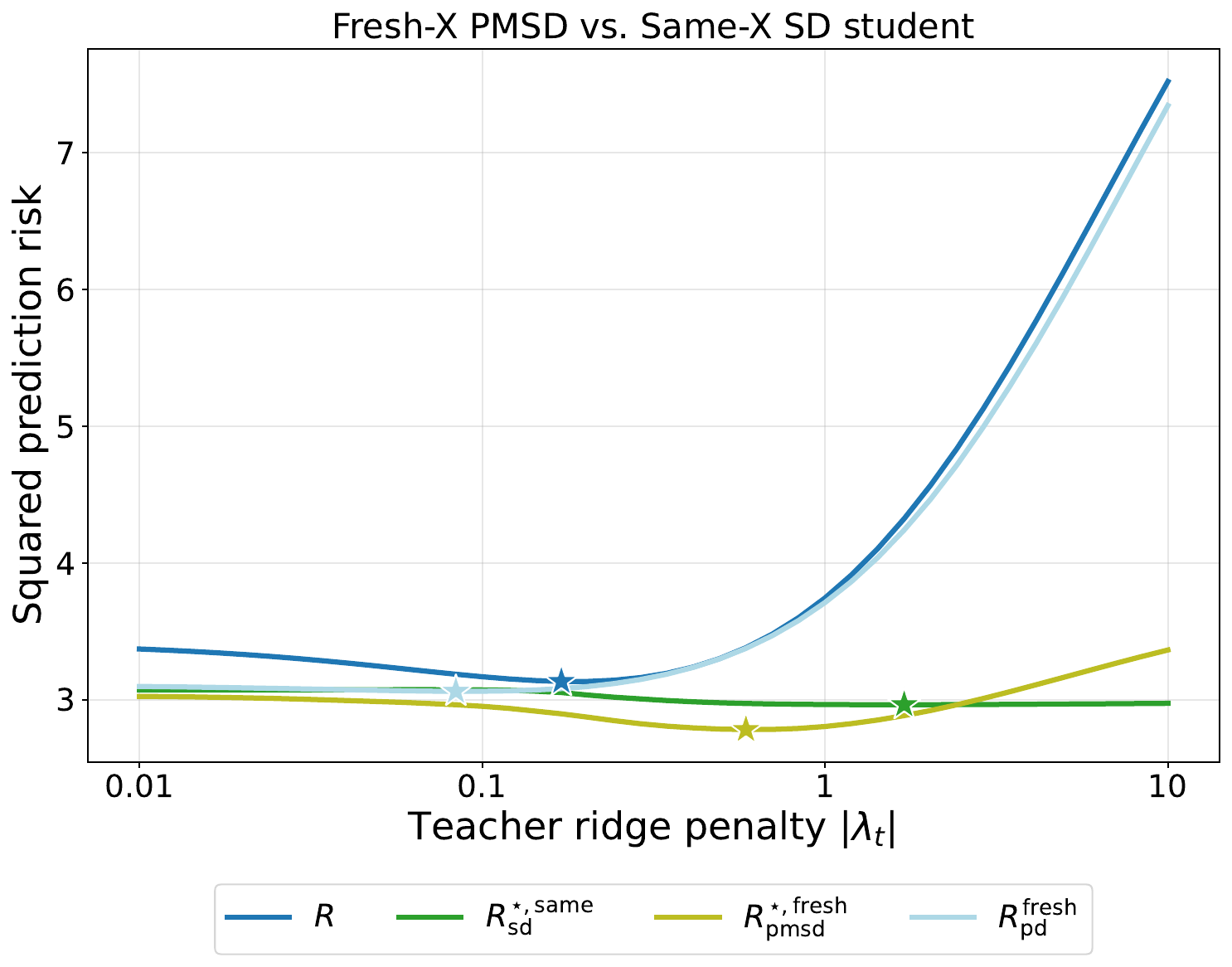}
    \caption{Comparison with the same-$X$ SD student when $p = 400$, $n_t = 200$, $n_s = 4000$, and $\Sigma_t = \Sigma_s$ follow an AR1 covariance model. At each value of $|\lambda_t|$ on the x-axis, we evaluate the teacher regularizations $\lambda_t$ and $-\lambda_t$, tune $\lambda_s$ over a grid on $[-100, 100]$ in each case, and report the lower risks. The signal is $\beta = a u_1 + \sqrt{r^2-a^2}v$, where $a = 1.7$, $r = 2$, $u_1$ is the leading eigenvector of $\Sigma_t$, and $v$ is orthogonal to $u_1$.}
    \label{fig:vs_same_x_ar1}
\end{figure}

\clearpage

\subsection{Synthetic experiments for logistic regression}

\subsubsection{Constant feature correlation}

\bigskip

\begin{figure*}[!ht]
  \centering
    \begin{subfigure}[t]{0.45\textwidth}
    \centering
    \includegraphics[width=\textwidth]{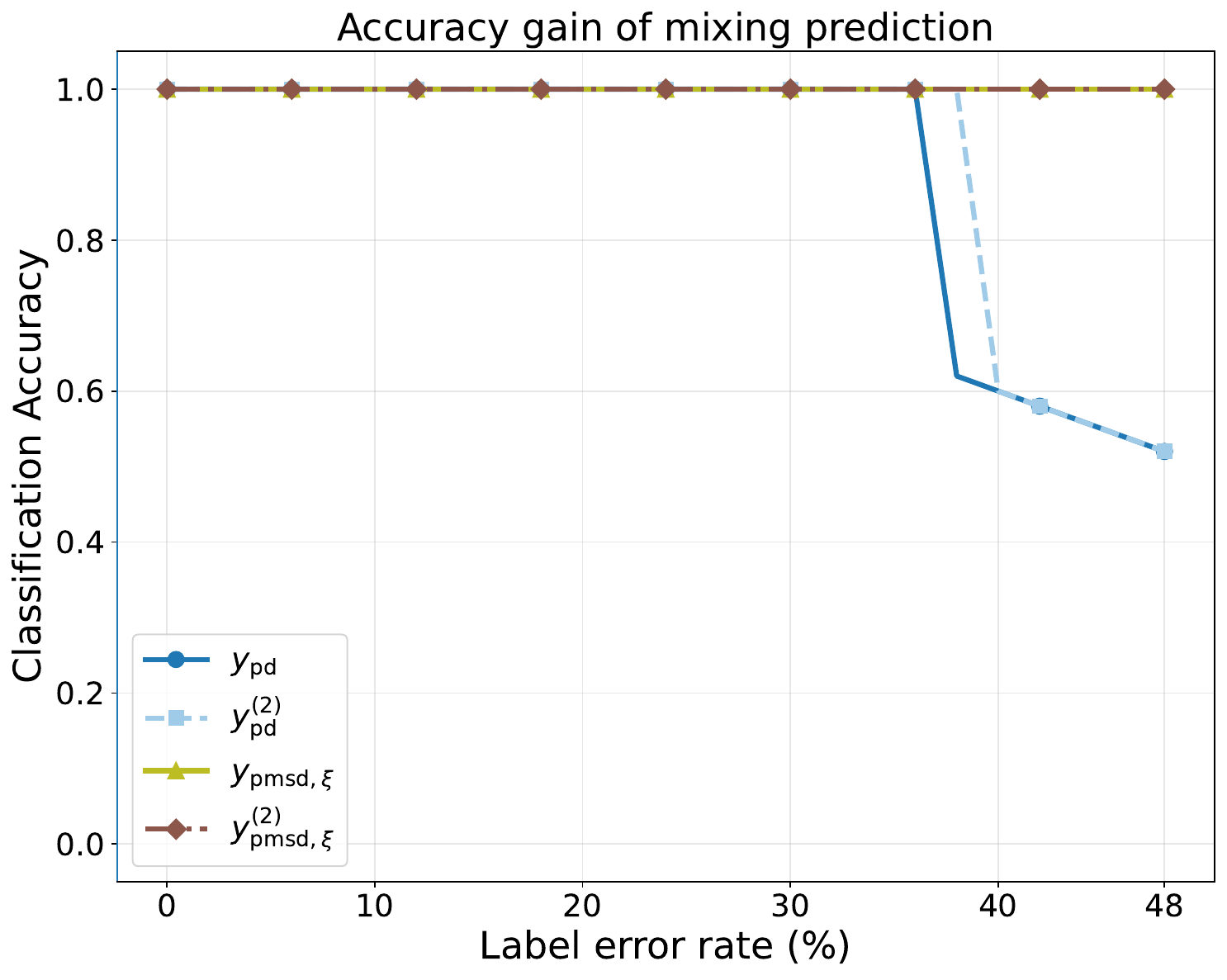}
  \end{subfigure}
  \hfill
  \quad
    \begin{subfigure}[t]{0.45\textwidth}
    \centering  \includegraphics[width=\textwidth]{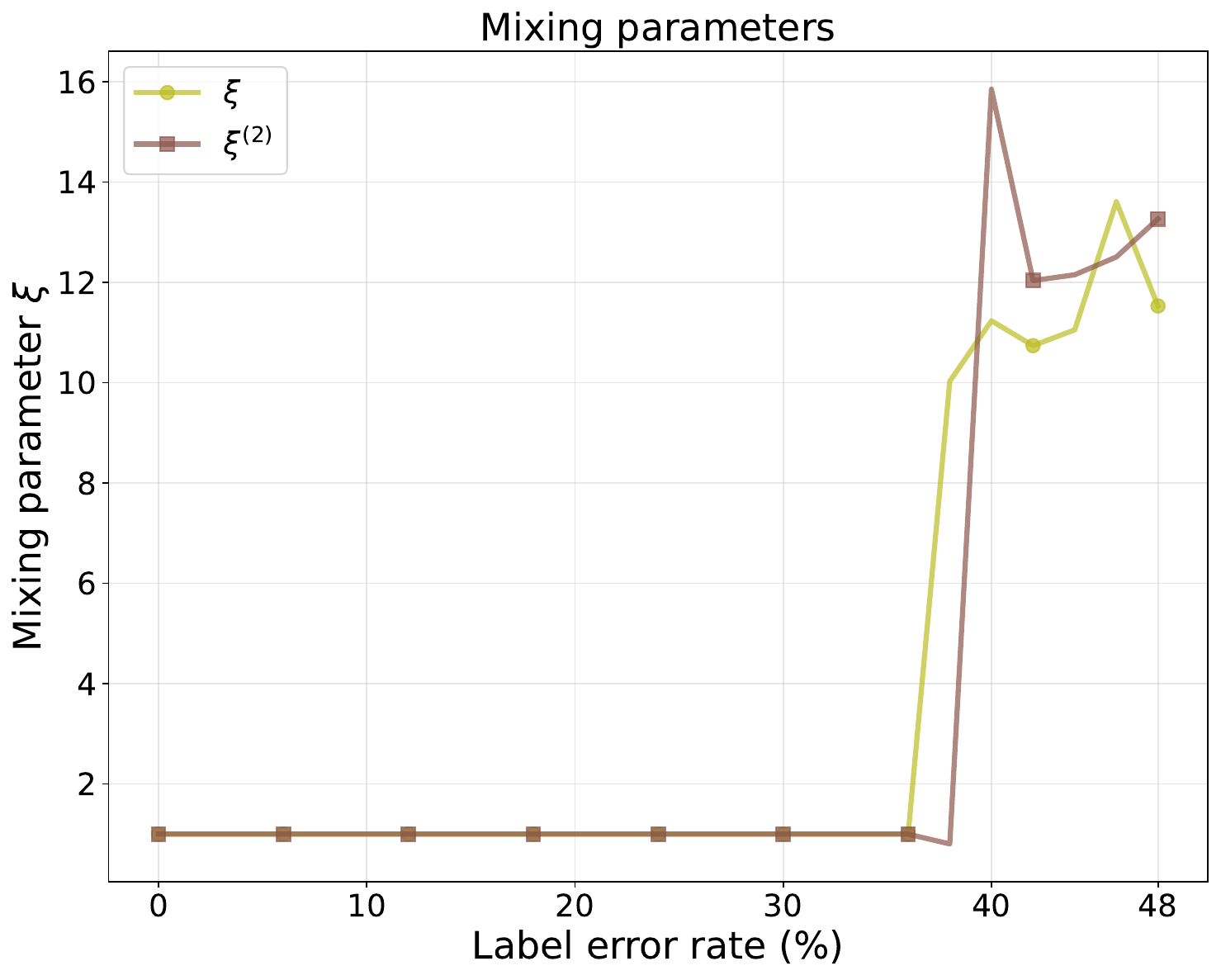}
  \end{subfigure}
  \caption{Synthetic logistic regression with constant feature correlation $c = 0.25$, $\lambda = 0.01$, and $n = 500$ samples per class; here, $\rho < 0.5$.}
  \label{fig:synthetic_constant_small}
\end{figure*}

\begin{figure*}[!ht]
  \centering
    \begin{subfigure}[t]{0.45\textwidth}
    \centering
    \includegraphics[width=\textwidth]{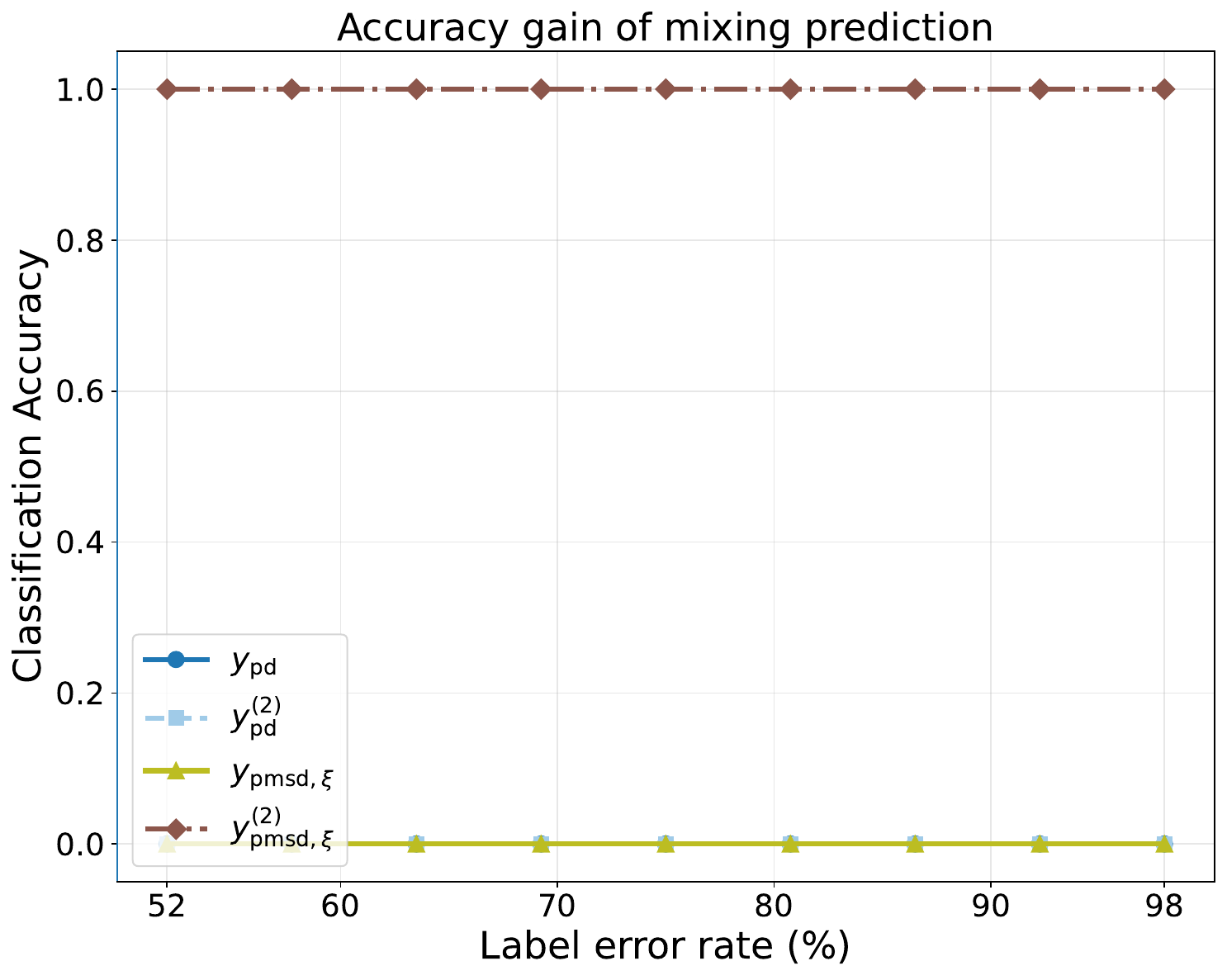}
  \end{subfigure}
  \hfill
  \quad
    \begin{subfigure}[t]{0.45\textwidth}
    \centering  \includegraphics[width=\textwidth]{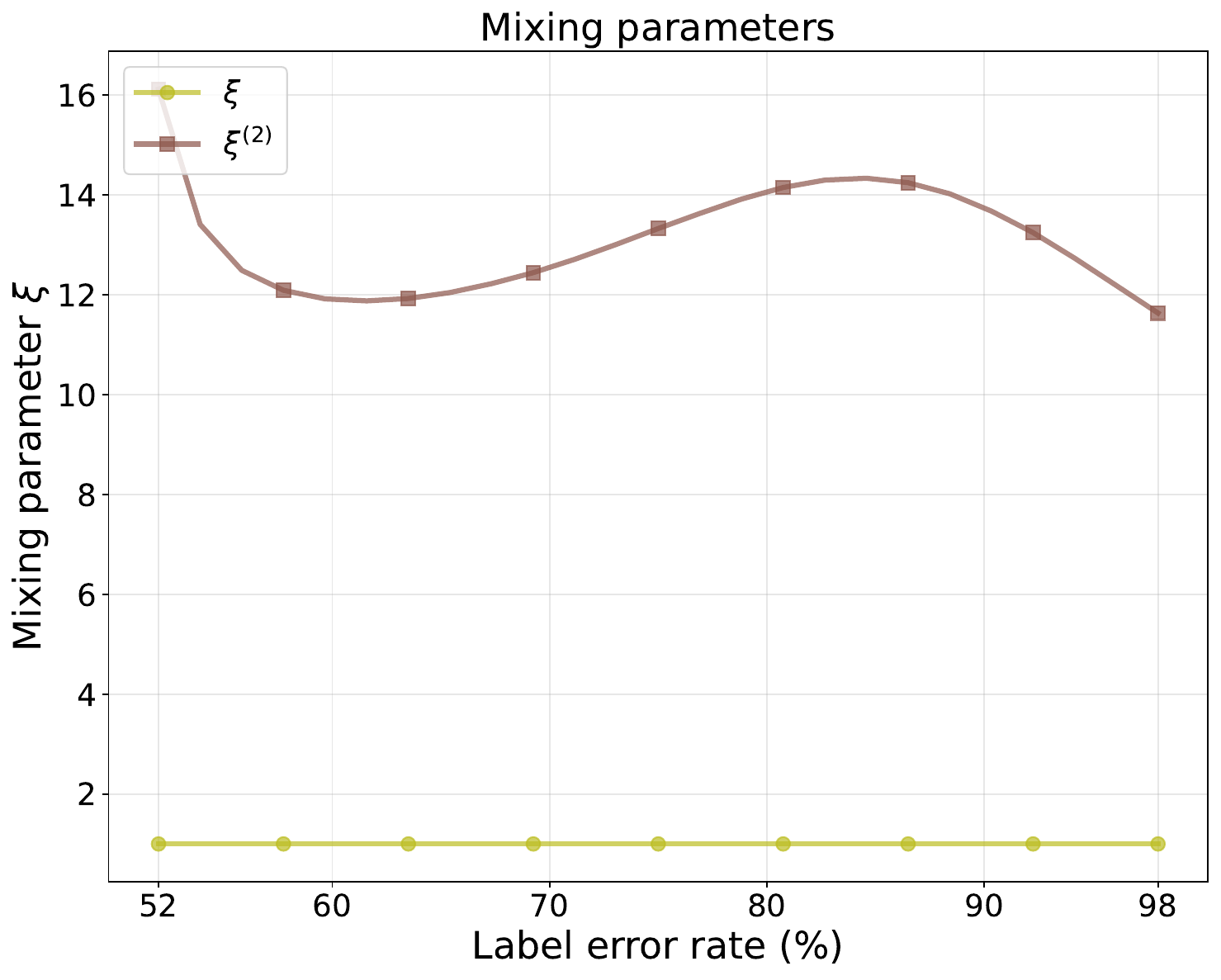}
  \end{subfigure}
  \caption{Synthetic logistic regression with constant feature correlation $c = 0.25$, $\lambda = 10$, and $n = 1500$ samples per class; here, $\rho > 0.5$.}
  \label{fig:synthetic_constant_large}
\end{figure*}

\clearpage
\subsubsection{Uniform feature correlation}

\bigskip

\begin{figure*}[!ht]
  \centering
    \begin{subfigure}[t]{0.45\textwidth}
    \centering
    \includegraphics[width=\textwidth]{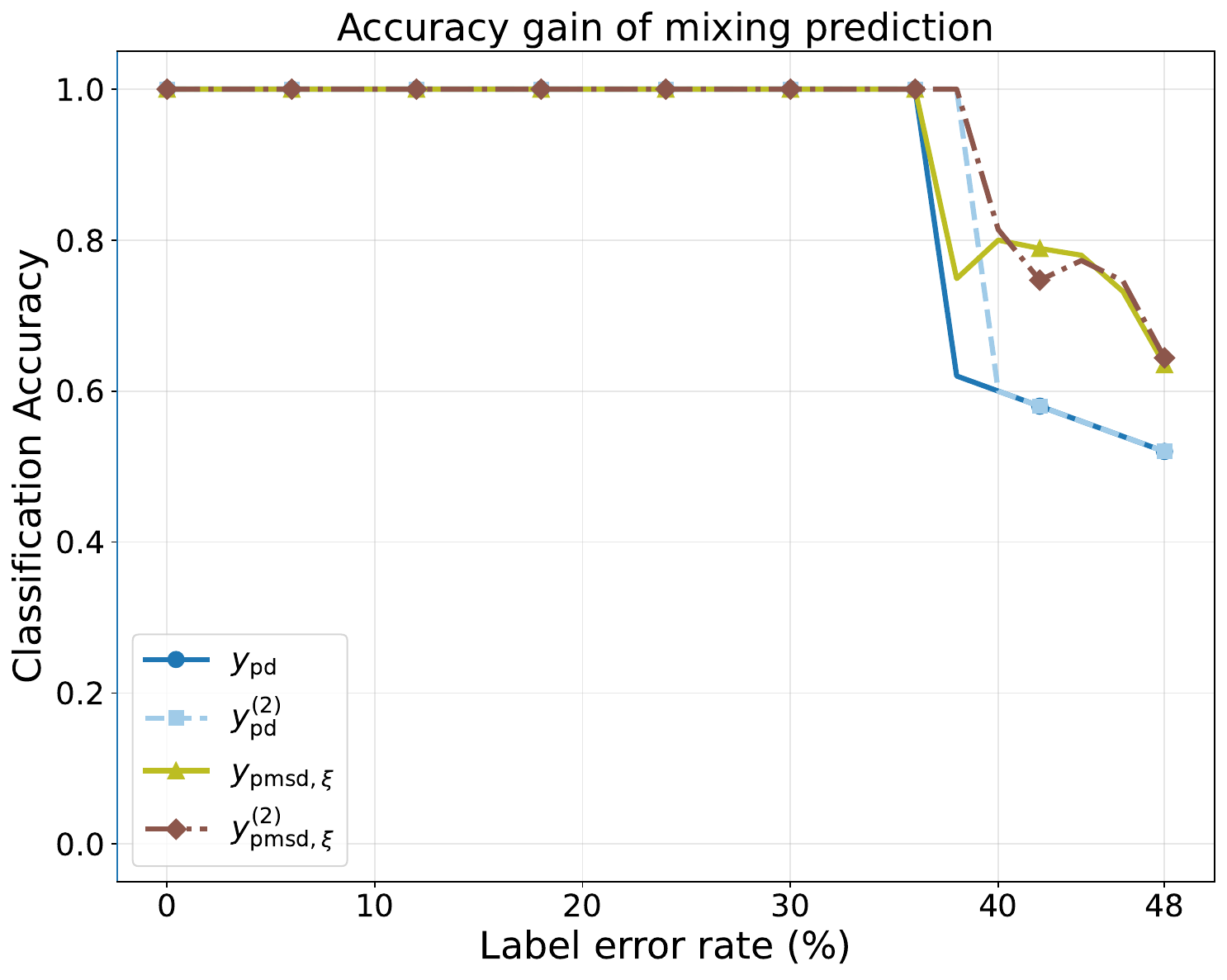}
  \end{subfigure}
  \hfill
  \quad
    \begin{subfigure}[t]{0.45\textwidth}
    \centering  \includegraphics[width=\textwidth]{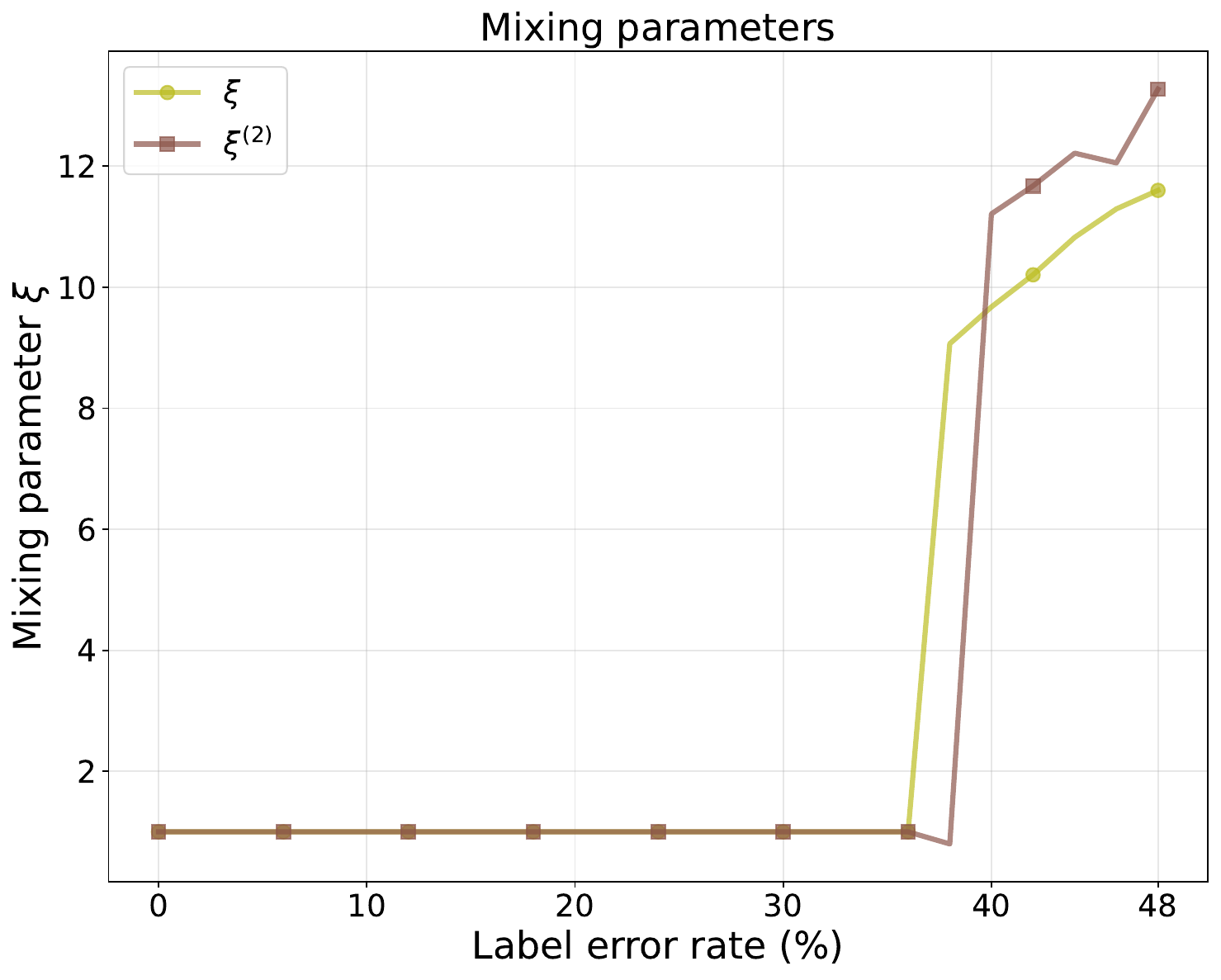}
  \end{subfigure}
  \caption{Synthetic logistic regression with uniform feature correlation (defined in \Cref{sec:synthetic_details}), $\lambda = 0.01$, and $n = 500$ samples per class; here, $\rho < 0.5$.}
  \label{fig:synthetic_uniform_small}
\end{figure*}

\begin{figure*}[!ht]
  \centering
    \begin{subfigure}[t]{0.45\textwidth}
    \centering
    \includegraphics[width=\textwidth]{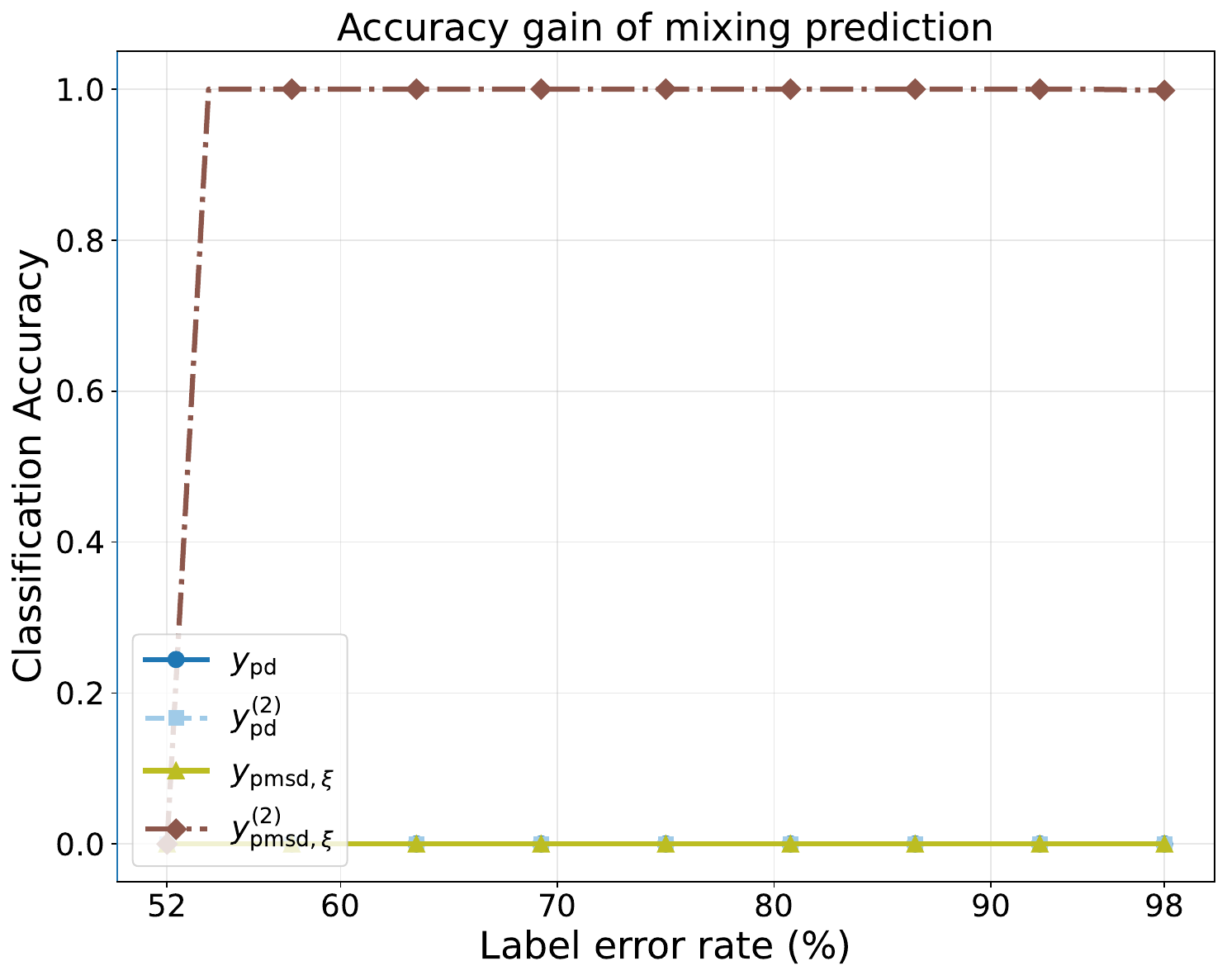}
  \end{subfigure}
  \hfill
  \quad
    \begin{subfigure}[t]{0.45\textwidth}
    \centering  \includegraphics[width=\textwidth]{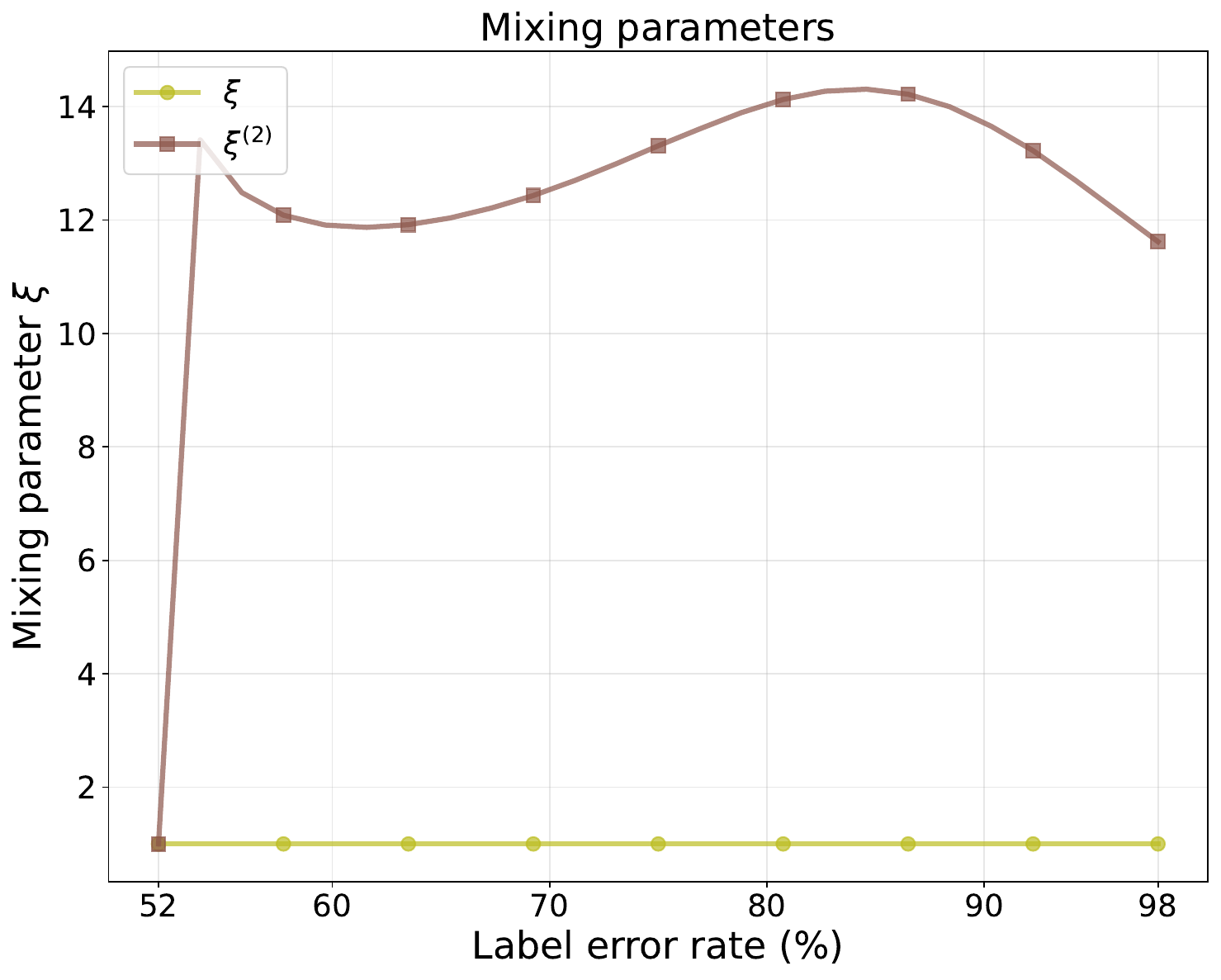}
  \end{subfigure}
  \caption{Synthetic logistic regression with uniform feature correlation (defined in \Cref{sec:synthetic_details}), $\lambda = 10$, and $n = 1500$ samples per class; here, $\rho > 0.5$.}
  \label{fig:synthetic_uniform_large}
\end{figure*}

\clearpage

\section{Experiment details}
\label{sec:experiment_details}

\subsection{Ridge regression experiments}

For the real-world ridge regression experiments in the left panel of \Cref{fig:teaser} and the middle and right panels of \Cref{fig:ridge_tuning}, we use UCI Communities and Crime \citep{communities_and_crime}, UCI Blog Feedback \citep{blogfeedback}, and UCI Airfoil \citep{airfoil}. We divide each dataset into a teacher training set, a fresh unlabeled set, a calibration set, and a test set, with the sizes specified below or in the corresponding figure captions. After handling missing values as described below, we center and standardize all variables using the means and standard deviations computed on the teacher training set.

\textbf{UCI Communities and Crime.}
The UCI Communities and Crime dataset \citep{communities_and_crime} combines socioeconomic data from multiple sources. It consists of 1,994 samples with 127 covariates. The response, ``ViolentCrimesPerPop,'' is the rate of violent crimes per 100,000 population. We remove 5 nonpredictive identifiers and 23 covariates from the LEMAS survey that contain 1,675 missing values across the 1,994 samples. The remaining $99$ covariates have no missing values and are used in our experiments.

\textbf{UCI Blog Feedback.} The UCI Blog Feedback dataset \citep{blogfeedback} was constructed from crawled and processed HTML documents of blog posts. We use its training split, which contains 52,397 samples, and reserve 30\% of these samples for testing. The sizes of the teacher training, fresh unlabeled, and calibration sets are given in the corresponding figure captions.

\textbf{UCI Airfoil.} The UCI Airfoil Self-Noise dataset \citep{airfoil} contains aerodynamic and acoustic measurements from NASA tests of two- and three-dimensional airfoil blade sections. The continuous response is the sound pressure level, measured in decibels. The dataset contains $n = 1503$ samples and $p = 5$ covariates: \texttt{frequency}, \texttt{attack-angle}, \texttt{free-stream-velocity}, \texttt{chord-length}, and \texttt{suction-side-displacement-thickness}. We reserve 30\% of the samples for testing. The sizes of the teacher training, fresh unlabeled, and calibration sets are given in the corresponding figure captions.

\textbf{CIFAR-10.} The CIFAR-10 dataset \citep{krizhevsky2009learning} consists of $32 \times 32$ color images from 10 classes, with 50,000 training images and 10,000 test images. We use subsets of the training and test splits, with sizes specified in \Cref{fig:ridge_tuning}.

For the CIFAR-10 ridge regression experiment in the left panel of \Cref{fig:ridge_tuning}, we extract 512-dimensional features using a ResNet-18 \citep{he2016deep} pretrained on ImageNet and available in PyTorch. We perform no additional data preprocessing and fit ridge regression to these frozen representations. The resulting predictor is a vector-valued function $f:\RR^{512}\to\RR^{10}$. Using the one-hot label vector $y\in\RR^{10}$ as the response, we define the aggregate squared risk by summing over the 10 output dimensions:
\begin{align}
R(f) = \EE_{(x,y)}\biggl[
\sum_{k=1}^{10} (y_k-f_k(x))^2
\biggr],
\end{align}
where $f_k:\RR^{512}\to\RR$ is the ridge predictor for the $k$th component of the one-hot response.

\clearpage

\subsection{Linear probing experiments}

For the linear probing experiments on Caltech-101, Caltech-256, and CIFAR-100, we extract features using a ResNet-34 \citep{he2016deep} pretrained on ImageNet and available in PyTorch. We then train a linear classifier with a softmax output on top of these frozen features. Results appear in the right panel of \Cref{fig:teaser}, \Cref{tab:cifar100_0.2_hierarchical}, and \Cref{sec:linear_probing_tables}.

\textbf{Training details.} The teacher is trained by minimizing cross-entropy against the corrupted one-hot labels. The PD student is trained by minimizing cross-entropy against the teacher's softmax predictions. Each linear classifier is trained for 200 epochs using stochastic gradient descent with momentum 0.9 and a learning-rate decay factor of 0.98. The initial learning rates are $0.001$ for the teacher and $0.05$ for the student. We perform no data preprocessing, and all experiments are run on an NVIDIA T4 GPU.

\textbf{Caltech-101.} Caltech-101 \citep{griffin2007caltech} is an object recognition dataset with 101 object categories plus a background category. It contains approximately 9,000 images, with 40 to 800 images per category, and is therefore highly imbalanced. We draw disjoint subsets containing $n_t = 1500$ teacher training samples, $n_s = 4500$ fresh unlabeled samples, $n_{\mathrm{cal}} = 500$ calibration samples, and 2,000 test samples.

\textbf{Caltech-256.} Caltech-256 \citep{griffin2007caltech} contains 30,607 images spanning 256 object categories. We draw disjoint subsets containing $n_t = 5000$ teacher training samples, $n_s = 15000$ fresh unlabeled samples, $n_{\mathrm{cal}} = 1000$ calibration samples, and 6,000 test samples.

\textbf{CIFAR-100.} CIFAR-100 \citep{krizhevsky2009learning} contains 60,000 images from 100 classes. These classes are grouped into 20 superclasses, each containing five related fine-grained classes; for example, beaver, dolphin, otter, seal, and whale belong to the same superclass. We draw disjoint subsets containing $n_t = 5000$ teacher training samples, $n_s = 15000$ fresh unlabeled samples, $n_{\mathrm{cal}} = 1000$ calibration samples, and 9,000 test samples.

\textbf{Random corruption.} Under random label corruption with rate $\rho$, we select a fraction $\rho$ of the teacher's training samples uniformly at random and replace each label with one drawn uniformly from the remaining classes.

\textbf{Hierarchical corruption.} This scheme models errors between semantically similar classes. Because only CIFAR-100 provides the required superclass structure, we use hierarchical corruption only for this dataset. At corruption rate $\rho$, we select a fraction $\rho$ of the teacher's training samples uniformly at random and replace each label with one drawn uniformly from the other classes in the same superclass.

\clearpage
\subsection{Synthetic experiments}
\label{sec:synthetic_details}

\subsubsection{Ridge regression}

For the synthetic ridge experiments in \Cref{fig:synthetic_ridge} and the additional experiments in \Cref{sec:additional_exps_synthetic_ridge}, we use the following covariance and signal models.

\begin{itemize}[leftmargin=7mm]
    \item Isotropic covariance: $\Sigma = I_p$.

    \item One-spike covariance: $\Sigma = I_p + 5vv^{\top}$, where $v=g/\|g\|_2$ and $g\sim\mathcal{N}(0,I_p)$.
    
    \item AR1 covariance: $\Sigma_{ij}=\rho^{|i-j|}$ for all $i,j$, with correlation parameter $\rho=0.25$.

    \item Isotropic signal: $\beta\sim\mathcal{N}(\bm{0},(r^2/p)I_p)$.

    \item Top-aligned signal: let $m\in(0,1)$ be the alignment proportion, $a\in(0,1)$ the alignment factor, and $k=\lfloor mp\rfloor$, where $1\leq k\leq p-1$. If $V=[v_1,\ldots,v_p]$ contains the eigenvectors of $\Sigma$ in decreasing order of their eigenvalues, define
    \[
    \Sigma_\beta
    :=pV\operatorname{diag}\!\Bigl(
    \underbrace{\frac{a}{k},\ldots,\frac{a}{k}}_{k},
    \underbrace{\frac{1-a}{p-k},\ldots,\frac{1-a}{p-k}}_{p-k}
    \Bigr)V^\top,
    \qquad
    \beta\sim\mathcal{N}\!\Bigl(\bm{0},\frac{r^2}{p}\Sigma_\beta\Bigr).
    \]
    Thus, a proportion $a$ of the expected signal energy lies in the top $k$ eigendirections. In \Cref{fig:synthetic_ridge}, we set $m=0.1$ and $a=0.9$.
\end{itemize}

The empirical curves in \Cref{fig:synthetic_ridge} are averaged over 30 independent simulations.

\subsubsection{Logistic regression}
For the synthetic logistic regression experiments, we follow \citet{das2023understanding} and construct a block-diagonal kernel matrix
\[
K:=\operatorname{diag}(K_1,K_0)\in\mathbb{R}^{2n\times 2n},
\]
where the two blocks correspond to the ground-truth classes and have zero cross-class correlation. Each block has unit diagonal. In the constant feature-correlation setting, every off-diagonal entry equals $c=0.25$. In the uniform feature-correlation setting, the off-diagonal entries are taken from $(1/n)ZZ^\top$, where the entries of $Z\in\mathbb{R}^{n\times n}$ are drawn independently from the uniform distribution on $[0.3,0.7]$. The second block is constructed analogously. Both constructions yield a positive semidefinite kernel matrix.

We train kernel logistic regression in the dual parameterization. For pseudo-labels $q=(q_1,\ldots,q_{2n})\in[0,1]^{2n}$ and dual variables $\alpha\in\mathbb{R}^{2n}$, let $f=K\alpha$ and apply the sigmoid componentwise to obtain $\hat p=\sigma(f)$. We minimize the regularized binary cross-entropy objective
\begin{equation}
    \mathcal{L}_q(\alpha)
    \;=\;
    -\frac{1}{2n}\sum_{i=1}^{2n}
    \Bigl[
        q_i\log \hat{p}_i + (1-q_i)\log(1-\hat{p}_i)
    \Bigr]
    +\frac{\lambda}{2}\,\alpha^\top K\alpha.
\end{equation}
For the first PD student, $q_i=\hat y_i$ is the teacher's hard pseudo-label; for the second PD student, $q_i=y_{\pd}(x_i)$ is the first student's soft prediction. This is the dual form of the training objectives in \Cref{eq:class_training_loss,eq:class_training_loss_2}. We minimize it using L-BFGS-B \citep{byrd1995limited} with analytic gradients and initialize at $\alpha=\mathbf{0}$.

\end{document}